\documentclass[12pt,fleqn]{article}

\usepackage{paralist}
\usepackage{enumitem}
\usepackage{a4}
\usepackage{amsthm}

\usepackage{ifthen}
\usepackage{amsmath}
\usepackage{amssymb}
\usepackage{mathtools} 
\usepackage{graphicx} 
\usepackage{rotating}
\usepackage{color}
\usepackage{lscape}

\usepackage{rotating}

\usepackage{textcomp} 

\usepackage{tikz}
\usetikzlibrary{calc,decorations.markings,matrix,arrows,positioning}
\tikzset{>=latex}

\makeatletter
\newcommand{\ostar}{\mathbin{\mathpalette\make@circled\star}}
\newcommand{\make@circled}[2]{%
  \ooalign{$\m@th#1\smallbigcirc{#1}$\cr\hidewidth$\m@th#1#2$\hidewidth\cr}%
}
\newcommand{\smallbigcirc}[1]{%
  \vcenter{\hbox{\scalebox{0.77778}{$\m@th#1\bigcirc$}}}%
}
\makeatother

\swapnumbers
\theoremstyle{plain}
\newtheorem{theorem}{Theorem}[section]
\newtheorem{lemma}[theorem]{Lemma}
\newtheorem{proposition}[theorem]{Proposition}
\newtheorem{corollary}[theorem]{Corollary}

\theoremstyle{definition}
\newtheorem{definition}[theorem]{Definition}

\newtheorem{example}[theorem]{Example}

\newtheorem{remark}[theorem]{Remark}
\newtheorem{remarks}[theorem]{Remarks}

\newtheorem{notation}[theorem]{Notation}

\newtheorem{Galoisconnections}[theorem]{Galois connections}

\newtheorem*{remarksnote}{Remarks}

\newtheorem*{remarknote}{Remark}

\numberwithin{equation}{theorem} 

\newcommand{\New}[1]{\emph{#1}}
\let\Def=\New

\newcommand{\duclose}[2][d]{\prescript{#1\mspace{-1.8mu}}{}{\overline{#2}}}
\newcommand{\du}[1][d]{\ensuremath{\text{\sf R}_{#1}}}

\newcommand{\Rclose}[2][d]{\prescript{#1\mspace{-1.8mu}}{}{\overline{#2}}}
\newcommand{\dR}[1][d]{\ensuremath{\text{\sf R}_{#1}}}

\DeclareMathOperator{\End}{End}

\DeclareMathOperator{\Pol}{Pol}
\newcommand{\Pola}[1][]{\Pol^{(#1)}}
\DeclareMathOperator{\Inv}{Inv}

\DeclareMathOperator{\Op}{Op}
\newcommand{\Opa}[1][]{\Op^{(#1)}}
\DeclareMathOperator{\Rel}{Rel}
\newcommand{\Rela}[1][]{\Rel^{(#1)}}

\DeclareMathOperator{\RRel}{{\sf R\mspace{-1.8mu}Rel}}

\newcommand{\RRela}[1][]{\RRel^{(#1)}}

\newcommand{\dRRel}[1][d]{\prescript{#1\mspace{-1.7mu}}{}\RRel}

\DeclareMathOperator{\GQuord}{GQuord\mspace{1.7mu}}
\newcommand{\dGQuord}[1][d]{\prescript{#1\mspace{-1.7mu}}{}\GQuord} 
\newcommand{\GQuorda}[1][]{\GQuord^{(#1)}} 

\DeclareMathOperator{\wGQuord}{wGQuord\mspace{1.7mu}}
\newcommand{\wGQuorda}[1][]{\wGQuord^{(#1)}} 

\DeclareMathOperator{\LOC}{LOC}
\newcommand{\dLOC}[1][d]{#1\mbox{-}\!\LOC}

\newcommand{\X}{\ensuremath{\Xi}}

\DeclareMathOperator{\gQuord}{gQuord} 

\DeclareMathOperator{\Ima}{Im} 

\newcommand{\preserves}{\mbox{ $\triangleright$ }}

\newcommand{\id}{\mathsf{id}}

\newcommand{\cK}{\mathcal{K}}
\newcommand{\cL}{\mathcal{L}}
\newcommand{\cC}{\mathcal{C}}

\newcommand{\bDelta}{\mathbf{\Delta}}

\newcommand{\bc}{\boldsymbol{c}} 

\newcommand{\bOne}{\mathbf{1}}

\newcommand{\wXi}{\Xi^{\text{\rm w}}}

\newcommand{\N}{\mathbb{N}}
\newcommand{\Z}{\mathbb{Z}}

\newcommand{\Sg}[2][]{\ensuremath{\langle #2 \rangle_{#1}}}

\newcommand{\trl}{\mathsf{trl}} 
\newcommand{\btrl}{\mathsf{btrl}} 
\newcommand{\ftrl}{\mathsf{ftrl}} 

\newcommand{\ebtrl}{\mathsf{ebtrl}} 
\newcommand{\eftrl}{\mathsf{eftrl}} 

\DeclareMathOperator{\Pow}{\mathfrak{P}}

\let\rho=\varrho
\let\epsilon=\varepsilon
\let\phi=\varphi
\let\kappa=\varkappa

\let \restrictionORIGINAL=\restriction
\renewcommand{\restriction}{\mathclose\restrictionORIGINAL}

\tikzset{myStyle/.style={baseline=(center.base), font=\small,
    every node/.style={inner sep=0.25em} }}

\NewDocumentCommand{\LinePic}{ O{} O{} O{} O{}  O{1} }{ 
  \begin{tikzpicture}[myStyle, scale=#5*1 ]
    \node (center) at (0,0) {\phantom{$\cdot$}}; 
    \path (0,0)  node (s1) {$#1$}
        ++(.5,0)  node (s2) {$#2$}
        ++(.5,0)  node (s3) {$#3$}
        ++(1,0)  node (s4) {$#4$};
    \draw (s1) -- (s2) -- (s3) -- (s4);
  \end{tikzpicture}
}  

\newcommand{\SquareUnwrapped}[4]{ 
  \node (center) at (0.5,-0.5) {\phantom{$\cdot$}}; 
  \path (0,0)  node (nw) {$#2$}
      ++(1,0)  node (ne) {$#4$}
      ++(0,-1) node (se) {$#3$}
      ++(-1,0) node (sw) {$#1$};
  \draw (nw) -- (ne) -- (se) -- (sw) -- (nw);
}  
\NewDocumentCommand{\SquareXY}{ O{} O{} O{} O{} O{1} O{1} }{ 
  \begin{tikzpicture}[myStyle, xscale=#5*1, yscale=#6*1 ]
    \SquareUnwrapped{#1}{#2}{#3}{#4}
  \end{tikzpicture}
}  
\NewDocumentCommand{\Square}{ O{} O{} O{} O{} O{1} }{ 
  \SquareXY[#1][#2][#3][#4][#5][#5]
}  

\NewDocumentCommand{\SquareAxes}{ O{} O{} O{1} O{0} O{1} }{  
  \begin{tikzpicture}[myStyle, scale=#3*0.8] 
    \node (center) at (0.5,0.5) {\phantom{$\cdot$}}; 
    \draw (0,0) -- node[above]{$#1$} (1,0) node[right]{$#4$}
      (0,0) -- node[left]{$#2$} (0,1) node[above]{$#5$};
  \end{tikzpicture}
}  
\NewDocumentCommand{\CubeAxes}{ O{} O{} O{} O{} O{} O{} O{} }{  
  \begin{tikzpicture}[myStyle, scale=#4*0.85] 
    \node (center) at (0.5,0.75) {\phantom{$\cdot$}}; 
    \draw (0,0) -- node[above]{$#1$} (1,0) node[right]{$#5$}
      (0,0) -- node[left]{$#2$} (0,1) node[above]{$#6$}
      (0,0) -- node[below left=-0.25em]{$#3$} (0.5,-0.5) node[below right=-0.2em]{$#7$};
  \end{tikzpicture}
}  

\newcommand{\CubeNodes}[8]{  
  \node at (0.75,-0.75) (center) {\phantom{$\cdot$}}; 
  \path (0,0)  node (back_nw)      {$#2$}
      ++(1,0)  node (back_ne)      {$#4$}
      ++(0,-1) node (back_se)      {$#3$}
      ++(-1,0) node (back_sw)      {$#1$}
        (0.5,0.5) node (front_nw) {$#6$}
      ++(1,0)  node (front_ne)     {$#8$}
      ++(0,-1) node (front_se)     {$#7$}
      ++(-1,0) node (front_sw)     {$#5$};
}  
\newcommand{\CubeUnwrapped}[8]{ 
  \CubeNodes{#1}{#2}{#3}{#4}{#5}{#6}{#7}{#8}
  \draw (back_nw) -- (back_ne) -- (back_se) -- (back_sw) -- (back_nw)
    (front_nw) -- (front_ne) -- (front_se) -- (front_sw) -- (front_nw)
    (back_nw) -- (front_nw)
    (back_ne) -- (front_ne)
    (back_se) -- (front_se)
    (back_sw) -- (front_sw);
}  
\newcommand{\CubeDUnwrapped}[8]{ 
  \CubeNodes{#1}{#2}{#3}{#4}{#5}{#6}{#7}{#8}
  \draw (front_nw) -- (front_ne) -- (front_se) -- (front_sw) -- (front_nw)
    (back_nw) -- (back_ne) (back_sw) -- (back_nw)
    (back_nw) -- (front_nw)
    (back_ne) -- (front_ne)
    (back_sw) -- (front_sw);
  \draw[densely dotted] (back_ne) -- (back_se) -- (back_sw)
    (back_se) -- (front_se);
}  
\NewDocumentCommand{\Cube}{ O{} O{} O{} O{} O{} O{} O{} O{} O{1} }{  
  \begin{tikzpicture}[myStyle, scale=#9*1 ]
    \CubeUnwrapped{#1}{#2}{#3}{#4}{#5}{#6}{#7}{#8}
  \end{tikzpicture}
}  

\NewDocumentCommand{\CubeDeep}{ O{} O{} O{} O{} O{} O{} O{} O{} O{1}  }{  
  \begin{tikzpicture}[myStyle, xscale=#9*1, yscale=1.5 ]
    \CubeUnwrapped{#1}{#2}{#3}{#4}{#5}{#6}{#7}{#8}
  \end{tikzpicture}
}  
\NewDocumentCommand{\CubeD}{ O{} O{} O{} O{} O{} O{} O{} O{} O{1} }{  
  \begin{tikzpicture}[myStyle, scale=#9*1 ]
    \CubeDUnwrapped{#1}{#2}{#3}{#4}{#5}{#6}{#7}{#8}
  \end{tikzpicture}
}  

\NewDocumentCommand{\DeltaZeroCubeD}{ O{} O{} O{} O{} O{} O{} O{} O{} O{1} }{  
  \begin{tikzpicture}[myStyle, scale=#9*1]
    \CubeDUnwrapped{#1}{#2}{#3}{#4}{#5}{#6}{#7}{#8}
    \draw (back_sw)  to[out=30,in=180-30] (back_se)
      (back_nw)  to[out=30,in=180-30] node[auto]{$\delta$} (back_ne)
      (front_sw) to[out=30,in=180-30] (front_se);
    \draw[dashed] (front_nw) to[out=30,in=180-30] (front_ne);
  \end{tikzpicture}
}  

\NewDocumentCommand{\DeltaOneCubeD}{ O{} O{} O{} O{} O{} O{} O{} O{} O{1} }{  
  \begin{tikzpicture}[myStyle, scale=#9*1]
    \CubeDUnwrapped{#1}{#2}{#3}{#4}{#5}{#6}{#7}{#8}
    \draw (back_sw)  to[out=120,in=240]node[left]{$\delta$} (back_nw)
      (back_se)  to[out=120,in=240]  (back_ne)
      (front_sw) to[out=120,in=240] (front_nw);
  \end{tikzpicture}
}  
\NewDocumentCommand{\DeltaTwoCubeD}{ O{} O{} O{} O{} O{} O{} O{} O{} O{1} }{  
  \begin{tikzpicture}[myStyle, scale=#9*1]
    \CubeDUnwrapped{#1}{#2}{#3}{#4}{#5}{#6}{#7}{#8}
    \draw (back_sw) to[out=180+30,in=180] node[auto,swap]{$\delta$} (front_sw)
      (back_se) to[out=0,in=30] (front_se)
      (back_nw) to[out=180+30,in=180] (front_nw);
    \draw[dashed] (back_ne) to[out=0,in=30] (front_ne);
  \end{tikzpicture}
}  

\newcommand{\ThreeSquareUnwrapped}[9]{ 
  \node (center) at (1,-1) {\phantom{$\cdot$}}; 
  \path (0,0)  node (aa) {$#1$}
      ++(1,0)  node (ba) {$#2$}
      ++(1,0)  node (ca) {$#3$}
      ++(0,-1) node (ab) {$#4$}
      ++(-1,0) node (bb) {$#5$}
      ++(-1,0) node (cb) {$#6$}
      ++(0,-1) node (ac) {$#7$}
      ++(1,0)  node (bc) {$#8$}
      ++(1,0)  node (cc) {$#9$};
  \draw (aa) -- (ba) -- (ca);
  \draw (ab) -- (bb) -- (cb);
  \draw (ac) -- (bc) -- (cc);
  \draw (aa) -- (cb) -- (ac);
  \draw (ba) -- (bb) -- (bc);
  \draw (ca) -- (ab) -- (cc);
}  
\NewDocumentCommand{\ThreeSquareXY}{ O{} O{} O{} O{} O{} O{} O{} O{} O{} }{ 
  \begin{tikzpicture}[myStyle, xscale=1, yscale=1 ]
    \ThreeSquareUnwrapped{#1}{#2}{#3}{#4}{#5}{#6}{#7}{#8}{#9}
  \end{tikzpicture}
}  

\author{Andrew Moorhead\footnote{Funding from the ERC (Grant Agreement no. 101071674, POCOCOP). Views and opinions expressed are however
those of the authors only and do not necessarily reflect those of the European Union or the European Research
Council Executive Agency.}
\and Reinhard P\"oschel%
 \and Institute of Algebra\\TU Dresden (Germany)
}

\date{September 7, 2026}

\title{Higher-dimensional generalized quasiorders}

\begin{document}

    \maketitle

\textbf{Abstract.} \small It is well known that the polymorphism clone of a
      quasiorder is determined by its unary part, in the sense that
      any operation is a polymorphism if it satisfies the condition
      that all unary functions obtained by fixing some of its
      variables at constant values are polymorphisms. We show that a
      clone satisfies an analogous property for a particular
      fixed arity $d$ if and only if it is determined by a collection of what we call higher-dimensional
      generalized quasiorders whose dimension is equal to $d$. We define higher-dimensional generalized quasiorders with the appropriate generalizations of reflexivity and transitivity to the higher-dimensional setting, which rely on an explicit analysis of `rectangular' arity.

\normalsize
\section*{Introduction}

The `preservation' relation $f \preserves \rho$ between a finitary
operation and a finitary relation prescribes a fundamental Galois
connection between the lattice of finitary operation
clones and the lattice of finitary 
relational clones over a finite set, which since its discovery has guided
research for decades. Broadly speaking, much of this research is
motivated by the following template pair of questions:
\begin{itemize}
\item[] \emph{Given a set of relations which satisfy some property,
    what can be said about its corresponding (operation) clone of
    polymorphisms?}

\item[] \emph{Given a set of operations which satisfy some property,
    what can be said about its corresponding (relational) clone of
    invariant relations?}
\end{itemize}

An early result in the spirit of the first kind of question was
discovered by Mal'cev~\cite[Teorema~1, p.4]{Mal1954} (English
translation \cite{Mal1963}): an operation $f$ preserves an
equivalence 
relation $\rho$ if and only if all unary functions obtained from $f$
by fixing all but at most one of its variables at constants preserve
$\rho$ (we have modified the statement slightly in anticipation of our
general theory and allow constant unary 
functions obtained by fixing \emph{all} variables of $f$ at a constant
value). The functions obtained
in this manner are called \emph{(unary) translations} of $f$; we
denote the set of all such by $\trl_1 (f)$. Actually, symmetry is not
necessary in Mal'cev's argument, so in fact every quasiorder satisfies
the following equivalence:
\[
  f \preserves \rho \iff \trl_1 (f) \preserves \rho.
\]
The truth of the above equivalence (which we denote $\Xi_1(\rho)$) for
quasiorders $\rho$ leads naturally to the following two questions:

\begin{itemize}
\item[] \emph{What other relations, aside from quasiorders, satisfy
    $\Xi_1$?}
\item[] \emph{What can be said about the analogous higher arity
    conditions?}
\end{itemize}

The first question is answered by the authors of~\cite{JakPR2024}. 
In their work, they
discover a novel generalization of transitivity to arbitrary arity
relations and using this generalization define the class of
\emph{generalized quasiorders} to be those relations which are both
reflexive (contain all constant tuples) and transitive. They prove
that, 
indeed, $1$-dimensional generalized quasiorders satisfy
$\Xi_1$. Moreover, they show that any other $\rho$ which satisfies
$\Xi_1$ has a primitive positive definition in a generalized
quasiorder (which can always taken to be the unary part of
$\Pol\rho$, considered as a relation). Within the
theory we present in this article, the class of relations they study become the
\emph{$1$-dimensional} generalized quasiorders.

The second question turns out to have connections to commutator
theory, which we introduce by giving a simple example. Consider the
minority operation $m$ on the two element set $\{0,1\}$ defined by
$m(x,y,z) = x+ y + z $, where here addition is the standard addition
in $\mathbb{Z}_2$, and let
$ \rho:= \{ (x,y,z, m(x,y,z)) \mid x,y,z \in \{0,1\} \} $ be the graph of
$m$. It is well-known that $\Pol\rho$ is a maximal among the proper
clones over $\{0,1\}$ and that it consists of all
$\mathbb{Z}_2$-polynomial functions, i.e.\ sums of the form
$ \sum x_i + c. $ Since every unary operation on $\{0,1\}$ belongs to
$\Pol\rho$, it follows that $\Xi_1(\rho)$ fails. However, it turns
out that $\rho$ \emph{does} satisfy $\Xi_2$, which is the following
implication:
\[
  f \preserves \rho \iff \trl_2(f) \preserves \rho,
\]
where $\trl_2(f)$ is the set of all binary functions obtained from $f$
by fixing all but at most two variables at constants (the arity of
unary and constant functions obtained in this way is increased with
the appropriate addition of fictitious variables).

The connection to the commutator is that $\rho$ is actually the
relation used to check that $\mathbb{Z}_2$ is an \emph{abelian}
algebra~\cite{FreeM1987}. In this context, $\rho$ is given the name $M(\bOne,\bOne)$
and referred to as the \emph{$(\bOne,\bOne)$-matrices}, where here `$\bOne$' is
understood to be the full congruence on $\mathbb{Z}_2$. Although it is
usually glossed over in the literature, $\rho$ and $M(\bOne,\bOne)$ are
actually different objects, since the former is a subset of
$\{0,1\}^4$ (and hence has a coordinate `geometry' that is a line),
while the latter is a subset of $2 \times 2$ (and hence has a
coordinate `geometry' that looks like a square). That is, elements of
$\rho$ are tuples, which we can depict as vertex labeled `lines':
\[
  \rho = \left\{ \LinePic[x][y][z][m(x,y,z)][1.5]\middle|\; x,y,z \in \{0,1\}
    \subseteq \{0,1\}^4 \right\},
\]
while elements of $M(\bOne,\bOne)$ are usually defined to be
$2\times2$-matrices, which we can depict as vertex labeled `squares':
\[
  M(\bOne,\bOne) = \left\{ \SquareXY[y][z][x][m(x,y,z)][1.5][1] \middle|\; x,y,z \in
    \{0,1\} \right\} \subseteq \{0,1\}^{2\times 2}
\]

It is not hard to check that $M(\bOne,\bOne)$ is what is called a
\emph{$2$-dimensional equivalence relation}, which means
that it is reflexive, symmetric, and transitive when expressed either
as a binary relation on its set of columns or as a binary relation on
its set of rows.

It turns out that any $2$-dimensional equivalence relation satisfies
the property $\Xi_2$, and it follows from this observation that the
binary commutator for an algebra is determined by its binary
polynomials. This generalizes to higher dimensions: a collection of
vertex-labeled $d$-dimensional hypercubes is called a
\emph{$d$-dimensional equivalence relation} if each of the $d$-many
ways of interpreting it as a collection of pairs of
$(d-1)$-dimensional labeled cubes is an equivalence
relation. Similarly to the unary and binary case, a $d$-dimensional
equivalence relation satisfies the equivalence $\Xi_d$ (which is
defined by substituting $\trl_1$ in $\Xi_1$ with the natural
definition $\trl_d(f)$), so it follows that the \emph{$d$-ary higher
  commutator} for an algebra is determined by its $d$-ary
polynomials \cite{Moo2024}. The present work will not discuss
  commutator theory in any depth, but to avoid making a false claim we
  mention an important point: there are many potential
  commutators and the above 
  statement is only true for the $d$-ary \emph{hyper}commutator, which
  is defined using $d$-dimensional equivalence relations  (for an
  in-depth analysis of the relationship between the higher
  hypercommutator and the higher \emph{term condition} commutator, we
  refer the reader to~\cite{Moo2021}). Within the theory we develop in
this paper, $d$-dimensional equivalence relations are examples of
$d$-dimensional generalized quasiorders of \emph{compound arity}
$\underbrace{2^1\times \dots \times 2^1}_d$.

After some discussion, we resolved to find a general theory which
incorporated both lines of research, i.e.\ we wanted to answer the
following question:
\begin{itemize}
\item[] \emph{Which relations satisfy $\Xi_d$?}
\end{itemize}
Our answer to this question is complete in the same sense as the
answer given in \cite{JakPR2024} for the $1$-dimensional case is
complete. We define the class of $d$-dimensional quasiorders and show
that each such satisfies the $\Xi_d$ property
(Theorem~\ref{thm:PolofGQuordMain2}). Any other relation $\rho$ which
satisfies $\X_d$ has a positive primitive definition in a special kind
of $d$-dimensional generalized quasiorder, which can always be taken
to be the collection of $d$-ary polymorphisms of $\rho$ considered as
a relation (Proposition~\ref{A3Xi}).

Our investigation of the $\Xi_d$ property relies fundamentally on
generalizations of reflexivity and transitivity to the higher
dimensional setting. We have chosen what seems to be the weakest
possible generalization of each property which still captures
$\Xi_d$. However, when taken together, our higher dimensional
generalizations of reflexivity and transitivity have stronger
consequences and a satisfying perspective emerges, in which a
reflexive and transitive relation $\rho$ and its $d$-ary polymorphisms
can be viewed as similar objects. 

We first consider what we call \emph{(weak) generalized
  quasiorders of simple arity} (Definition~\ref{Adef:Refl+Trans}),
which are relations  
$\rho$ which possess a `rectangular' coordinate system given by the
cube $B^d$ for some nonementy set $B$, i.e.\ that we consider
$\rho \subseteq A^{B^d}$. Actually, it is convenient to treat elements
of such relations as $d$-ary operations with inputs from $B$ whose
outputs are in $A$. This allows us view such $\rho \subseteq A^{B^d}$ and the
$d$-ary operations in $\Pol\rho \subseteq A^{A^d}$ through the same
lens, in particular, it makes as much sense to fix variables of a
function $f \in A^{A^d}$ at constants in $A$ as it does to fix
variables of a function $g \in A^{B^d}$ at variables in $B$. We use
this more general notion of translation to generalize reflexivity and
transitivity to the higher dimensional setting. After some analysis,
it unfolds that a $d$-dimensional (weak) generalized
quasiorder of simple arity is exactly the $d$-ary part of a certain kind of function
\emph{minion} (Proposition~\ref{prop:doublearrowisminion}).

In more detail, the minions that occur in our setting, in addition to
being closed under permutation, identification, and addition of
fictitious variables, are also closed under the action of fixing some
variables at constant values in $B$ (i.e.\ what we call \emph{basic
  translations}). Our definitions of \emph{(weak) reflexivity} and
transitivity can be thought of as local properties whose combination
ensure that a $d$-dimensional $\rho \subseteq A^{B^d}$ extends to a
function minion $\rho^{\leftrightarrow}$ which satisfies the minion
analogue of the property $\Xi_d$ (when stated as a condition on
$\Pol\rho$). We compare these two conditions:
\begin{align*}
  g \in \rho^{\leftrightarrow} &\iff \trl_d(g) \subseteq \rho 
  &&\text{(analogue of $\Xi_d$ for $\rho^{\leftrightarrow}$)} \\
  f \in \Pol\rho &\iff \trl_d(f) \subseteq \Pola[d]\rho
  &&\text{($\Xi_d$ as a condition on $\Pol\rho$)}
\end{align*}
So, the most basic kinds of relations $\rho$ that we found to
satisfy $\Xi_d$ are just fragments of minions which satisfy an
analogous version of the structure which $\Xi_d$ imposes on
$\Pol\rho$. In fact, since we know that $\Pol\rho$ is a minion (it
is already a clone), we find that $\Pola[d]\rho$, i.e. the set of
$d$-ary polymorphisms of $\rho$, is also a $d$-dimensional simple
arity generalized quasiorder whenever $\rho$ satisfies $\Xi_d$
(Proposition~\ref{prop:StarisCloneIffGQuord}). 

For example, if $\rho \subseteq 2^1$ is an ordinary quasiorder
relation on $A$, then $\rho^{\leftrightarrow}$ consists of all
functions $g: 2^n \to A$ all of whose unary basic translations belong
to $\rho$ (similarly, $\Pol\rho$ consists of all operations whose
unary basic translations are unary polymorphisms of
$\rho$). Geometrically, this can be thought of as the collection of 
all hypercubes whose vertices are labeled by $A$, with the property 
that all `oriented lines' within a particular $A$-labeled hypercube
are labeled by pairs from $\rho$. It is easy to check that this particular $\rho^{\leftrightarrow}$ is indeed a minion.

We then extend the class of generalized quasiorders of simple arity to
the generalized quasiorders of  \emph{compound arity}. There are several
ways that one could do this; here we elect to treat a relation 
\[
  \rho \subseteq A^{B_1^{d_1} \times \dots \times B_s^{d_s}}
\]
as a set of multisorted operations which take inputs from the product
of the simple arity sorts $B_1^{d_1}, \dots, B_s^{d_s}$, and then
demand that each of the relations produced from $\rho$ by `currying'
the different simple arities is a weak generalized quasiorder. Hence,
our definition of a generalized quasiorder of compound
arity exactly
mimics the definition of a higher dimensional equivalence relation of
arity $2^1 \times \dots \times 2^1$. One of our main results is that
$\Xi_d$ holds for such $\rho$, where now $d = d_1 + \dots + d_s$ is
the sum of all the constituent simple arities
(Theorem~\ref{thm:PolofGQuordMain2}).

The theory of generalized quasiorders provides a new tool to study
clones of operations with constants. Indeed, if $F$ is a function
clone with constants, then it is closed under the act of substituting
a variable by a constant, which can be viewed as taking lower
dimensional `rectangular slices' of a higher dimensional rectangular
object. The property $\Xi_d$ turns this around: given an operation
clone $F$ (which we view as a collection of labeled higher dimensional
rectangles), when is it the case that \emph{every} higher dimensional
rectangle whose $d$-dimensional slices belong to $F$ also belongs to
$F$? We show that this happens only when the clone fragment $F^{(d)}$
is a generalized quasiorder, and conversely, any generalized
quasiorder which is a clone fragment extends to such a clone
(Proposition~\ref{prop:StarisCloneIffGQuord}). Clone fragments which
are generalized quasiorders can be described as the closed sets of
clone fragments determined by the Galois connection
$\Pola[d]-\dGQuord$, which we spell out in
Section~\ref{sec:PolQuord}. The remainder of the paper is devoted to
examples.

We should make some remarks on the inevitable peculiarities regarding
constants that need to be dealt with in one way or another. There are
three distinct, yet interconnected, roles that constants play in this
article:
\begin{enumerate}[label=\textup{(\alph*)}]
\item Constant operations in an operation clone $F\subseteq \Op(A)$
  for a set $A$.
\item Constants in a set $B$ which are used to fix some variables of a
  function $g: B^n \to A$ at particular values to produce its basic
  translations.
\item Constants which occur in the \emph{image} of a relation $\rho$
  (informally, the collection of all possible outputs of elements of
  $\rho$).
\end{enumerate}
These three distinct occurrences of `constant' have some collisions
within our theory. For example, given a function clone $F$ over $A$,
we can consider (a) from the above list and suppose that $F$ has all
constant operations. Hence, in $F$ we are allowed to evaluate some
variables of a function at a constant operation, but these constant
operations are usually forbidden to be nullary in the standard
development of clone theory. On the other hand, if we allow the kinds
of constants that occur in (b) (with $B = A$), then $F$ will contain
all nullary operations corresponding to the values in the image of
functions in $F$. So, even when $A = B$, the two kinds of constants
from (a) and (b) above behave differently, provided one works with the
standard definition of a clone of operations.
	
There are both practical and aesthetic reasons to allow the nullary
constant substitutions that occur in (b). Indeed, for our treatment it
is important that the arity of a function decreases upon substitution
of a variable for a constant and that these substitutions are distinct
from other ways of manipulating functions (e.g.\ addition of
fictitious variables or identification of variables). If all possible
variables of a function are fixed at constant values, then a nullary
function is obtained, and it is natural to regard collections of
nullary functions as $0$-dimensional weak generalized
quasiorders. Aesthetically, there is no reason to forbid the inclusion
of $0$-dimensional weak generalized quasiorders from our theory, which
are unsurprisingly just subsets of the domain.

Actually, the $0$-dimensional weak generalized quasiorders are
connected to the constants that occur in (c) in the above list: the
image of a weak generalized quasiorder $\rho$ is a $0$-dimensional
weak generalized quasiorder and our definition of \emph{internal
  reflexivity} (Definition~\ref{Adef:Refl+Trans}) only implies that
all constant tuples of elements belonging to this $0$-dimensional
quasiorder belong to $\rho$ (reflexivity further imposes that all
constant tuples are included). Examples of this are already
well-known. The \emph{weak congruences} of an algebra $\mathbb{A}$ are
exactly the congruence relations of subalgebras of $\mathbb{A}$;
similarly, the weak generalized quasiorders over $A$ are just the
generalized quasiorders over subsets of $A$ (and so the invariant weak
generalized quasiorders of an algebra would be the invariant
generalized quasiorders of its subalgebras).

Given a weak generalized quasiorder $\rho$ over $A$ whose image is a
proper subset of $A$, it is straightforward to describe $\Pol\rho$
as a clone over $A$ in terms of $\Pol\rho$ as a clone over the image
of $\rho$: any operation belonging to the latter clone can be
arbitrarily extended to produce an operation in the former and
$\Pol\rho$ as a clone over $A$ consists of all such extensions. So,
if we are ultimately interested in polymorphism clones, it might seem
pointless to make the distinction. There is a good reason to do this
though, since our definition of generalized quasiorders
of compound arity
(Definition~\ref{Adef2:Refl+Trans}) requires the weak versions for
simple arities (Definition~\ref{Adef:Refl+Trans}).

The distinction between the type (a) constants and type (b) constants
above leads to a slight discrepancy in the use of a particular
notation. In Definition~\ref{MstarPoe}, we define for a clone fragment
$M \leq A^{A^d}$ the set of operations $M^{*_d}$, while in
Definition~\ref{def:star} we define for a rectangular relation
$\rho \subseteq A^{B^d}$ the collection of functions $\rho^{*_d}$ in
an essentially identical way as $M^{*_d}$, but now in a more general
setting. The only difference is that $M^{*_d}$ does not have nullary
operations, while $\rho^{*_d}$ does. This is a technical detail which
plays no essential role in our results. Actually, Behrisch develops
the standard Pol-Inv Galois connection for clones with nullary
operations \cite{Beh2014} and shows that this can be achieved with no
essential problems. So, if the reader likes, they are free to think of
the clones in this paper as having nullary operations, in which case
there is no discrepancy with the use of $^*_d$.

We also remark that, while we consider in this paper only clones and
relational clones over finite sets $A$, much of the theory of
generalized quasiorders goes through without any modifications for
infinite $A$. In particular, Theorem~\ref{thm:PolofGQuordMain2} is
true for arbitrary sets $A$. The relaxed assumption on the cardinality
of $A$ does not extend to the sets $B_1, \dots, B_s$, though, which
are necessarily finite since we consider only finitary relations over
$A$. The material in Section~\ref{sec:PolQuord}, on the other hand,
requires that the underlying set $A$ under consideration is
finite. Since we are only mentioning this here in the introduction, we
advise the cautious reader to ignore the possibility that $A$ can be
infinite in some of these results during their first reading.

Here is an outline of the organization of this paper:

In \emph{Section~\ref{sec:prelim}}, basic notation, notions and facts
are collected concerning preclones, clones, minions, and Galois
connections.  
In \emph{Section~\ref{sec:Xi}}
we define translations of operations of arbitrary arity and define our
`motivating' property $\Xi_d$ and its natural generalization $\wXi_{d}$.

In \emph{Section~\ref{sec:Rrelations}} we take some care to
distinguish the several notions of arity at play 
and define the crucial objects of our paper, namely \emph{R-relations of
simple and of compound arity} (Definition~\ref{def:Rrelation}).

In \emph{Section~\ref{sec:trl}}, the notion of \emph{translation} for
elements of any  
R-relation of simple arity is introduced, thus generalizing the
translations defined in Section~\ref{sec:Xi} for operations. Each
translation is a composition of so-called elementary basic or
fictitious translations.

In \emph{Section~\ref{sec:GQuord}} we define internal reflexivity,
reflexivity, and transitivity for R-relations of simple arity, based
on the adjoints of a monotone Galois connection
(Definitions~\ref{def:ostar} and \ref{Adef:Refl+Trans}). 
Naturally, a \emph{(weak) generalized quasiorder of simple arity} is an
R-relation which is (weakly) reflexive and transitive
(Definition~\ref{Adef:Refl+Trans}). Some examples and corresponding
figures demonstrate the introduced properties.
 A sequence of technical lemmas culminates in
Proposition~\ref{prop:doublearrowisminion}, which is our minion
characterization of generalized quasiorders of simple arity. We also
prove Lemma~\ref{lem:PolofGQuord}, which is our main tool for proving
the property $\Xi_d$ for generalized quasiorders.
  
After defining (internal) reflexivity and transitivity for R-relations
of compound arity and thus also \emph{(weak)
generalized quasiorders of compound arity}, in
\emph{Section~\ref{sec:GQuord2}} we prove 
our Main Theorem~\ref{thm:PolofGQuordMain2} which states that $\Xi_d$
holds for \textsl{any} $d$-dimensional generalized quasiorder.

In \emph{Section~\ref{sec:PolQuord}} we consider the Galois connection
  $\Pola[d]-\dGQuord$ relating $d$-ary
polymorphisms to $d$-dimensional generalized quasiorders. We
characterize the Galois closed clone fragments $M$
(Theorem~\ref{thm:dPolQuord} and Corollary~\ref{A3cor}, also via
conditions for $Q$ if $M$ is of the form $M=\Pola[d]Q$).

 In \emph{Section~\ref{sec:GQuord-dimension}} we introduce the
  \emph{generalized quasiorder dimension} of a clone $F$ being the least $d$
  such that $F$ is determined by a $d$-dimensional generalized
  quasiorder (namely by $F^{(d)}$ considered as R-relation).
We show that all Boolean
clones with constants are either $1$- or $2$-dimensional
(Example~\ref{ex:Booleancase}) and show that
for larger finite sets, any dimension can occur
(Example~\ref{ex:3GQuord}). We also give an 
example of a maximal clone which has no non-trivial generalized quasiorders
(and therefore has dimension $\infty$, since it cannot be determined by
a generalized quasiorder, Example~\ref{ex:3noGQuord}).

Finally, in \emph{Section~\ref{sec:simvscom}}, we provide
examples which justify the distinction between 
generalized quasiorders of simple arity and of compound arity.
In the last \emph{Section~\ref{sec:conclusion}}  
we state our conclusions and list some questions and problems for
further research.

    \section{Preliminaries}\label{sec:prelim}

    \begin{notation}\label{notations}
We briefly introduce or recall some notions and notation. $\N:=\{0,1,2,\dots\}$ ($\N_{+}:=\N\setminus\{0\}$)
is the set of (positive) natural numbers. $\Opa[n](A)$ and
$\Rela[m](A)$ denote the sets of $n$-ary operation $f:A^{n}\to A$ and
$m$-ary relations $\rho\subseteq A^{m}$ on some base set $A$ (assumed
to be finite throughout the paper),
respectively.  $\Op(A):=\bigcup_{n\in\N_{+}}\Opa[n](A)$,
$\Rel(A):=\bigcup_{m\in\N_{+}}\Rela[m](A)$. 

For an arbitry set
$\alpha$, $A^{\alpha}$ is the set of all functions $g:\alpha\to
A$. Given a subset $\rho\subseteq A^{\alpha}$ (also called relation), the \Def{image of
  $\rho$} is the union of all images of the elements of $\rho$:
\begin{align*}\label{def:ImageRrelation}\tag*{\ref{notations}(1)}
  \Ima\rho:=\bigcup\{\Ima g\mid g\in\rho\}\subseteq A.
\end{align*}
Clearly, $\rho\subseteq (\Ima\rho)^{\alpha}$, thus $\rho$ can be
considered  also as a relation on the base set $\Ima\rho$ instead of
$A$.

Concerning constants, let  $\cC=\cC_{A}$ be the set of all constant
operations in $\Op(A)$ and $\cC^{(n)}:=\cC\cap\Opa[n](A)$. Unary or $n$-ary
operations constant to $a$ are denoted by $c_{a}$
(i.e. $c_{a}:x\mapsto a$) or $c_{a}^{(n)}$, respectively. For
$K\subseteq A$, let $\cC_{K}$ denote the set of constant operations
which are constant to an element from $K$. 

    \end{notation}

    \begin{definition}\label{M5}

      A set $F \subseteq \Op(A)$ is called a \New{preclone} if it
      contains $\id_{A}$ and is closed under the operations $\zeta$,
      $\tau$ and $\circ$ that are defined as follows. Let
      $f\in\Opa[n](A)$ and $g \in \Opa[m](A)$, $n,m\in\N_{+}$. Then

      \begin{enumerate}[label=\textup{(\arabic*)}]

      \item\label{M5-1} $\id_A(x):=x$ (\New{identity operation});

      \item \label{M5-2}
        $(\zeta f)(x_{0},x_{1},\dots,x_{n-1}) :=
        f(x_{1},\dots,x_{n-1},x_{0})$ (\New{cyclic shift}),\newline if
        $n = 1$ then $\zeta f := f$;

      \item \label{M5-3} $(\tau f)(x_{0},x_{1},x_{2},\dots,x_{n-1})
        := f(x_{1},x_{0},x_{2},\dots,x_{n-1})$\\
        (\New{permuting the first two arguments}), if $n = 1$ then
        $\tau f := f$;

      \item \label{M5-4}
        $ (f\circ g)(x_{0},\dots,x_{m-1},x_{m},\dots,x_{m+n-2})$
        \newline \hspace*{\fill}
        $:=f(g(x_{0},\dots,x_{m-1}),x_{m},\dots,x_{m+n-2})$
        (\New{composition}).
    
      \end{enumerate}

      For later use we introduce here also the operations $\nabla$
      (\New{adding a fictitious argument} at first place) and $\Delta$
      (\New{identification of the first two arguments}):
      \begin{enumerate}[label=\textup{(\arabic*)}]
        \setcounter{enumi}{4}
      \item \label{M5-5}
        $(\nabla f)(x_{0},x_{1},\dots,x_{n})
        :=f(x_{1},\dots,x_{n})$,

      \item\label{M5-6}
        $(\Delta
        f)(x_{0},\dots,x_{n-2}):=f(x_{0},x_{0},\dots,x_{n-2})$ if
        $n\geq 2$, and $\Delta f=f$ for $n=1$.
      \end{enumerate}
    \end{definition}


\begin{remarksnote}
  The $(m+n-1)$-ary function $f \circ g$ (defined in \ref{M5-4})
  sometimes is called \New{linearized composition} (or
  \New{superposition}), because this is a special case of the general
  \emph{linearized composition}, \emph{linearization} or
  \emph{superposition} mentioned in \cite[2.1]{BruDPS93},
  \cite[page~2]{GraW1984} or \cite[Section~2.1]{Leh2010},
  respectively.

\noindent
Preclones, also known as operads, can be thought as ``clones where
identification of variables is not allowed'' (cf.~\ref{M5A}).
The term \emph{preclone} was introduced by \'Esik and Weil
\cite{EsiW2005} in a study of the syntactic properties of recognizable
sets of trees. A general characterization of preclones as Galois
closures via so-called matrix collections can be found in
\cite{Leh2010}.
The notion of \emph{operad} originates from the work in algebraic
topology by May \cite{May1972} and Boardman and Vogt
\cite{BoaV1973}. For general background and basic properties of
operads, we refer the reader to the survey article by Markl
\cite{Mar2008}.
\end{remarksnote}

Clones are special preclones. There are many (equivalent) definitions
of a clone. One of these definitions is that a clone is a set
$F\subseteq\Op(A)$ closed under \ref{M5}\ref{M5-1}-\ref{M5-6},
\cite[1.1.2]{PoeK79}. Therefore we have:

\begin{lemma}\label{M5A}
  A preclone is a clone if and only if it is also closed under
  $\nabla$ (adding ficticious variables) and $\Delta$ (identification
  of variables).
\end{lemma}

For $F\subseteq \Op(A)$, the clone generated by $F$ is denoted by
$\Sg{F}$ or $\Sg[A]{F}$. The $d$-ary operations $F^{(d)}$ of a clone
$F$ form a so-called clone fragment. More precicely, we define:

\begin{definition}\label{clonefragment}
  A set $M\subseteq \Opa[d](A)=A^{A^{d}}$ is called a \Def{$d$-ary
    clone fragment}, notation $M\leq A^{A^{d}}$, if
  $M=\Sg[A]{M}^{(d)}$. A $d$-ary clone fragment $M$ is called
  \Def{$d$-ary clone fragment with constants}  if it contains all
  ($d$-ary) constant 
  functions ($\cC^{(d)}\subseteq M$), i.e., we have $M=\Sg[A]{M\cup \cC}^{(d)}$.

\end{definition}
\begin{remarknote}
  From the characterization of clones (on finite sets $A$) by
  invariant relations, it follows that $M$ is a $d$-ary clone fragment
  (clone fragment with constants, resp.) if and only if there exists a relation
  $\rho$ on $A$ such that $M=\Pola[d]\rho$ (and $\cC\preserves\rho$,
  resp.).
\end{remarknote}

Clones as well as preclones are special minions which we shall
consider in connection with 
R-relations (cf.\ \ref{prop:doublearrowisminion}). 

\begin{definition}[minion]\label{def:minion}
Let $\Op(B,A)$ denote the set $\bigcup_{n\in\N}A^{B^{n}}$ of all
functions $g:B^{n}\to A$ ($n\in\N$). For $g:B^{n}\to A$ and 
$\pi:n\to m$ ($m\in\N$, $m\in\N_{+}$?) the function $g^{\pi}:B^{m}\to A$ 
given by
\begin{align*}
  g^{\pi}(x_{0},\dots,x_{m-1}):=g(x_{\pi(0)},\dots,x_{\pi(n-1)})
\end{align*}
is called a \emph{minor} of $g$. Note that $g^{\pi}=g\circ\mu_{\pi}$ for
$\mu_{\pi}:B^{m}\to B^{n}: 
(x_{0},,\dots,x_{m-1})\mapsto (x_{\pi(0)},\dots,x_{\pi(n-1)})$. The
$d$-ary part of a minion is denoted by $M^{(d)}:=M\cap A^{B^{d}}$.

A subset $M\subseteq\Op(B,A)$
is called a \emph{(function) minion} if it is closed under taking minors
(equivalently, permutation of variables, identification of variables
and adding fictitious variables analogous to
Definition~\ref{M5}\ref{M5-2},\ref{M5-3},\ref{M5-5},\ref{M5-6}). 

 A \emph{$d$-ary minion fragment} is a
set $\rho\subseteq A^{B^{d}}$ such that there exists a minion
$M\subseteq\Op(B,A)$ with $M^{(d)}=\rho$.

\end{definition}

\begin{Galoisconnections}\label{Galoisconnections}
  We briefly recall the notion of Galois connection. Let $(P,\leq_{P})$
  and $(Q,\leq_{Q})$ be posets (in most cases one takes power sets
  $(\Pow(X),\subseteq)$ and $(\Pow(Y),\subseteq)$ with respect to
  inclusion). A pair $(\phi,\psi)$ of mappings $\phi:P\to Q$ and
  $\psi:Q\to P$ form an \New{(antitone) Galois connection} or a
  \New{monotone Galois connection}, respectively, if 
  \begin{align*}
   p\leq_{P}\psi(q)&\iff \phi(p)\geq_{Q}q
\text{ or }\\
   p\leq_{P}\psi(q)&\iff \phi(p)\leq_{Q}q,
\text{ respectively.}
  \end{align*}

For an (antitone) Galois connection, the mappings $\phi$ and $\psi$
are antitone (order reversing) and are sometimes called
\emph{polarities}, while in the case of a monotone Galois connection,
$\phi$ and $\psi$ are monotone mappings (order homomorphisms) and are
called respectively the \emph{lower adjoint} and \emph{upper adjoint}.  

For both kinds of Galois connections each mapping uniquely determines
the other (provided they exist and form a Galois connection).

We mention two important examples needed later.

\textbf{(1) The Galois connection induced by a binary relation
  $R\subseteq X\times Y$:}

Consider $P=\Pow(X)$ and $Q=\Pow(Y)$ with inclusion as
ordered sets. Then the mappings $\phi:P\to Q$ and $\psi:Q\to P$ given
by 
\begin{align*}
  \phi(X')&:=\{y\in Y\mid \forall x\in X': xRy\} \text{ (for
            $X'\subseteq X$),}\\ 
  \psi(Y')&:=\{x\in X\mid \forall y\in Y': xRy\}
           \text{ (for $Y'\subseteq Y$)}
\end{align*}
 form an (antitone)
Galois connection. A famous Galois connection of this kind
(emphasized as ``\textsl{The most basic Galois connection of algebra}'' in
\cite[p.~147]{McKMT87})
is given by $X=\Op(A)$, $Y=\Rel(A)$ and $R=\preserves$ with the
operators
\begin{align*}
  \Inv F&:=\{\rho\in\Rel(A)\mid \forall f\in F: f\preserves\rho\}
          \text{ (invariant relations of $F$)}\\ 
  \Pol T&:=\{f\in\Op(A)\mid \forall \rho\in T: f\preserves\rho\}
          \text{ (polymorphisms of $T$)}
\end{align*}
for $F\subseteq\Op(A)$ and $T\subseteq\Rel(A)$. 
Here $f\preserves\rho$
($f$ \Def{preserves} $\rho$) means that $f$ maps elements (tuples) from
$\rho$ into $\rho$ (i.e., $f(r_{1},\dots,r_{n})\in\rho$ for
$r_{1},\dots,r_{n}\in\rho$) if applied coordinatewise.

\textbf{(2) The monotone Galois connection induced by a mapping\newline
  \hspace*{4ex}$\lambda:X\to \Pow(Y)$:}

Again we consider the ordered sets $P=\Pow(X)$ and $Q=\Pow(Y)$. 
Then the mappings $\phi:P\to Q$ and $\psi:Q\to P$ given by
\begin{align*}
  \phi(X')&:=\bigcup_{x\in X'} \lambda(x) \text{ (lower adjoint),}\\
  \psi(Y')&:=\{x\in X\mid \lambda(x)\subseteq Y'\} \text{ (upper adjoint)}
\end{align*}
for $X'\subseteq X$ and $Y'\subseteq Y$ form a monotone Galois connection.

\textbf{(3) Properties of a Galois connection (monotone Galois
  connection, respectively):} For $p,p'\in P$ and $q,q'\in Q$ we have
\begin{enumerate}[label=\textup{(\Roman*)}]

\item $p\leq_{P} p'\implies \phi(p)\geq_{Q}\phi(p')$,
  ($\phi(p)\leq_{Q}\phi(p')$, resp.),
\item $q\leq_{Q} q'\implies \psi(q)\geq_{P}\psi(q')$,
  ($\psi(q)\leq_{P}\psi(q')$, resp.),
\item $p\leq_{P}\psi(\phi(p))$,
\item $q\leq_{Q}\phi(\psi(q))$ ($\phi(\psi(q))\leq_{Q}q$, resp.),
\item $\phi(\psi(\phi(p)))=\phi(p)$,
\item $\psi(\phi(\psi(q)))=\psi(q)$,
\item $p\mapsto \psi(\phi(p))$ is a closure operator,
\item $q\mapsto \phi(\psi(q))$ is a closure operator (kernel operator,
  resp.).
\end{enumerate}
  It is known that a pair $(\phi,\psi)$ is a Galois connection (or
  monotone Galois connection, respectively) if and only
  if it satisfies (I)-(IV). Therefore these four properties often are
  used also as definition.

\end{Galoisconnections}

\begin{remark}\label{rem:PolInv}
The Galois closures (with respect to the operators in (VII) and
(VIII)) of the ``most basic'' Galois connection $\Pol-\Inv$ (see
\ref{Galoisconnections}(1) above)  are
well-known and can be
characterized as follows: $\Pol\Inv F=\Sg[A]{F}$ (clone generated by
$F$), $\Inv \Pol Q=[Q]_{\exists,\land,=}$ (relational clone, generated
by $Q$, equivalently characterizable as closure with respect to
primitive positive formulas (pp-formulas), i.e., formulas containing variable and
relational symbols and only $\exists,\land,=$). We refer to, e.g.,
\cite[1.2.1, 1.2.3, 2.1.3(i)]{PoeK79}, \cite{BodKKR69a}, \cite{Poe04a}, \cite{KerPS2014}.

For the restriction to $d$-ary operations, i.e., for the Galois
connection $\Pola[d]-\Inv$ we have:
$\Pola[d]\Inv M=\Sg[A]{M}^{(d)}$
($d$-ary clone fragment generated by $M\subseteq A^{A^{d}}$), $\Inv\Pola[d]
Q=\dLOC[d][Q]_{\exists,\land,=}$ (\New{$d$-local closure} of the relational
clone generated by $Q$). Here the operator $\dLOC[d]$  for
$Q\subseteq\Rel(A)$ is defined as follows:
\begin{align*}
  \dLOC Q:=\{\rho\in\Rel(A)\mid \forall\, B\subseteq\rho,\; |B|\leq
  d\;\exists\, \sigma\in Q: B\subseteq\sigma\subseteq\rho\}.
\end{align*}
It is equivalently definable by the closure under so-called ``$d$-directed
unions'' (generalizing disjunction in the $1$-dimensional case $d=1$).
We refer to, e.g.,
\cite[1.9, 4.2(b), 1.13(ii)]{Poe80}.
\end{remark}

\section{Translations and the properties $\Xi_{d}$ and $\wXi_{d}$}\label{sec:Xi}
      
In this section we first introduce a natural generalization of the classical
notion of a \emph{translation} of an $n$-ary operation $f: A^{n}\to A$
for some nonempty set $A$, which is a unary polynomial function
obtained from $f$ by evaluating all but one of its arguments at a
constant.
Namely, we do not
necessarily only consider unary polynomial functions obtained from
$f$, but also binary, ternary, and so on.
We then define the property $\Xi_{d}$, which is the foundation for the
central motivating question of this paper, and consider some questions
connected with this property.

\begin{definition}[translations of operations]\label{def:trlA}
  Let $f:A^{n}\to A$ be an $n$-ary operation and $d\in \N$.
An operation $g:A^{d}\to A$ is called a \Def{$d$-ary translation} or
\Def{$d$-translation} of $f$, if $g$ can be obtained from $f$
by substituting some of its variables for constants and then prepending and
  appending some fictitious variables. 
Operations $g$ obtained from $f$ by applying only one of these steps
(respectively called \Def{basic} and \Def{fictitious}
$d$-translations) are defined as follows:
\begin{align*}\label{trlA-1}\tag*{\ref{def:trlA}(1)}
    g(x_{0},\dots,x_{d-1}):= &f(\underbracket{\dots}_{c\in A},
       x_{0},\underbracket{\dots}_{c\in A},
        x_{1},\underbracket{\dots}_{c\in A},
        \dots,x_{d-1},\underbracket{\dots}_{c\in A})
     &\text{ if $n\geq d$,}\\
\label{trlA-2}\tag*{\ref{def:trlA}(2)}
     g(x_0, \dots, x_{d-1}) := &f(x_{u},\dots,
                              x_{u+n-1}) \text{ ($u\in\{0,\dots,d-n\}$)}
                      &\text{ if $n\leq d$.}
\end{align*}
 
Note that for both basic and fictituous translations, if $d=n$ we get $g=f$. 
Let $\trl_{d}(f)$ ($\btrl_{d}(f)$ and $\ftrl_{d}(f)$, respectively) denote the
set of all $d$-translations (basic and fictitious, respectively) of $f$.
Thus $\trl_{d}(f)=\ftrl_{d}(\btrl(f))$ by definition. We put
$\btrl_{d}(f)=\emptyset$ for $n<d$, and $\ftrl_{d}(f)=\emptyset$ for $n>d$.
\end{definition}

\begin{remarks}\label{rem:trlA} (a) For all what follows in this section it makes no
  difference if one considers $\trl_{d}(f)$ in full generality or only
  the above special $d$-translations \ref{trlA-1} or
  \ref{trlA-2}. Therefore the reader
  may think of a $d$-translation (in $\trl_{d}(f)$) as one of \ref{trlA-1} or \ref{trlA-2} (definitions and results are independent of this
  point of view, see, e.g., the next Lemma~\ref{lem:trlApreserves}).
 Later, in a more general context, we shall
see that $\trl_{d}(f)$ is closed under
substituting constants for variables and prepending or appending
fictitious variables in arbitrary order (and thus coincides with the
Definition of $\trl_{d}(f)$ in \ref{def:trl1}, see \ref{lem:trldecomposition}\ref{trldecomposition-iii}).

(b) \ref{trlA-1} is a  \Def{basic $d$-translation} obtained from $f$ by substituting constants on
  $n-d$ places and filling up the (non-substituted) places by the
  variables $x_{0},x_{1},\dots,x_{d-1}$, and \ref{trlA-2} is a
  \Def{fictitious $d$-translation} obtained from $f$ by 
prepending $u$ and appending
  $d-n-u$ many fictitious variables.
%

  (c) For $d=1$ in \ref{trlA-1} our notion of basic translation
  corresponds to the notion of \Def{translation} in
  \cite[Definition~1.4.7]{Ihr2003}, or \Def{basic translation} in
  \cite{Mal1963}, or \Def{$1$-translation} in \cite[p.~375]{Gra2008}
  However, Gr{\"a}tzer's $d$-translations differs from our definition
  for $d>1$. Notice that creating the distinction between a basic
  $d$-translation and a fictitious $d$-translation was not necessary
  in the classical setting $d=1$, since $\trl_1(f)$ only enlarges
  $\btrl_1(f)$ by those constant unary functions which take a value in
  the range of $f$. For larger values of $d$, we need to consider the
  basic translations obtained from $f$ which fix all but $t$ many
  variables at a constant value, for any $0 \leq t \leq d$, and then
  adjust the arities with fictitious translations.
  
\end{remarks}

\begin{lemma}\label{lem:trlApreserves}
  Let $\rho\in\Rel(A)$ and $f:A^{n}\to A$ and assume
$\cC\preserves\rho$. Then
\begin{align*}
  \trl_{d}(f)\preserves\rho\iff \btrl_{d}(f)\cup \ftrl_{d}(f)\preserves\rho
   &\iff
  \begin{cases}
    \btrl_{d}(f)\preserves\rho &\text{if } d\leq n,\\
    \ftrl_{d}(f)\preserves\rho &\text{if } d\geq n.
  \end{cases}
\end{align*}

\end{lemma}

\begin{proof}
  The second implication is clear in both directions by
  definition of  $\btrl$ and $\ftrl$. The first implication
  ``$\Longrightarrow$'' is trivial. To show ``$\Longleftarrow$'', we
  distinguish two cases:

Case $d\geq n$: Then $\emptyset\neq\ftrl_{d}(f)$ and we have
$\Sg{\{f\}}=\Sg{\ftrl(f)}\preserves\rho$ because
two functions which are equal up to fictitious 
  variables generate the same clone and therefore preserve the same
  relations. Thus $f\preserves\rho$ which implies
  $\trl_{d}(f)\preserves\rho$ since $\trl_{d}(f)\subseteq
  \Sg{\{f\}\cup \cC}$ (by definition) and $\cC\preserves\rho$.

  Case $d\leq n$: Let $g\in\trl_{d}(f)$. We have to show
  $g\preserves\rho$. By definition, there exists $t\in\{0,1,\dots,d\}$
  such that $g\in\ftrl_{d}(\btrl_{t}(f))$ (i.e.\ $g$ is obtained from
  $f$ by first fixing at least $n-t$ of its variables at constants and
  then enlarging the arity with initial and terminal fictitious
  variable sequences). Since
  $\btrl_t(f)\subseteq\Sg{\btrl_d(f)\cup \cC}$ (as $t\leq d$)
  and $\btrl_{d}(f)\preserves\rho$ by assumption, we get
  $\btrl_{t}(f)\preserves\rho$ and therefore also
  $g\in\ftrl_{d}(\btrl_{t}(f))\preserves\rho$, which is what was
  to be shown.
\end{proof}

Having defined a suitable generalization of a translation, we can now
define the central property $\Xi_{d}$ which we investigate in this
paper. But we need before another definition, which shall be discussed
in more generality later, cf.\ Definition \ref{def:star}.

\begin{definition}\label{MstarPoe}
  For a $d$-ary clone fragment $M\leq A^{A^{d}}$
  (cf.~\ref{clonefragment}) with
  $\cC^{(d)}\subseteq M$
  we define
  \begin{align}\label{MstarPoe1}\tag*{\ref{MstarPoe}(1)}
    M^{*_d}:=\{f\in \Op(A)\mid \trl_{d}(f)\subseteq M\}.
  \end{align}
\end{definition}

\begin{remarks}\label{rem:MstarPoe}
\leavevmode

(A) According to \ref{Galoisconnections}(2),
  ${}^{*_{d}}:\Pow(\Opa[d](A))\to\Pow(\Op(A)):M\mapsto M^{*_{d}}$ is the
upper adjoint of a monotone Galois connection with
lower adjoint
$\trl_{d}:\Pow(\Op(A))\to\Pow(\Opa[d](A))$, induced by the mapping
$\lambda:f\mapsto \trl_{d}(f)$. For a more general context see
  Definition~\ref{def:star} (for R-relations).

Since the upper adjoint of a monotone Galois connection
 is always meet-preserving, we have:
   $(\bigcap_{i\in I} M_{i})^{*_{d}}=\bigcap_{i\in I}M_{i}^{*_{d}}$
for clone fragments $M_{i}\leq A^{A^{d}}$, $i\in
I$. This also can be seen directly using the definitions:
\begin{align*}
  f\in(\textstyle\bigcap_{i\in I} M_{i})^{*_{d}}
&\iff \trl_{d}(f)\subseteq\textstyle\bigcap_{i\in I}M_{i}
\iff \forall i\in I: \trl_{d}(f)\subseteq M_{i}\\
&\iff \forall i\in I: f\in M_{i}^{*_{d}}
\iff f\in \textstyle\bigcap_{i\in I}M_{i}^{*_{d}}. \tag*{\qed}
\end{align*}

(B) For the first $(d+1)$-ary projection $e_{0}^{(d+1)}$ we have
$\trl_{d}(e_{0}^{(d+1)})=\{e_{0}^{(d)}\}\cup \cC^{(d)}$. Therefore,
for $M\subseteq\Opa[d](A)$ with $e_{0}^{(d)}\in M$ (e.g., for a clone
fragment), we
have $e_{0}^{(d+1)}\in M^{*_{d}}$ if and only if $\cC^{(d)}\subseteq M$.
\end{remarks}

\begin{definition}[\textbf{The property
    $\boldsymbol{\Xi_{d}}$}]\label{def:Xi}
  For a relation $\rho\in\Rel(A)$ we consider the following property
  $\Xi_{d}$ in three equivalent formulations:
  \begin{align}\label{Xi1}\tag*{(*)}
    \Xi_{d}(\rho):&\iff\quad \forall f\in\Op(A):\;
                    f\preserves\rho\iff
                    \trl_{d}{f}\preserves\rho\\\label{Xi2}\tag*{(**)} 
                  &\iff \quad \forall f\in\Op(A):\;
                    f\in\Pol\rho\iff \trl_{d}{f}\subseteq\Pola[d]\rho\\\label{Xi3}\tag*{(***)}
                  &\iff \quad 
                    \Pol\rho=(\Pola[d]\rho)^{*_d}.
  \end{align}
  This can be extended to sets $Q\subseteq\Rel(A)$ just by
  substituting $Q$ for $\rho$ in the above definition, e.g.,
  $\Xi_{d}(Q)\iff\Pol Q=(\Pola[d] Q)^{*_d}$.
\end{definition}

\begin{remarksnote} 
Because of Lemma~\ref{lem:trlApreserves}
the condition $\trl_{d}(f)\preserves\rho$ in \ref{Xi1} can
be reduced to basic or fictitious translations.

  As already mentioned in the introduction, the property $\Xi_{d}$
  (generalizing the property $\Xi_{1}$ for congruence relations) leads
  to the central motivating question: 
\begin{center}
\emph{Which relations do satisfy
    $\Xi_{d}$?}
\end{center}
We observe two necessary conditions. The implication
``$\Longrightarrow$" in \ref{Xi1} leads to the requirement that all
constants should preserve $\rho$. In fact, since
$e_{0}^{(d+1)}\preserves\rho$ is trivially satisfied, \ref{Xi1}
gives $\cC^{(d)}\subseteq\trl_{d}(e_{0}^{(d+1)})\preserves\rho$
(cf.\ \ref{rem:MstarPoe}(B)).
Thus we also can conclude $\Ima\rho=A$ (this is actually a special case of Lemma~\ref{lem:weakimage}). 
Further, \ref{Xi3} implies 
that $M^{*_{d}}$ is a clone for $M=\Pola[d]\rho$ whenever
$\Xi_{d}(\rho)$, in particular $M=(M^{*_{d}})^{(d)}$ must be a clone
fragment. This also motivates why we restricted to clone fragments $M$
containing all constants in Definition~\ref{MstarPoe}.
In general, a clone fragment $M\leq\Opa[d](A)$ is a clone fragment
with constants if and only if $M^{*_{d}}$ contains the
  projection $e^{d+1}_{0}$.
\end{remarksnote}

Nevertheless we can relax the condition $\cC\preserves\rho$ by
considering only a subset of constants which leads to the property
$\wXi_{d}$ below.

\begin{definition}\label{def:weakXi}
Let $K\subseteq A$. If in \ref{trlA-1} only constants from $K$ are
allowed to be substituted, then we call
\begin{align*}\label{weakXi-1}\tag*{\ref{def:weakXi}(1)}
    g(x_{0},\dots,x_{d-1}):= &f(\underbracket{\dots}_{c\in K},
       x_{0},\underbracket{\dots}_{c\in K},
        x_{1},\underbracket{\dots}_{c\in K},
        \dots,x_{d-1},\underbracket{\dots}_{c\in K})
     &\text{ if $n\geq d$}
\end{align*}
a \New{basic $d$-translation of $f$ over $K$}. Likewise all notions
defined in \ref{def:trlA} can be considered ``over $K$''. We shall use
the notation $\btrl^{K}_{d}(f)$ and $\trl^{K}_{d}(f)$
for (basic)
$d$-translations of $f$ over $K$ ($\ftrl_{d}$ remains
untouched). Definition~\ref{MstarPoe} then modifies to 
  \begin{align}\label{weakXi-2}\tag*{\ref{def:weakXi}(2)}
    M^{*^{K}_d}:=\{f\in \Op(A)\mid \trl^{K}_{d}(f)\subseteq M\}
  \end{align}
for a $d$-ary clone fragment $M$ with $\cC_{K}^{(d)}\subseteq M$
($*^{K}_{d}$ being the upper adjoint to $\trl^{K}_{d}$).
The property $\Xi_{d}$ (cf.~\ref{def:Xi}) shall be modified only for a special $K$,
namely $K=\Ima\rho$ (cf.~\ref{notations}), called ``weak $\Xi$'':
  \begin{align}\label{weakXi-3}\tag*{\ref{def:weakXi}(3)}
    \wXi_{d}(\rho):&\iff \left[\forall f\in\Op(A):\;
                    f\preserves\rho\iff
                    \trl^{\Ima\rho}_{d}{f}\preserves\rho\right]\\
\label{weakXi-4}\tag*{\ref{def:weakXi}(4)}
      &\iff \Pol\rho=(\Pola[d]\rho)^{*^{\Ima\rho}_{d}}
  \end{align}
The ``w'' in $\wXi_{d}$ stands for ``\textbf{w}eak'' because the condition
$ \trl^{\Ima\rho}_{d}{f}\preserves\rho$ is weaker than  
$\trl_{d}{f}\preserves\rho$.
\end{definition}

\begin{remarknote}
  Clearly, \ref{weakXi-3} and \ref{weakXi-4} correspond to
  \ref{def:Xi}\ref{Xi1} and \ref{Xi2}. Since
  $\trl_{d}(f)=\trl_{d}^{A}(f)$, we have
  $\Xi_{d}(\rho)\iff \wXi_{d}(\rho)\land\Ima\rho=A$. As we shall see
  in Proposition~\ref{M6}, $M^{*_{d}}$ is always a preclone. For
  $M^{*^{K}_{d}}$ this is not true in general but it remains true if
  $K$ is a subalgebra of $(A,\Pol\rho)$ (although we prove it only for
  $K=A$ in Proposition~\ref{M6}). The next
lemma shows that $\Ima\rho$ is always a subalgebra if $\wXi_{d}(\rho)$.
\end{remarknote}

\begin{lemma}\label{lem:weakimage}
  $\wXi_{d}(\rho)$ implies that $\Ima\rho$ is the least (nontrivial)
  subalgebra of 
  $(A,\Pol\rho)$ and $c_{a}\in\Pol\rho$ for all $a\in\Ima\rho$, i.e., $\cC_{\Ima\rho}\subseteq\Pol\rho$.
\end{lemma}

\begin{proof}
  Let $K:=\{a\in A\mid c_{a}\in\Pol\rho\}$. It is
  clear that $K$ is a subalgebra
  (since $c_{f(a_{1},\dots,a_{n})}=f(c_{a_{1}},\dots,c_{a_{n}})\in\Pol\rho$ for
   $c_{a_{1}},\dots,c_{a_{n}},f\in\Pol\rho$) and it is contained in
   each subalgebra $B$ of $(A,\Pol\rho)$ (since, for $b\in B$,
   $c_{a}(b)=a$ implies $a\in B$ for each $a\in K$). We are going to
   prove $K=\Ima\rho$, which will finish the proof. In fact, $a\in K$,
   i.e., $c_{a}\preserves\rho$, implies $(a,\dots,a)\in\rho$ and
   therefore $a\in\Ima\rho$. Conversely, let $a\in\Ima\rho$.
  The following
  argument is a generalization of the argument given in the
  remarks after
  Definition~\ref{def:Xi} for the case when $\Ima\rho = A$.
The first $(d+1)$-ary projection $e_0^{d+1}$ is a
  polymorphism of every relation, so in particular
  $e_0^{d+1} \preserves \rho$. By ``$\implies$'' of~\ref{weakXi-3}, it
  follows that $e_0^{d+1}(a,x_1, \dots, x_d) = c_a(x_1, \dots, x_d)$
  also preserves
  $\rho$ (since it belongs to $\trl_{d}^{\Ima\rho}e_{0}^{d+1}$). Hence, $c_{a}\preserves\rho$, i.e., $a\in K$. 
\end{proof}

The condition
$\trl^{\Ima\rho}_{d}(f)\preserves\rho$ in $\wXi_{d}(\rho)$ can be simplified analogously
to  Lemma~\ref{lem:trlApreserves} as follows.

\begin{lemma}\label{lem:weaktrlApreserves}
  Let $\rho\in\Rel(A)$ and $f:A^{n}\to A$ and assume
$\cC_{\Ima\rho}\preserves\rho$. Then
\begin{align*}
  \trl^{\Ima\rho}_{d}(f)\preserves\rho&{\iff} 
\btrl^{\Ima\rho}_{d}(f)\cup \ftrl_{d}(f)\preserves\rho
   {\iff}
  \begin{cases}
    \btrl^{\Ima\rho}_{d}(f)\preserves\rho &\text{if } d\leq n,\\
    \ftrl_{d}(f)\preserves\rho &\text{if } d\geq n.
  \end{cases}
\end{align*}

\end{lemma}

\begin{proof} Looking at the proof of Lemma~\ref{lem:trlApreserves},
  it turns out that each constant which 
  appears in the proof never changes (no new constants come into play)
  therefore each $\btrl$ and $\trl$ can be replaced by 
  $\btrl^{K}$ and $\trl^{K}$ for an arbitrary $K\subseteq A$.
\end{proof}

The following example is a generalization of
  \cite[Example~2.4]{JakPR2023a} and shows that $M^{*_{d}}$ is not a
  clone in general (for $M$ being a clone fragment containing all
  constants).
 
\begin{example}\label{ex:notclone}   Let $A=\{0,1,2\}$ and $M:=\Pola[d]\rho$ ($d\in\N_{+}$)
  for the binary relation
  $\rho=\{(a,a')\in A^{2}\mid a\leq a'\leq a+1\}=
  \big(\begin{smallmatrix} 0&1&2&0&1\\0&1&2&1&2
  \end{smallmatrix}\big)$.
  Define $f:A^{d+1}\to A$ as follows:
  \begin{align*}
    f(x_{1},\dots,x_{d+1}):=
    \begin{cases}
      0&\text{if }x_{1}+\ldots+x_{d+1}\leq d+1,\\
      2&\text{if } x_{1}=\ldots=x_{d+1}=2, \text{ i.e., }\sum_{i=1}^{d+1} x_{i}=2d+2,\\
      1&\text{otherwise.}
    \end{cases}
  \end{align*}

  Then each basic $d$-translation $g\in\btrl_{d}(f)$, cf.\ \ref{trlA-1}, preserves
  $\rho$, what we are going to show first.  

  By definition, permuting the arguments in $f$ does not change the
  function value. Therefore we can assume without loss of generality,
  that a basic $d$-translation $g$ of $f$ has the form
  $g(x_{1},\dots,x_{d})=f(x_{1},\dots,x_{d},c)$ for some constant
  $c\in A$. Let $r_{1},\dots,r_{d}\in\rho$ ($r_{i}=(a_{i},a'_{i})$ for
  $1\leq i\leq d$) and $\bc:=(c,c)$. Then
  $a_{1}+\ldots+a_{d}+c\leq a'_{1}+\ldots+a'_{d}+c\leq
  a_{1}+\ldots+a_{d}+c+d$ (note $a_{i}\leq a_{i}'\leq a_{i}+1$ by
  definition of $\rho$). Therefore, by definition of $f$,
  $f(a_{1},\dots,a_{d},c)\leq f(a'_{1},\dots,a'_{d},c)\leq
  f(a_{1},\dots,a_{d},c)+1$ which gives
  \begin{align*}
    f(r_{1},\dots,r_{d},\bc)=
    (f(a_{1},\dots,a_{d},c),f(a'_{1},\dots,a'_{d},c))\in\rho.
  \end{align*}
  Thus $g\preserves\rho$. Because of
  Lemma~\ref{lem:trlApreserves} this implies that each $d$-translation of $f$
  preserves $\rho$, i.e.,  $\trl_{d}(f)\subseteq M$. Consequently we
  have $f\in M^{*_{d}}$. 
  But $g:A\to A$ given by $g(x)=f(x,\dots,x)$ (belonging to the clone
  generated by $f$) does not preserve $\rho$ since
  $g(r)=f(r,\dots,r)=(0,2)\notin\rho$ for $r=(1,2)\in\rho$. Thus
  $f\in M^{*_d}$ but $g\notin M^{*_d}$. Hence $M^{*_d}$ is not a clone
  (and therefore $\rho$ does not satisfy $\Xi_{d}$). Since $d$ was
  chosen arbitrarily, this holds for every $d\in \N_{+}$.

\end{example}

Since $M^{*_d}$ is not always a clone, there also arises the question:
what is the algebraic nature of the sets $M^{*_d}$? The following proposition gives the answer.

\begin{proposition}\label{M6}
  Let $M\leq \Opa[d](A)$ be a clone fragment and $\cC^{(d)}\subseteq M$. Then $M^{*_d}$ is a preclone
  (cf.~{\rm\ref{M5}}).
\end{proposition}

\begin{proof}[Scetch of the proof] It is easy to see that $M^{*_{d}}$ satisfies
  \ref{M5}\ref{M5-1}-\ref{M5-3}. In fact, the $d$-translations of
  $\id_{A}$ are the $d$-ary projections and the constants $\cC^{d}$ and therefore
  $\id_{A}\in M^{*_{d}}=\{f\in\Op(A)\mid \trl_{d}(f)\subseteq M\}$
  since $M$ by assumption
  contains all projections (as clone fragment) and $\cC^{d}$.

  Further, if $f$ and $f'$
  differ only by permutation of variables, then any function belonging to
  $\trl_{d}(f)$ can be obtained from a function belonging to $\trl_{d}(f')$ by a permutation of
  variables, and vice-versa. Thus $f\in M^{*_{d}}$ iff $f'\in M^{*_{d}}$, since $M$ is a
  clone fragment and is closed under permutation of variables (in
  particular under $\zeta$ and $\tau$).

It remains to show $\ref{M5}\ref{M5-4}$.
With the notation from \ref{M5} let 
\begin{align*}
  h(x_{0},\dots,x_{m-1},x_{m},\dots,x_{m+n-2})
        :=f(g(x_{0},\dots,x_{m-1}),x_{m},\dots,x_{m+n-2})
\end{align*}
and $f,g\in M^{*_{d}}$. One has to show $h\in M^{*_{d}}$, i.e.,
$\trl_{d}(h)\subseteq M$. Taking $h'\in\trl_{d}(h)$ and using the
above presentation of $h$ and the definition of $d$-translations, it
is technical but not hard to 
find $f'\in\trl_{d}(f)$ and $g'\in\trl_{d}(g)$ such that $h'$ is a
composition of $g'$ and $f'$: $h'\in\Sg{g',f'}$. Consequently,
$h'\in M$ since $\trl_{d}(f)$ and $\trl_{d}(g)$ are contained in $M$
by assumption and $M$ is (as clone fragment) closed under composition.
We do not go further into these technical details because the
Proposition shall not be used further.
\end{proof}

\normalsize
%

We already know that $M^{*_{d}}$ is a clone for a clone fragment
$M=\Pol\rho$ whenever $\rho$ satisfies $\Xi_{d}$. Are there more such
clone fragments? Thus there arises the question:
\begin{center}
  \emph{For which clone fragments $M\leq \Opa[d](A)$
      the preclone $M^{*_{d}}$ is a clone?}
\end{center}
The answer shall be given in Proposition~\ref{prop:StarisCloneIffGQuord}
via generalized quasiorders
and also in Corollary~\ref{A3cor} in terms of the $\du$-closure 
with respect to
the following closure operator $M\mapsto\duclose[d] M$.

\begin{definition}\label{def:Rclosure} For $M\subseteq A^{A^{d}}$ let
  \begin{align*}
    \Rclose{M}:=\bigcap\{N\mid M\cup \cC\subseteq N \leq A^{A^{d}},\text{ and
    $N^{*_{d}}$ is a clone}\}.
  \end{align*}
  A clone fragment $M\leq A^{A^{n}}$ is called \New{\dR-closed} if
  $\Rclose{M}=M$.

\end{definition}

\begin{remark}\label{rem:Rclosure} 
The operator $M\mapsto\duclose{M}$ (for $M\subseteq \Opa[d](A)$) is a
closure operator. $\duclose{M}$ is the least $d$-ary clone fragment $N$
containing $M$ and all constants $\cC^{(d)}$ such that $N^{*_{d}}$ is
a clone (it follows from Remark~\ref{rem:MstarPoe}(A) that
$(\duclose{M})^{*_{d}}$ is a clone).
By definition we have: \textit{a clone fragment $M$ is \du-closed if
  and only if the preclone $M^{*_{d}}$ is a clone.}




\end{remark}

  \section{R-relations and arities}\label{sec:Rrelations}
  
  In the previous section we defined a higher arity generalization of
  a translation of an operation $f \in \Op(A)$ and then defined the
  property $\Xi_d$ which we wish to understand. It turns out that in
  order to understand the $\Xi_d$ property, it is helpful to view
  relations as sets of multisorted operations. This perspective allows
  us to view both clone fragments and relations from a single vantage
  point, where the only key distinction between the two kinds of
  relations is the operations in a clone compose, while our
  interpretation of relations as sets of multisorted operations does
  not in general permit composition.

  Viewing a clone fragment as a relation is not a new idea. Indeed, it
  is often the case that a set of $n$-ary operations
  $X \subseteq A^{A^n}$ is considered as a relation of arity $|A^n|$,
  i.e.\ the tuples of $A^n$ are ordered and then treated as
  coordinates. The interpretation of a set of operations as a relation
  underlies the most basic aspects of the theory.
   Clearly, for any finite set $\alpha$ (the unusual notation $\alpha$
   will be useful later) 
  we may consider relations as
  subsets $\rho \subseteq A^{\alpha}$, with the understanding that if $\rho$
  needs to be turned into a `true' relation, we can order $\alpha$ with a
  bijection $\eta: |\alpha| \to \alpha$ and then consider instead the relation
  \[
    \rho\circ\eta = \{ g \circ \eta\mid g \in \rho \} \subseteq
    A^{|\alpha|}.
  \]
  Although subsets of $A^{\alpha}$ contain exactly the same information as
  subsets of $A^{|\alpha|}$, there can be structure in the set $\alpha$ that is
  missing from the ordinal $|\alpha|$. To give a specific example, the set
  $2^2$ is from a geometrical viewpoint a square, and hence a
  $2$-dimensional object, while its cardinality $4$ is from a
  geometrical viewpoint a line, and hence a $1$-dimensional
  object. The results we present here rely critically on this
  rectangular structure, since all of our definitions make reference
  to it.

  Another basic and important example is the one given earlier, where
  we consider some set $M \subseteq A^{A^n}$ of $n$-ary operations on
  $A$. Note that there are three notions of arity at play here. On the
  one hand, every $f \in M$ is a function of arity $n$, while on the
  other hand, $M$ can be naturally considered as a $|A^n|$-ary
  relation. We distinguish these two arities, calling the former type
  a \emph{functional arity} and the latter type a \emph{relational
    arity}. Since formally $A^n$ and $|A^n|$ are distinct, we call
  $A^n$ a \emph{rectangular arity}. These ideas can be extended as
  follows.

\begin{definition}\label{def:Rrelation}
  Let $s \geq 1$, let $A$ and $B_1, \dots, B_{s}$ be nonempty sets,
  and let $d_1, \dots, d_{s} \geq 0$. We call any
  \[
    \rho \subseteq A^{B_1^{d_1} \times \dots \times B_{s}^{d_{s}}}
  \]
  an \emph{R-relation}. In this situation, we say that $\rho$ has
  \emph{rectangular arity} (for short \emph{arity})
  $\alpha := B_1^{d_1} \times \dots \times B_{s}^{d_{s}}$,
  \emph{relational arity}
  $m := |B_1^{d_1} \times \dots \times B_{s}^{d_{s}}|$, and
  \emph{functional arity} $(d_1, \dots, d_{s})$. We denote by
  $d_\alpha := d_1 + \dots + d_{s}$ the \emph{dimension} of $\rho$. A
  rectangular arity of the form $B^d$ is called a \emph{simple}
  rectangular arity, while all other rectangular arities are called
  \emph{compound} rectangular arities.

The set of all R-relations of (rectangular) arity $\alpha$ on base set
$A$ is denoted by $\RRela[\alpha](A)$. 
\end{definition}

\begin{remark}\label{rem:Rrelation}

Let us make some remarks about
Definition~\ref{def:Rrelation}. First, we note that for a given
relation $\sigma \subseteq A^m$ of arity $m$, there are potentially many
formally distinct corresponding R-relations $\rho$ with some
rectangular arity $B_1^{d_1}\times \dots \times B_{s}^{d_{s}}$, such
that $\sigma = \rho\circ\eta$  when precomposing the elements of $\rho$
with a bijection $\eta$ from the relational arity
$m=|B_1^{d_1}\times \dots \times B_{s}^{d_{s}}|$ to the rectangular
arity $B_1^{d_1}\times \dots \times B_{s}^{d_{s}}$. 

However, the
collection of all relations of arity $m$ and the collection of all
R-relations with rectangular arity
$B_1^{d_1}\times \dots \times B_{s}^{d_{s}}$ contain exactly the same
information from the perspective of polymorphisms. 

Note that there is an obvious bijection between
R-relations of rectangular arity
$\alpha=B_{1}^{d_{1}}\times\ldots\times B_{s}^{d_{s}}$ and arity
$\tilde\alpha=\tilde B_{1}^{d_{1}}\times\ldots\times \tilde
B_{s}^{d_{s}}$ as long as the cardinalities $|B_{i}|=|\tilde B_{i}|$
agree ($i\in\{1,\dots,s\}$).

Hence, the only distinction between the usual notion of a relation and
an R-relation is the manner in which relations are represented and, as
we mentioned earlier, the representation of relations as R-relations
is the foundation on which our definitions rest. In fact, since
different rectangular arities can possess the same relational arity,
we will see that our perspective allows for the isolation of distinct
sets of relations which all possess the same relational arity.

We also remark that Definition~\ref{def:Rrelation} is designed
to give us access to standard function notation when dealing with
elements of an R-relation. Indeed, every R-relation can be treated
as a set of ``multisorted'' and multiple arity operations.



Specifically,
given an R-relation of rectangular arity
\[\alpha = \alpha_{1}\times\ldots\times\alpha_{s}=B_1^{d_1} \times \dots \times B_{s}^{d_{s}}\] and
$g\in \rho$, we can access the value of $g$ at a particular input with
the notation
\begin{align*}
 g(X^{(\alpha)}) \text{ or, more explicit, }
  g(X^{(\alpha_{1})},\dots,X^{(\alpha_{s})}), 
\end{align*}
where $X^{(\alpha)}=(X^{(\alpha_{1})},\dots,X^{(\alpha_{s})})$ and
   $X^{(\alpha_{i})}=(x_{0}^{(\alpha_{i})},\dots,x_{d_{i}-1}^{(\alpha_{i})})\in
   \alpha_{i}=B_{i}^{d_{i}}$, 
  ($i\in\{1,\dots,s\}$).
\end{remark}

\begin{example}\label{RR3} Let
$\rho\in\RRela[\alpha](A)$ with
$\alpha=\alpha_{1}\times\alpha_{2}=B_{1}^{1}\times
B_{2}^{2}=\{0,1,2\}^{1}\times\{0,1\}^{2}$. Thus $d_{1}=1$, $d_{2}=2$.
 Here
(as in further examples) we take ordinals for $B_{i}$: $B_{1}=3=\{0,1,2\}$ and $B_{2}=2=\{0,1\}$.
 The
dimension is
$d_{\alpha}=d_{1}+d_{2}=3$, the relational arity is
$|\alpha|=3^{1}\cdot 2^{2}=12$. 

An element  of $\rho$ is a mapping $g:\alpha\to A: X^{(\alpha)}\mapsto
g(X^{(\alpha)})$ where $X^{(\alpha)}=(X^{(\alpha_{1})},X^{(\alpha_{2})})
=(x_{0}^{(\alpha_{1})},x_{0}^{(\alpha_{2})},x_{1}^{(\alpha_{2})})\in\alpha$
and can be represented as ``rectangular cube'' as follows:

\nopagebreak
    {\hspace*{10ex}\includegraphics{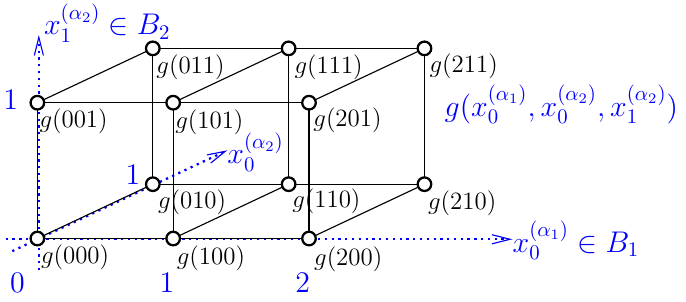}}

\end{example}

Using the above notation in \ref{rem:Rrelation}, we can give the
following definition: 

\begin{definition}\label{def:depend}
We say that $g$ \New{depends on $\alpha_j$} ($j\in\{1,\dots,s\}$) if
there exist  
$X^{(\alpha_{1})},\dots,X^{(\alpha_{s})}$ and $Y^{(\alpha_{j})}\in
\alpha_{j}=B_{j}^{d_{j}}$
such that 
\begin{align*}
g(X^{(\alpha_{1})},\dots,&X^{(\alpha_{j-1})},
X^{(\alpha_{j})},X^{(\alpha_{j+1})},\dots,X^{(\alpha_{s})})\\
&\neq
g(X^{(\alpha_{1})},\dots,X^{(\alpha_{j-1})},
Y^{(\alpha_{j})},X^{(\alpha_{j+1})},\dots,X^{(\alpha_{s})}).
\end{align*}
If $g$ does not depend on $\alpha_{j}$, then we also say that 
\Def{$g$ is constant at $\alpha_{j}$}, i.e., for any constant tuple
$C^{(\alpha_{j})}\in\alpha_{j}$ we have
\begin{align*}
g(X^{(\alpha_{1})},\dots,X^{(\alpha_{s})})=
g(X^{(\alpha_{1})},\dots,X^{(\alpha_{j-1})},
C^{(\alpha_{j})},X^{(\alpha_{j+1})},\dots,X^{(\alpha_{s})}).
\end{align*}
for all $X^{(\alpha_{1})},\dots,X^{(\alpha_{s})} $.
\end{definition}

Generalizing $\rho\circ\eta$ introduced above, we can
``move'' between R-relations of different rectangular arities as
follows:

\begin{definition}\label{def:aritymap}
Let $\mu:\beta\to\alpha$ be a mapping between rectangular arities
$\alpha$ and $\beta$. Then we have a mapping
$\RRela[\alpha](A)\to\RRela[\beta](A):\rho\mapsto \rho\circ\mu$ given
by precomposing $\mu$ to the elements $g:\alpha\to A$ of
$\rho\subseteq A^{\alpha}$:
\begin{align*}
  \rho\circ\mu:=\{g\circ\mu\mid g\in\rho\}\subseteq A^{\beta}.
\end{align*}
We call $\rho\circ\mu$ an \Def{arity transformation of $\rho$}.
\end{definition}

As mentioned above, R-relations $\rho\subseteq A^{\alpha}$ (with
relational arity 
$\alpha=B_{1}^{d_{1}}\times\ldots\times B_{s}^{d_{s}}$) can be viewed
as usual relations $\rho\circ\eta\subseteq A^{m}$ with
$m=|\alpha|$. Therefore such notions like preservation $\preserves$,
polymorphism, invariant relation immediately can be applied to
R-relations, too ($\Pol\rho:=\Pol(\rho\circ\eta)$).
Nevertheless, for completeness we give here the full definitions.

\begin{definition} \label{def:preservation}
Let $f:A^{n}\to A$ be an $n$-ary operation on $A$ and $\rho\subseteq
A^{\alpha}$ be an R-relation of (rectangular) arity $\alpha$. We say
that \New{$f$ preserves $\rho$} (notation: $f\preserves\rho$) or $f$
is a \New{polymorphism} of $\rho$, or $\rho$ is \New{invariant} for
$f$, if $f$ maps elements of $\rho$ (which are mappings $g:\alpha\to A$) into elements of $\rho$, more
precisely, if
\begin{align*}
  \forall g_{1},\dots,g_{n}\in\rho: f(g_{1},\dots,g_{n})\in\rho,
\end{align*}
where $f(g_{1},\dots,g_{n}):\alpha\to A$ is defined componentwise as
\begin{align*}
  f(g_{1},\dots,g_{n})(b):=f(g_{1}(b),\dots,g_{n}(b))\text{ for each $b\in\alpha$}.
\end{align*}
\end{definition}

\section{Translations for R-relations of simple arity}\label{sec:trl}

In Section~\ref{sec:Xi}, we defined translations only for operations
$f \in A^{A^d}$, for some nonempty $A$ and $d \geq 0$, in order to
take the clearest path to defining our motivating property $\Xi_d$. In
this section we make the straightforward generalization to functions
$g \in A^{B^d}$, for a nonempty $B$ which is potentially distinct from
$A$. Hence, we obtain a notion of translation which applies to any
R-relation of simple arity. The simplest kinds of translations are among
those that increase or decrease the arity of the function by one,
either by the addition of an initial or terminal fictitious variable, or by
substituting a single variable for a constant. We begin by giving
names to these kinds of translations.

\begin{definition}[elementary basic translations]\label{def:ebtrl1}

  Let $d\in\N$. For a $(d+1)$-ary function $g:B^{d+1}\to A$, a $d$-ary
  function $g':B^{d}\to A$ is called an \Def{elementary basic
    translation of $g$} if $g'$ is obtained from $g$ by substituting a
  single constant from $B$ for one of its variables and filling up the
  (non-substituted) places by the variables
  $x_{0},x_{1},\dots,x_{d-1}$, i.e.,\ $g'$ is of the form
  \begin{align*}\label{ebtrl1}\tag*{\ref{def:ebtrl1}(1)}
    g'(x_{0},\dots,x_{d-1}):=g(x_0, \dots, x_{i-1}, c, x_{i}, \dots, x_{d-1})
  \end{align*}
  for some $0 \leq i \leq d$.
  
  For $g \in A^{B^{d+1}}$, we denote by $\ebtrl(g)$ the set of all
  elementary basic translations of $g$. We use the same notation for
  sets of functions $F \subseteq A^{B^{d+1}}$ and denote by
  $\ebtrl(F)$ the union of all $\ebtrl(g)$ for $g \in F$.
 If $g$ is nullary, then we
formally define $\ebtrl(g):=\emptyset$ (there is nothing to substitute).

\end{definition}

\begin{definition}[basic translations]\label{def:btrl2}

  Let $g\in A^{B^{n}}$ ($n\in\N$) and $F \subseteq \bigcup_{n \in \N} A^{B^n}$. 
We define the set of
  \Def{basic translations} of $g$ and $F$ as
  \begin{align*}
    \btrl(g):=\bigcup_{i=0}^{n} \ebtrl^i(g), &&\btrl(F) :=
    \bigcup_{g\in F}\btrl(g),
  \end{align*}
where $\ebtrl^{0}(g):=\{g\}$ and
$\ebtrl^{i+1}(g):=\ebtrl(\ebtrl^{i}(g))$ for $i\in\{0,1,\dots,n-1\}$.

  The collection of $d$-ary functions belonging to $\btrl(F)$ is
  denoted by $\btrl_d(F)$. Obviously we have
  $\btrl_{d}(g)=\ebtrl^{n-d}(g)$ if $d\leq n$.

  \begin{remarknote}
    Note that this definition of $\btrl(F)$ is equivalent to the one
    given in Definition~\ref{trlA-1} in the case when $B=A$.
  \end{remarknote}
\end{definition}

\begin{definition}[elementary fictitious
  translations] \label{def:eftrl1}
  Let $d \geq 0$. For an $d$-ary function $g: B^{d} \to A$, a
  $(d+1)$-ary function $h: B^{d+1}\to A$ is called an \emph{elementary
    fictitious $d$-translation} of $g$, if $h$ is obtained from $g$ by
  prepending or appending to the variables of $g$ a single dummy
  variable, i.e., $h=\eftrl^{1|0}(g)$ or $h=\eftrl^{0|1}(g)$, where
  \begin{align*}\label{eftrl1}
    \eftrl^{1|0}(g)(x_0, \dots, x_{d}) &:= g(x_{1}, \dots, x_{d}), 
\tag*{\ref{def:eftrl1}(1)}\\     \label{eftrl2} 	
\eftrl^{0|1}(g)(x_0, \dots, x_{d})&:=g(x_{0}, \dots, x_{d-1}). \tag*{\ref{def:eftrl1}(2)}
  \end{align*}

  For $g \in A^{B^{d+1}}$, let
  $\eftrl(g):=\{\eftrl^{1|0}(g),\eftrl^{0|1}(g)\}$ be the set  of all
  elementary fictitious translations of $g$. We use the same notation for
  sets of functions $F \subseteq A^{B^{d+1}}$ and denote by
  $\eftrl(F)$ the union of all $\eftrl(g)$ for $g \in F$.



\end{definition}

\begin{definition}[fictitious translations] \label{def:ftrl2}

  Let $g\in A^{B^{n}}$ ($n\in\N$) and $F \subseteq \bigcup_{n \in \N} A^{B^n}$. We define the set of
  \Def{fictitious translations} of $g$ and $F$ as
  \begin{align*}
    \ftrl(g):=\bigcup_{i\geq 0} \eftrl^i(g), &&\ftrl(F) :=
    \bigcup_{g\in F}\ftrl(g),
  \end{align*}
where $\eftrl^{0}(g):=\{g\}$ and
$\eftrl^{i+1}(g):=\eftrl(\eftrl^{i}(g))$ for $i\in\N$.
  The collection of $d$-ary functions belonging to $\ftrl(F)$ is
  denoted by $\ftrl_d(F)$. Obviously we have
  $\ftrl_{d}(g)=\eftrl^{d-n}(g)$ if $n\leq d$.

  \begin{remarknote}
    Note that this definition of $\ftrl(F)$ is equivalent to the one
    given in Definition~\ref{trlA-2} in the case when $B=A$.
  \end{remarknote}
\end{definition}

\begin{definition}[translations]\label{def:trl1}

  For $F\subseteq \bigcup_{n \in \N} A^{B^n}$ let $\trl(F)$ be the
  least subset of $\bigcup_{n \in \N} A^{B^n}$ which is closed under
  $\btrl$ and $\ftrl$ and contains $F$ (thus we have
  $F\subseteq \trl(F)=\btrl(\trl(F))=\ftrl(\trl(F))$\,). The functions
  in $\trl(F)$ are called \Def{translations of $F$}. Moreover, for
  $d\in \N$, let $\trl_{d}(F):=\trl(F)\cap A^{B^{d}}$, its elements
  are called \Def{$d$-translations of $F$}.
\end{definition}

The following lemma shows how $d$-translations can be (nearly
uniquely) decomposed into
a sequence of elementary basic $d$-translations followed by a sequence
of elementary fictitious $d$-translations. 

For prepending $u$ and appending $v$ fictitious variables we use
the following notation:
$\eftrl^{u|v}(g):=(\eftrl^{1|0})^{u}(\eftrl^{0|1})^{v}(g)$ (cf.~\ref{def:eftrl1}),
$u,v\in\N$, in particular we have $\eftrl^{0|0}(g)=\{g\}$.



\begin{lemma}\label{lem:trldecomposition}
Let $g\in A^{B^{n}}$, $n,d\in \N$ and $\bar
d:=\min\{d,n\}$.
\begin{enumerate}[label=\textup{(\roman*)}]

\item\label{trldecomposition-i} $\ebtrl(\eftrl(g))=\{g\}\cup \eftrl(\ebtrl(g))$, i.e.,  
      $\ebtrl\circ \eftrl =\id\cup\eftrl\circ\ebtrl$.

\item \label{trldecomposition-ii}
For each $g'\in\trl_{d}(g)$, we have
$g'\in\eftrl^{u|v}(\ebtrl^{n-t}(g))\subseteq\eftrl^{d-t}(\ebtrl^{n-t}(g))$ for some $u,v,t\in\N$ with
$t\in\{0,\dots,\bar d\}$ and $u+v=d-t$. In particular, $g'$ has the form
  \begin{align*}\label{trl-1}\tag*{\ref{lem:trldecomposition}(1)}
    g'(x_{0},\dots,x_{d-1}):=g(\underbracket{\dots}_{c\in B},x_{u},\underbracket{\dots}_{c\in B},x_{u+1},\underbracket{\dots}_{c\in B},\dots,x_{u+t-1},\underbracket{\dots}_{c\in B})
  \end{align*}
(i.e., the first $u$ and the last $v$ variables of $g'$ are fictitious,
altogether there are $n-t$ places substituted in $g$ by constants; for $t=0$, all places in $g$ are substituted by constants and all
variables of $g'$ are fictitious).

\item\label{trldecomposition-iii}
     $\trl_{d}(g)
     =\displaystyle\bigcup_{t=0}^{\bar d}\eftrl^{d-t}(\ebtrl^{n-t}(g))
      =\ftrl_{d}(\btrl(g))$.

\end{enumerate}
\end{lemma}

\begin{proof}\ref{trldecomposition-i}:
  Let $g'\in\ebtrl(\eftrl(g))$. Then $g'$ is obtained from $g$ by adding
  a fictitious variable (prepending or appending, respectively) and
  then substituting a variable by a 
  constant. If the added variable is substituted by the constant, then
  we get $g$ back, i.e., $g'=g$. Otherwise, one can also first 
  substitute the constant (at the appropriate place) and then add the
  fictitious variable (prepending or appending, respectively), i.e.,
  $g'\in\eftrl(\ebtrl(g))$. In particular, substituting constants and
  adding fictitious variables commute in that case, i.e.,
  $\eftrl(\ebtrl(g))\subseteq\ebtrl(\eftrl(g))$.

\ref{trldecomposition-ii}: By definition, a translation of $g$ can be
obtained by consecutive application of $\ebtrl$ and $\eftrl$. Therefore the
statement directly follows from $\ref{trldecomposition-i}$ since one
can apply first a number of $\btrl$, say $n-t$ times (i.e., evaluating
$n-t$ variables of  $g$ at a constant, $0\leq t\leq n$) and then
adding fictitious variables, say $u$ prepending and $v$
appending. Since the result $g'$ 
is to be $d$-ary, the constraints $t\leq d$ (thus $t\leq \bar d$) and $n-(n-t)+u+v=d$
(equivalently $u+v=d-t$) follow.
\ref{trl-1} expresses this situation.

\ref{trldecomposition-iii} directly follows from 
\ref{trldecomposition-ii} and the definitions.
\end{proof}

Note that the last equation in \ref{trldecomposition-iii} shows that
Definition~\ref{def:trlA} of $\trl_{d}$ coincides with (the seemingly
different) 
Definition~\ref{def:trl1} for $B=A$.

\begin{remark}\label{mu-translation}

There is a further tool for the representation of translations.
Let $\alpha:=B^{n}$ and $\beta:=B^{d}$. The equation \ref{trl-1} can be interpreted as a mapping
$A^{\alpha}\to A^{\beta}$ with assigns to each $g\in A^{\alpha}$ a
specific $d$-ary translation $g'\in A^{\beta}$. This mapping can be
described uniquely by a mapping $\mu:\beta\to\alpha$ such that 
$g'=g\circ\mu$ (i.e., as a special arity transformation, notation see Definition~\ref{def:aritymap}).
Namely, for $X=(x_{0},\dots,x_{d-1})\in \beta$ we
put 
\[\mu(X):=(\underbracket{\dots}_{c\in B},x_{u},
\underbracket{\dots}_{c\in B},x_{u+1},
\underbracket{\dots}_{c\in B},\dots,x_{u+t-1},
\underbracket{\dots}_{c\in B})\in\alpha,\] 
i.e., the argument of $g'$ of the left side of equation \ref{trl-1}
is mapped to the argument of
$g$ at the right side, thus $g'(X)=g(\mu(X))$. To give a precise
definition of translation in this way, it is enough to define $\mu$
for elementary translations 
(since each translation is a composition of elementary ones).

According to \ref{ebtrl1} an elementary basic translation $g'$ of $g$ is given by
$g'(X)=g(\mu(X))$ where $\mu$ is the mapping
\begin{align*}
  \mu:B^{d}\to B^{d+1}:(x_{0},\dots,x_{d-1})\mapsto(x_0, \dots, x_{i-1}, c, x_{i}, \dots, x_{d-1}),
\end{align*}
i.e., adding the constant $c$ at a new(!) place before $x_{i}$
  (in case that the last variable is evaluated at a constant in $g$,
  the new place is after $x_{d-1}$).

Analogously, according to \ref{eftrl1} and \ref{eftrl2}, an elementary
fictitious translation $g'$ of $g$ is given by $\eftrl^{1|0}(X)=g(\nu_{1|0}(X))$
or $\eftrl^{0|1}(X)=g(\nu_{0|1}(X))$ with the mappings
\begin{align*}
  &\nu_{1|0}:B^{d+1}\to B^{d}:(x_0, \dots, x_{d})\mapsto
                   (x_{1}, \dots,x_{d}),\\
    &\nu_{0|1}:B^{d+1}\to B^{d}:(x_0, \dots, x_{d})\mapsto
                   (x_{0}, \dots, x_{d-1}),
\end{align*}
i.e., deleting the first or last component, respectively.

\end{remark}




\section{Adjoints and higher-dimensional generalized
  quasiorders of simple arity}\label{sec:GQuord}

We wish to generalize the properties of transitivity and reflexivity
from the context of binary relations to arbitrary R-relations. As
noted in the introduction, this has already been achieved in two (in
some sense orthogonal) directions: in \cite{JakPR2024} the right
generalization of these properties was discovered for what we now
understand to be $1$-dimensional R-relations, i.e.\ those relations
with rectangular arity $\alpha = m^1$ for some finite ordinal $m$,
while in \cite{Moo2024} the right generalization was found for what
we now understand to be R-relations with rectangular arity
$\alpha = \underbracket{2^1 \times \dots \times 2^1}_n$ (which is not
to be confused with rectangular arity $2^n$).

As mentioned, we limit our exposition to R-relations of simple arity at
first. 

\begin{definition}[The monotone Galois connection $\ebtrl - \ostar$]
\label{def:ostar}
  Let $d\in\N$. According to \ref{Galoisconnections}, the mapping
  $\ebtrl:A^{B^{d+1}}\to \Pow(A^{B^{d}}):h\mapsto\ebtrl(h)$ induces a
  monotone Galois connection, whose operators are denoted by 
  \begin{align*}
    \ebtrl(\sigma)&:=\bigcup_{h \in \sigma} \ebtrl(h)
      \text{ for $\sigma\subseteq A^{B^{d+1}}$ and}\\
    \ostar(\rho)&:=\{ h\in A^{B^{d+1}}\mid \ebtrl(h) \subseteq
       \rho\} \text{ for $\rho\subseteq A^{B^{d}}$.}
  \end{align*}
 i.e., we have
  \begin{align*}
  \ebtrl: &\Pow(A^{B^{d+1}}) \to \Pow(A^{B^{d}}):\sigma\mapsto\ebtrl(\sigma)\\
\ostar: &\Pow(A^{B^{d}}) \to \Pow(A^{B^{d+1}}): \rho\mapsto \ostar(\rho),
  \end{align*}
where $\ostar$ is the upper and $\ebtrl$ the lower adjoint, i.e., for
$\sigma\subseteq  A^{B^{d+1}}$ and $\rho\subseteq A^{B^{d}}$ we have (cf.\ \ref{Galoisconnections}):
\begin{align*}\label{ostar-1}\tag*{\ref{def:ostar}(1)}
\ebtrl(\sigma)\subseteq\rho&\iff \sigma\subseteq\ostar(\rho), 
  \text{ in particular,}\\
\label{ostar-2}\tag*{\ref{def:ostar}(2)}
  \ebtrl(h)\subseteq\rho&\iff h\in\ostar(\rho) \text{ (for $h\in
  A^{B^{d+1}}$).}
\end{align*}
\end{definition}

 Note
that we have defined an infinite collection of such monotone Galois
connections, one for each $d \geq 0$. Since their essential properties
do not rely explicitly on the choice of $d$, we do not distinguish
these functions with indices, and instead rely on the context in which
they are used to specify their domains.

For our definitions of reflexivity and transitivity for an R-relation
$\rho \subseteq A^{B^d}$ of rectangular arity $\alpha = B^d$ we need
two further mappings, one increasing the arity the other decreasing
the arity by $1$. 

\begin{definition}\label{def:Delta} Let $d\in\N$. 
For  $h \in A^{B^{d+1}}$, let
\[
  \bDelta(h) := \{ h(x_0, \dots, x_i, \dots, x_{j-1}, x_i,
  x_j, \dots, x_{d-1})\mid  0 \leq i < j \leq d \},
\]
denote the set of all functions obtained from $h$ by identifying
exactly two variables (there should be
    no confusion between $\bDelta$ and the operator $\Delta$
    introduced in Definition \ref{M5}\ref{M5-6}). Note that
    $\bDelta(h)=\emptyset$ for unary $h$ (in case $d=0$).
We get the following two operators induced by
$\eftrl$ (cf.\ Definition~\ref{def:eftrl1}) and $\bDelta$, which we denote with the same symbol.
\begin{align*}
   \eftrl:& \Pow(A^{B^{d}}) \to \Pow(A^{B^{d+1}}):
    \rho \mapsto \bigcup_{g \in \rho} \eftrl(g),\\
   \bDelta:& \Pow(A^{B^{d+1}}) \to \Pow(A^{B^d}):
  \sigma \mapsto \bigcup_{h \in \sigma} \bDelta(h).
\end{align*}

\end{definition}

Note that by definition all operators $\ebtrl,\ostar,\eftrl,\bDelta$
defined so far are monotone with respect to inclusion.

\begin{definition}\label{Adef:Refl+Trans}
  Let $\rho \subseteq A^{B^d}$ be an R-relation of arity
  $\alpha = B^d$. We say that $\rho$ is
  \begin{enumerate}[label=\textup{(\roman*)}]
  \item \label{Adef:internally-reflexive}\emph{internally reflexive}
    if $\eftrl(\rho) \subseteq \ostar(\rho) $,

  \item \label{Adef:transitive} \emph{transitive} if
    $\mathbf{\Delta}(\ostar(\rho)) \subseteq \rho$,

  \item \label{Adef:weakSimpleQuord} a \emph{weak $d$-dimensional
      generalized quasiorder (of simple rectangular arity)} if $\rho$
    is internally reflexive and transitive,

  \item \label{Adef:reflexive}\emph{reflexive} if it is internally
    reflexive and additionally contains every constant function
    $c_{a}^{(d)}:(x_0, \dots, x_{d-1})\mapsto a$ for $a \in A$, 

  \item \label{Adef:SimpleQuord} a \emph{$d$-dimensional
      generalized quasiorder (of simple rectangular arity)} if $\rho$
    is reflexive and transitive.
  \end{enumerate}
\end{definition}

\begin{remark}\label{rm:equivalentreflexivity}
  Internal reflexivity of a relation $\rho \subseteq A^{B^d}$ is
  equivalently defined by each of the following conditions:

  \begin{enumerate}
  \item[(i')]$\ebtrl(\eftrl(\rho)) \subseteq \rho$,
  \item[(i'')]$\eftrl(\ebtrl(\rho)) \subseteq \rho$ (provided $d \geq
    1$).
  \end{enumerate}
Indeed, (i') is equivalent to
\ref{Adef:Refl+Trans}\ref{Adef:internally-reflexive}
  because $\ebtrl$ is the lower adjoint of $\ostar$. The equivalence
  of (i') and (i'') directly follows from
\ref{lem:trldecomposition}\ref{trldecomposition-i}:
   $\ebtrl(\eftrl(\rho))=\rho\cup\eftrl(\ebtrl(\rho))$.

Moreover, in \ref{lem:adjointproperty}\ref{adj-ii} we shall see that $\rho$ is
internally reflexive if and only if $\rho=\trl_{d}(\rho)$.

Further we mention 
that an internally reflexive relation
$\rho\subseteq A^{\alpha}$ ($\alpha=B^{d}$) contains all constant maps 
$g_{a}:(x_{0},\dots,x_{d-1})\mapsto a$ for $a\in \Ima\rho$ (notation see
\ref{def:ImageRrelation}) because $\Ima\rho=\ebtrl^{d}(\rho)$ and
since $g_{a}\in\eftrl^{d}(\ebtrl^{d}(\rho))\subseteq\rho$ 
(for $a\in\Ima\rho$) by (i') and (ii').
Therefore $\cC_{\Ima\rho}\subseteq\Pol\rho$, and an internally reflexive $\rho$ is reflexive if and only
if $\Ima\rho=A$. Consequently, a relation $\rho$ is a generalized
quasiorder if and only if it is a weak generalized quasiorder and
$\Ima\rho=A$. Unary relations are always weak generalized quasiorders
(transitivity becomes trivial).

Note that an internally reflexive relation
becomes reflexive if we
change the underlying base set, namely if we
consider $\rho$ as a relation $\rho\subseteq (\Ima\rho)^{\alpha}$ on
the base set $\Ima\rho$.
 \end{remark}

\begin{example}\label{ex:reftra} Let $\rho\subseteq A^{\alpha}$ be an
  R-relation of simple arity. For three examples, namely $\alpha=3^{1}$, 
  $\alpha=2^{2}$ and $\alpha=2^{3}$, we show how internal reflexivity
  and transitivity of $\rho$ can be understood geometrically. 

At first
  an element $g\in\rho$ is schematically given (the mapping $g:\alpha\to A$ is
  represented as an object in a $d_{\alpha}$-dimensional space, cf.\ 
  Example~\ref{RR3}).
 
  Then a typical elementary basic translation $g'$ of $g$ is choosen and
  its elementary ficticious translations
  $\eftrl(g')=\{\eftrl^{1|0}(g'),\eftrl^{0|1}(g')\}$ (cf.\
  Definition~\ref{def:eftrl1}) are presented. Internal reflexivity means that
  $g'\in\ebtrl(\rho)$ implies $\eftrl(g')\subseteq\rho$.

Finally (except for $\alpha=2^{3}$) a
  typical element $h\in\ostar(\rho)$ is shown and we explain what its
  elementary basic translations are and which elements belong to
  $\bDelta(h)$.  Transitivity means that
  $\ebtrl(h)\subseteq\rho$ implies $\bDelta(h)\subseteq\rho$.

 The labels (image values) for $g$, $g'$ and $h$ are clear from the 
  presentation of functions and therefore mostly are deleted.

\begin{figure}
\begin{center}
\includegraphics{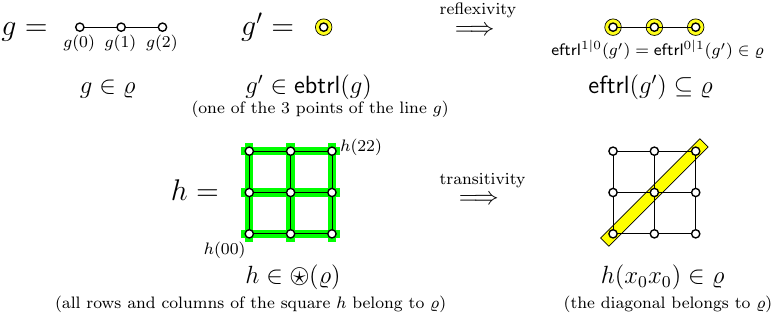}
\end{center}
\caption{Internal reflexivity and transitivity for an R-relation with rectangular arity $\alpha=3^{1}$\label{fig:tra1}}
\end{figure}

(1) \fbox{$\alpha=3^{1}$} (see Figure~\ref{fig:tra1}): $g:3^{1}\to A$ is just a ($1$-dimensional) triple
$(g(0),g(1),g(2))$, whose points are the elementary basic translations
$g'$ of $g$, the function $h:3^{2}\to A:(i,j)\mapsto h(i,j)$ is a $(3\times
3)$-square  
(matrix), whose rows $h(x,j)=(h(0,j),h(1,j),h(2,j))$ and columns
$h(i,y)=(h(i,0),h(i,1),h(i,2))$ are the elementary basic translations,
the set $\bDelta(h)$ here consists of a single diagonal which is the
triple $(h(0,0),h(1,1),h(2,2))$.

\begin{figure}
\includegraphics{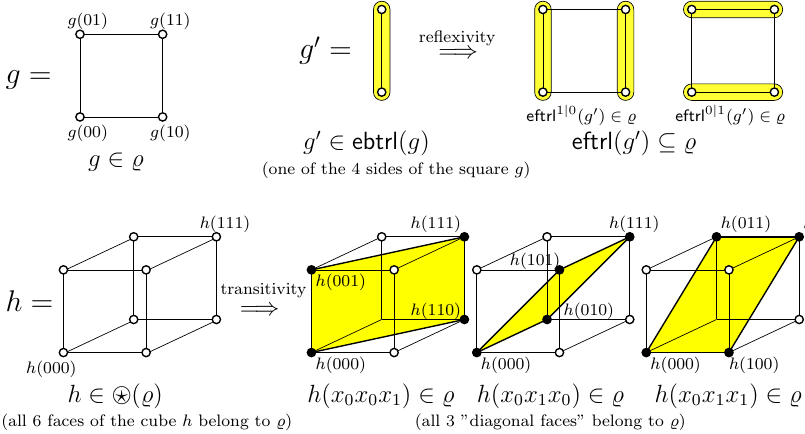}
\caption{Internal reflexivity and transitivity for an R-relation with rectangular arity
  $\alpha=2^{2}$ \label{fig:tra2}}
\end{figure}

(2) \fbox{$\alpha=2^{2}$} (see Figure~\ref{fig:tra2}): $g:2^{2}\to A$
is a $(2\times 2)$-square, whose sides are the elementary basic
translations $g'$ of $g$.
The function
$h:2^{3}\to A$ is a $(2\times 2\times 2)$ cube, the
elementary basic translations of $h$ are exactly the six faces of the
cube and therefore can be understood as possible elements of
$\rho$. According to Definition~\ref{def:Delta}, the set $\bDelta(h)$ consists of three diagonal faces,
namely $h(x_{0},x_{0},x_{1})$, $h(x_{0},x_{1},x_{0})$ and $h(x_{0},x_{1},x_{1})$.

(3) \fbox{$\alpha=2^{3}$} (see Figure~\ref{fig:tra3}): $g:2^{3}\to A$ is a
$(2\times 2\times 2)$-cube, whose elementary basic translations $g'$ are the six faces of $g$. We do not present here $h:2^{4}\to A$ as $4$-dimensional
hypercube, but restrict only to the property of internal
reflexivity.

\begin{figure}\begin{center}
\includegraphics{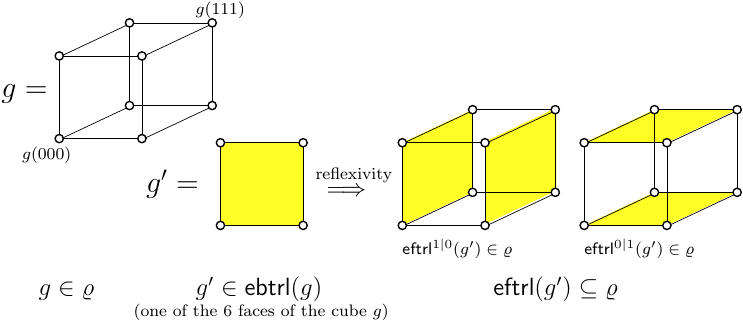}
\end{center}
\caption{Internal reflexivity for an R-relation with rectangular arity
  $\alpha=2^{3}$ \label{fig:tra3}}
\end{figure}

\end{example}

Clones containing constants provide many examples of reflexive
R-relations as the following lemma shows.

\begin{lemma}\label{lem:clonesGiveReflexive}
Let $F\leq \Op(A)$ be a clone, $\cC\subseteq F$ and $d\in\N_{+}$. Then
$F^{(d)}\subseteq A^{A^{d}}$, considered as R-relation of simple arity $A^{d}$, is
reflexive.
\end{lemma}

\begin{proof}
  By \ref{rm:equivalentreflexivity}(i'), $F^{(d)}$ is internally reflexive
  iff $\ebtrl(\eftrl(F^{(d)}))\subseteq F^{(d)}$. But the latter is clear since
  $F$ is a clone and therefore closed under adding fictitious
  variables and (because of $\cC\subseteq F$) substituting constants.
  Moreover, $F^{(d)}$ is reflexive since, by
  assumption, each constant function belongs to $F^{(d)}$. 
\end{proof}
 
\begin{lemma}\label{lm:adjointproperties}
  Let $\rho \subseteq A^{B^d}$ be an R-relation of
  arity $\alpha = B^d$. The following hold.
  \begin{enumerate}[label=\textup{(\roman*)}]
  \item \label{adjointprops1} If $\rho$ is internally reflexive, then
    $\ostar(\rho)$ and, in case $d\geq 1$, $\ebtrl(\rho)$ are
    internally reflexive. 

  \item \label{adjointprops2} If $\rho$ is transitive, then
    $\ostar(\rho)$ is transitive.

  \item \label{adjointprops3} If $\rho$ is a weak generalized
    quasiorder, then $\ostar(\rho)$ is a weak generalized quasi\-order.
  \end{enumerate}
\end{lemma}

\begin{proof}\ref{adjointprops1}: 
Let $\rho \subseteq A^{B^d}$ be
  internally reflexive. We first show that 
  $\ostar(\rho)\subseteq A^{B^{d+1}}$ is also internally reflexive. 
So, using definition
  \ref{rm:equivalentreflexivity}(i''), we must demonstrate that
  \begin{align*}
    \eftrl(\ebtrl(\ostar(\rho))) \subseteq \ostar(\rho), 
   \text{ or equivalently, } \ebtrl(\eftrl(\ebtrl(\ostar(\rho)))) \subseteq \rho
  \end{align*}
In fact, by internal reflexivity of $\rho$ we have
$\ebtrl(\eftrl(\rho))\subseteq \rho$ (cf.\
\ref{rm:equivalentreflexivity}(i')) and by the kernel operator property
\ref{Galoisconnections}(3) we have $\ebtrl(\ostar(\rho))\subseteq\rho$,
 consequently we get
$\ebtrl(\eftrl(\ebtrl(\ostar(\rho))))\subseteq
\ebtrl(\eftrl(\rho))\subseteq\rho$ as desired. 

 Now we show that $\ebtrl(\rho)$ is also internally reflexive.
 Since $\rho$ is internally reflexive (and $d\geq 1$) we have
 $\eftrl(\ebtrl(\rho))\subseteq\rho$ by
 \ref{rm:equivalentreflexivity}(i''), consequently 
 $\ebtrl(\eftrl(\ebtrl(\rho)))\subseteq\ebtrl(\rho)$ which shows
 internal reflexivity of $\ebtrl(\rho)$ by definition
 \ref{rm:equivalentreflexivity}(i'). 

\ref{adjointprops2}: Take $\rho \subseteq A^{B^d}$ to be transitive. We show
  that $\ostar(\rho)\subseteq A^{B^{d+1}}$ is transitive. By \ref{Adef:Refl+Trans}\ref{Adef:transitive}, we need
  to demonstrate that
  $\mathbf{\Delta}(\ostar(\ostar(\rho))) \subseteq \ostar(\rho)$,
  which is equivalent to showing that
    $\ebtrl(\mathbf{\Delta}(\ostar(\ostar(\rho)))) \subseteq \rho$
    (cf.\ \ref{def:ostar}(1)).

  So, let $g \in \ostar(\ostar(\rho))\subseteq A^{B^{d+2}}$ and let
  $g'$ be the result of identifying two variables in $g$. We need to
  see that $\ebtrl(g') \subseteq \rho$.

  Without loss of generality, we suppose that these identified
  variables are the first two (the other cases are argued
  similarly). So,
  \[
    g'(x_0, \dots, x_d) = g(x_0, x_0, x_1, \dots, x_d).
  \]
  Let $h \in \ebtrl(g')$. The following case analysis proves that
  $h \in \rho$ which finishes the proof of $\ref{adjointprops2}$.

Case 1: $h$ is obtained from $g'$ by evaluating $x_0$ at a constant
$c\in B$. 

Then we obtain that
    $h(x_1, \dots, x_{d})= g(c,c,x_1, \dots, x_{d})$.  Thus
    $h\in\ebtrl(\ebtrl(g))\subseteq \rho$ (for the inclusion remember
    that $g \in \ostar(\ostar(\rho))$ and that $\ebtrl$ is the lower adjoint
    of $\ostar$) which implies $h \in \rho$.

Case 2: some variable of $g'$ other than $x_0$ is evaluated at a
    constant. 

Then we obtain
    \[
      h(x_0, \dots, x_{d-1}) = g(x_0, x_0, x_1, \dots, x_{i-1}, c ,
      x_i, \dots, x_{d-1}).
    \]
Therefore $h$ can also be obtained from $g$ by evaluating at first a variable
(which is neither the first nor the second) at a constant, which gives
a function, say $g''\in\ebtrl(g)$ and then identifying the first two
variables of $g'$. Note $g''\in\ostar(\rho)$ since
$g\in\ostar(\ostar(\rho))$. Consequently, $h\in
\bDelta(g'')\subseteq\bDelta(\ostar(\rho))
\subseteq_{\ref{Adef:Refl+Trans}\ref{Adef:transitive}} \rho$ (the latter
inclusion follow from the transitivity of $\rho$)
and we are done.

\ref{adjointprops3}: By definition this immediately follows from
\ref{adjointprops1} and \ref{adjointprops2}.
\end{proof}

We collect some more properties of the internally reflexive or
transitive R-relations. 

\begin{lemma}\label{lem:adjointproperty}
  Let $\rho\subseteq A^{\alpha}$, $\alpha=B^{d}$, $d\in\N$.
  \begin{enumerate}[label=\textup{(\roman*)}]
  \item \label{adj-i}
    If $\rho$ is internally reflexive, then $\trl_{d}(\rho)\subseteq\rho$,
    or more explicitly, if $g\in\rho$ and
    $g'\in A^{B^{d}}$ is given by 
    \begin{align*}
       g'(x_{0},\dots,x_{d-1}):=g(\underbracket{\dots}_{c\in B},x_{u},\underbracket{\dots}_{c\in B},x_{u+1},\underbracket{\dots}_{c\in B},\dots,x_{u+t-1},\underbracket{\dots}_{c\in B})
    \end{align*} (for some $u,v,t\in\N$ with
$t\in\{0,\dots,d\}$ and $u+v=d-t$, cf.~{\rm\ref{trl-1}}) then
$g'\in\rho$. 

\item \label{adj-ii} $\trl_{d}(\rho)=\bigcup_{k=0}^{d}\eftrl^{k}(\ebtrl^{k}(\rho))$ is
  the internal reflexive closure of $\rho$, i.e., the least internally
  reflexive R-relation containing $\rho$. In particular, $\rho$
  is internally reflexive if and only if $\trl_{d}(\rho)=\rho$.

  \item \label{adj-iii} If $\rho$ is transitive, then
  $\bDelta^{k}(\ostar^{k}(\rho))\subseteq \rho$ for $k\in\N_{+}$.

\item \label{adj-iv} If $\rho$ is a weak generalized quasiorder
  (i.e., internally reflexive and transitive), then $g\in\rho$ implies
  $g^{\pi}\in\rho$ for each mapping (in particular for each permutation)
  $\pi:\{0,1,\dots,d-1\}\to\{0,1,\dots,d-1\}$ where $g^{\pi}$ is
  given by
  \begin{align*}
    g^{\pi}(x_{0},\dots,x_{d-1}):=g(x_{\pi(0)},\dots,x_{\pi(d-1)})
\quad\text{(cf.~\ref{def:minion})}.
  \end{align*}
Equivalently, we have $\rho\circ\mu_{\pi}\subseteq\rho$ for all arity
  transformations (cf.~\ref{def:aritymap}) with 
$\mu_{\pi}:\alpha\to\alpha:(x_{0},\dots,x_{d-1}) 
                \mapsto (x_{\pi(0)},\dots,x_{\pi(d-1)})$,
since $g^{\pi}=g\circ\mu_{\pi}$.


  \end{enumerate}

\end{lemma}

\begin{proof} 
\ref{adj-i}: Let $h\in\trl_{d}(\rho)$. From
Lemma~\ref{lem:trldecomposition} we can conclude
$h\in\eftrl^{k}(\ebtrl^{k}(\rho))$ for some $k\in\{0,1,\dots,d\}$
($k=0$ means $h\in\rho$). We are going to
show $\eftrl^{k}(\ebtrl^{k}(\rho))\subseteq\rho$ (by induction on $k$)
which gives $h\in\rho$ and will finish the proof. 

In fact, for $k=1$ we have
$\eftrl(\ebtrl(\rho))\subseteq\rho$ by
\ref{rm:equivalentreflexivity}(i'') since $\rho$ is internally reflexive. Assume
$\eftrl^{k}(\ebtrl^{k}(\rho))\subseteq\rho$ for
$k\in\{1,\dots,d-1\}$. Since $\ebtrl^{k}(\rho)$ is internally
reflexive by Lemma~\ref{lm:adjointproperties}\ref{adjointprops1} we
have
\begin{align*}
  \eftrl^{k+1}(\ebtrl^{k+1}(\rho))&=
\eftrl^{k}(\eftrl(\ebtrl(\ebtrl^{k}(\rho))))\\
&\subseteq_{\text{\ref{rm:equivalentreflexivity}(i'')}}
\eftrl^{k}(\ebtrl^{k}(\rho))\subseteq\rho.
\end{align*}

\ref{adj-ii}: The characterization of $\trl_{d}(\rho)$ as join of
$\eftrl^{k}(\ebtrl^{k}(\rho))$ as well as
$\rho\subseteq\trl_{d}(\rho)$ follows from
Lemma~\ref{lem:trldecomposition}\ref{trldecomposition-iii} (for
$n=d$ and $k=0$). It remains to show that $\trl_{d}(\rho)$ is the
least internally reflexive R-relation containing $\rho$. In fact,
$\ebtrl(\eftrl(\trl_{d}(\rho)))$ is contained in $ \trl(\rho)$ (by
definition~\ref{def:trl1}) and consists of $d$-ary functions (by
definition of elementary translations), thus
$\ebtrl(\eftrl(\trl_{d}(\rho)))\subseteq\trl_{d}(\rho)$ which shows
that $\trl_{d}(\rho)$ is internally reflexive (by
Remark~\ref{rm:equivalentreflexivity}(i')). Moreover, if
$\sigma\subseteq A^{B^{d}}$ is internally reflexive and contains
$\rho$, then 
$\trl_{d}(\rho)\subseteq\trl_{d}(\sigma)\subseteq_{\ref{adj-i}}\sigma$.

\ref{adj-iii}: We prove the statement by induction on $k$. For $k=1$ we
have $\bDelta(\ostar(\rho))\subseteq\rho$ by definition of
transitivity (cf.~\ref{Adef:Refl+Trans}\ref{Adef:transitive}). By 
Lemma~\ref{lm:adjointproperties}\ref{adjointprops1},
$\ostar^{k}(\rho)$ is transitive for $k\in\N_{+}$. Thus
\begin{align*}
  \bDelta^{k+1}(\ostar^{k+1}(\rho))=\bDelta^{k}(\bDelta(\ostar(\ostar^{k}(\rho))))
  \subseteq_{\text{\ref{Adef:Refl+Trans}\ref{Adef:transitive}}}\bDelta^{k}(\ostar^{k}(\rho))\subseteq \rho,
\end{align*}

the last inclusion holds by induction hypothesis.

\ref{adj-iv}: Let $g\in\rho$ and define $h:B^{2d}\to A$ by 
\begin{align*}
     h(x_{0},\dots,x_{d-1},x_{d},\dots,x_{2d-1}):=g(x_{d},\dots,x_{2d-1}). 
\end{align*}
Clearly $h\in\eftrl^{d}(g)$. Thus 
$\ebtrl^{d}(h)\subseteq\ebtrl^{d}\eftrl^{d}(g)\subseteq_{\ref{adj-ii}}
\trl_{d}(g)\subseteq\trl_{d}(\rho)\subseteq_{\ref{adj-i}}\rho$, 
which  implies $h\in\ostar^{d}(\rho)$ since $\ebtrl$
is the lower adjoint of $\ostar$ (cf.\
\ref{rm:equivalentreflexivity}(i')). Moreover we have  $g^{\pi}\in\bDelta^{d}(h)$
(in fact, $g^{\pi}$ can be obtained from $h$ by identifying $x_{d+i}$ with
$x_{\pi(i)}$  for $i=d-1,\dots,1,0$).
Consequently, $g^{\pi}\in\bDelta^{d}(h)\subseteq
\bDelta^{d}(\ostar^{d}(\rho))\subseteq_{\ref{adj-iii}}\rho$, which is what was
to be shown.
\end{proof}

There is another adjoint pair of operators that is relevant for
us. Indeed, we remarked immediately after Definition~\ref{MstarPoe}
that $M^{*_d}$ for a $d$-ary clone fragment $M$ is just the image of
$M$ under the upper adjoint to $\trl$. We have generalized $\trl$ to
arity relations of simple arity $\rho$, hence we get the following monotone Galois
connection. 

\begin{definition}\label{def:star}
 Let $d\in\N$. According to \ref{Galoisconnections}, the mapping
  $\trl_{d}:A^{B^{n}}\to\bigcup_{n\in\N}A^{B^{d}}:g\mapsto\trl_{d}(g)$ (cf.~\ref{def:trl1}) induces a
  monotone Galois connection, whose operators are given by 
\begin{align*}
  \trl_d:& \Pow(\bigcup_{n \in \N} A^{B^n}) \to \Pow(A^{B^d}):
  X \mapsto \trl_d(X), \text{ and}\\
  ^{*_d}:& \Pow(A^{B^d}) \to \Pow(\bigcup_{n \in \N} A^{B^n}):
  \rho \mapsto  \rho^{*_d},\quad
     \rho^{*_d}:=\{ g \in \bigcup_{n \in \N} A^{B^n}\mid  \trl_d(g) \subseteq \rho \}.
\end{align*}
It is clear
(since $\trl_{d}(\trl(g))=\trl_{d}(g)$)
that $\rho^{*_d}$ is closed under all translations, i.e.,
$\trl(\rho^{*_d}) \subseteq \rho^{*_d}$.

\end{definition}

\begin{definition}\label{doublearrow} 
Let $d\in\N$ and $\rho\subseteq A^{B^{d}}$. For $n\in\N$ we define the
R-relations $\rho_{n}\in A^{B^{n}}$ as follows:
 \begin{align*}
  \rho_{n}= \begin{cases}
    \ostar^{n-d}(\rho) & \text{ if $ n > d$},\\
    \ebtrl^{d-n}(\rho)  &\text{ if $n \leq d$}.
  \end{cases}
&&\text{Further we set $\rho^{\leftrightarrow}:=\bigcup_{n\in\N}\rho_{n}$.}
\end{align*}
Note $\rho_{d}=\rho$, $\rho_{n+1}=\ostar(\rho_{n})$ if $n\geq d$, and
$\rho_{n-1}=\btrl(\rho_{n})$ if $1\leq n\leq d$.
\end{definition}

\begin{proposition}\label{prop:doublearrowisminion} 
Let $\rho$ be an R-relation  of
  simple arity $\alpha=B^{d}$, $d\in\N_{+}$.
The following are equivalent: 
  \begin{enumerate}[label=\textup{(\alph*)}]
  \item\label{doublearrowisminion-a} $\rho^{\leftrightarrow}$ is a minion.
  \item\label{doublearrowisminion-b} $\rho$ is a weak generalized quasiorder.
  \end{enumerate}
\end{proposition}

\begin{proof}

\ref{doublearrowisminion-a}$\Rightarrow$\ref{doublearrowisminion-b}:
A minion is closed under adding fictitious variables and under
identification of variables. Thus
$\eftrl(\rho)=\eftrl(\rho_{d})\subseteq\rho_{d+1}=\ostar(\rho)$
and $\bDelta(\ostar\rho)=\bDelta(\rho_{d+1})\subseteq\rho_{d}=\rho$
showing internal reflexivity
(cf.~\ref{Adef:Refl+Trans}\ref{Adef:internally-reflexive}) and
transitivity (cf.~\ref{Adef:Refl+Trans}\ref{Adef:transitive}).

\ref{doublearrowisminion-b}$\Rightarrow$\ref{doublearrowisminion-a}:
For $\rho^{\leftrightarrow}$ being a minion, it is enough to show that
$\eftrl(\rho_{n})\subseteq\rho_{n+1}$,
$\bDelta(\rho_{n})\subseteq\rho_{n-1}$ and that each $\rho_{n}$ is closed
under permutation of variables (i.e., $g^{\pi}\in\rho_{n}$ for
$g\in\rho_{n}$ and a permutation $\pi:n\to n$,
cf.~\ref{lem:adjointproperty}\ref{adj-iv}), clearly then it is also
closed under arbitrary adding fictitious variables and identification
of variables and thus a minion (cf.~\ref{def:minion}).

From Lemma~\ref{lm:adjointproperties}\ref{adjointprops1}
follows that all $\rho_{n}$ are internally
reflexive and, for $n\geq d$, also transitive by
Lemma~\ref{lm:adjointproperties}\ref{adjointprops2}. 
  
Thus, for $n\geq d$, we have $\bDelta(\rho_{n+1})=\bDelta(\ostar(\rho_{n}))\subseteq\rho_{n}$
(cf.~\ref{Adef:Refl+Trans}\ref{Adef:transitive}), moreover,
$\eftrl(\rho_{n})\subseteq\ostar(\rho_{n})=\rho_{n+1}$ 
(cf.~\ref{Adef:Refl+Trans}\ref{Adef:internally-reflexive}) 
and, by Lemma~\ref{lem:adjointproperty}\ref{adj-iv},
each $\rho_{n}$ is closed under permutation of variables.
 Consequently, the
``upper'' part $\bigcup_{n\geq d}\rho_{n}$ of $\rho^{\leftrightarrow}$ satisfies all
conditions for being a minion.

All conditions for the ``lower'' part $\bigcup_{n\leq d}\rho_{n}$ of
$\rho^{\leftrightarrow}$ can be checked by ``lifting'' to the upper
part with $\eftrl$ and ``returning'' to the lower part with $\ebtrl$.

For this we need to know that $\rho^{\leftrightarrow}$ is closed under
$\ebtrl$ and $\eftrl$ which immediately follows from the definition of
$\rho^{\leftrightarrow}$ using internal reflexivity (e.g., for $n<d$, $\eftrl(\rho_{n})=\eftrl(\ebtrl(\rho_{n-1}))\subseteq\trl_{d}(\rho_{n+1})
\subseteq_{\ref{lem:adjointproperty}\ref{adj-i}}\rho_{n+1}$) and the
adjointness properties of $\btrl$ and $\ostar$ (e.g., for $n\geq d$, $\ebtrl(\rho_{n+1})\subseteq\rho_{n}$).

Thus, starting with a function $g\in\rho_{n}$ ($n\leq d$) we get an
function $\tilde g\in\rho_{d+1}$ by appending $d-n+1$
fictious variables, i.e., an element in the upper part.
Thus $\eftrl(\tilde g)$, $\tilde g^{\tilde\pi}$ (for a permutation
$\tilde\pi:d+1\to d+1$ whose restriction to $n$ is a permutation 
$\pi:n\to n$) and $\bDelta(\tilde g)$
belong to the upper part (we have chosen $d-n+1$ many fictitious
variables in order to ensure 
that $\bDelta(\tilde g)$ is in the upper part
also for $n=d$). Now one can substitute each originally prepended
fictitious variable by a constant (i.e., we apply $\ebtrl^{d-n+1}$) and we get
that the corresponding functions $\eftrl(g)$, $g^{\pi}$ (for a
permutation $\pi:n\to n$) and $\bDelta(g)$
also belong to $\rho^{\leftrightarrow}$.
\end{proof}

\begin{proposition}\label{prop:starcharacterization}
  Let $\rho\subseteq A^{B^{d}}$ be internally reflexive. Then $\rho^{*_{d}}=\rho^{\leftrightarrow}$.
\end{proposition}


\begin{proof} Let $g\in A^{B^{n}}$. Recall
  that $g\in\rho^{*_{d}}$ 
  means $\trl_{d}(g)\subseteq \rho$. We distinguish two cases.

Case 1: $n\leq d$: We have to show $\trl_{d}(g)\subseteq \rho\iff
g\in\rho_{n}=\ebtrl^{d-n}(\rho)$.

To see ``$\Longrightarrow$'', assume $\trl_{d}(g)\subseteq \rho$ and
consider $h\in\eftrl^{d-n}(g)$ given by appending $d-n$ ficticious variables:
\begin{align*}
  h(x_{0},\dots,x_{n-1},x_{n},\dots,x_{d-1}):=g(x_{0},\dots,x_{n-1},x_{n}). 
\end{align*}
These fictitious variables can be evaluated at constants, which shows
\begin{align*}
  g\in\ebtrl^{d-n}(h)\in \ebtrl^{d-n}(\eftrl^{d-n}(g))
\subseteq\ebtrl^{d-n}(\trl_{d}(g))\subseteq\ebtrl^{d-n}(\rho)
\end{align*}
 and we are done (for $n=d$ some parts become trivial because then
 $h=g$).

For ``$\Longleftarrow$'', $g\in\ebtrl^{d-n}(\rho)$ implies
  $\trl_{d}(g)\subseteq\trl_{d}(\ebtrl^{d-n}(\rho))\subseteq\trl_{d}(\rho)
 =_{\ref{lem:adjointproperty}\ref{adj-ii}}\rho$.

Case 2: $n>d$: We have to show $\trl_{d}(g)\subseteq \rho\iff
g\in\rho_{n}=\ostar^{n-d}(\rho)$, or equivalently 
 $\trl_{d}(g)\subseteq \rho\iff \ebtrl^{n-d}(g)\subseteq\rho$ (cf.~\ref{ostar-2}).

In fact, ``$\Longrightarrow$'' follows immediately from
  $\ebtrl^{n-d}(g)\subseteq\trl_{d}(g)$.

To see ``$\Longleftarrow$'', assume $\ebtrl^{n-d}(g)\subseteq\rho$ and
let $h\in\trl_{d}(g)$. We have to show $h\in\rho$. According to 
\ref{lem:trldecomposition}\ref{trldecomposition-iii} there is some
$t\in\{0,\dots,d\}$ such that we get
\begin{align*}
  h\in\eftrl^{d-t}(\ebtrl^{n-t}(g))&=\eftrl^{d-t}(\ebtrl^{d-t}(\ebtrl^{n-d}(g)))\\&\subseteq\eftrl^{d-t}(\ebtrl^{d-t}(\rho))\subseteq\trl_{d}(\rho)
=_{\ref{lem:adjointproperty}\ref{adj-ii}}\rho. \tag*{\qed}
\end{align*}
\renewcommand{\qedsymbol}{}
\end{proof}

We still need the following definition and lemma which are crucial for the proof of
the Theorems~\ref{thm:PolofGQuordMain} and \ref{thm:PolofGQuordMain2}.

\begin{definition}\label{def:SigmaofTuples}
 Let $f:A^{n}\to A$ and let $(g_{0},\dots,g_{n-1}) \in \left(A^{\alpha}\right)^n$, for $\alpha = B^d$. We define the following sets:
 \begin{align*}
   \Sigma(g_{0},\dots,g_{n-1})&:=\{(g'_{0},\dots,g'_{n-1})\mid
\ g'_{0} \in  \trl_{d}\{g_{0} \}, \dots, g'_{n-1}\in
                                \trl_{d}\{g_{n-1}\} \\
     &\text{\hspace*{12.5ex} and at most $d$ many $g'_{i}$'s are nonconstant}\},\\
f(\Sigma(g_{0},\dots,g_{n-1}))&:=\{f(g'_{0},\dots,g'_{n-1})\mid
   (g'_{0},\dots,g'_{n-1})\in \Sigma(g_{0},\dots,g_{n-1})\}.
 \end{align*}
\end{definition}

\begin{lemma}\label{lem:PolofGQuord}
  Let $\rho \subseteq A^\alpha$ be a transitive R-relation of arity $\alpha=B^{d}$, let $f:A^{n}\to A$, and let $(g_{0},\dots,g_{n-1})\in \rho^n$. 

Then we have:
  $f(\Sigma(g_{0},\dots,g_{n-1}))\subseteq\rho
               \implies f(g_{0},\dots,g_{n-1})\in\rho$.
\end{lemma}

\begin{proof} Let $g_{0},\dots,g_{n-1}\in \rho$ and assume
  $f(\Sigma(g_{0},\dots,g_{n-1}))\subseteq\rho$ holds for some $f:A^{n}\to A$. 
If $n\leq d$ then
$(g_{0},\dots,g_{n-1})\in\Sigma(g_{0},\dots,g_{n-1})$ and, by
assumption, we trivially can conclude $f(g_{0},\dots,g_{n-1})\in\rho$.

Thus we can assume $n>d$.
We have to show $f(g_{0},\dots,g_{n-1})\in\rho$.

Define a $dn$-ary operation $h\in A^{B^{dn}}$ by
\begin{align*}
  h(X^{(0)},\dots,X^{(n-1)}):=f(g_{0}(x_{0}^{(0)},\dots,x_{d-1}^{(0)}),\dots, 
     g_{n-1}(x_{0}^{(n-1)},\dots,x_{d-1}^{(n-1)}))
\end{align*}
using notation $X^{(i)}:=(x_{0}^{(i)},\dots,x_{d-1}^{(i)})$ for
$i\in\{0,1,\dots,n-1\}$. Thus $X^{(i)}\in B^{d}$, $g_{i}:B^{d}\to A$.
Identifying in $h$ the variables $x_{0}^{(i)}=\ldots=x_{d-1}^{(i)}$ for
$i\in\{0,\dots,n-1\}$, (i.e., $X^{(0)}=\ldots=X^{(n-1)}$) we get
$f(g_{0},\dots,g_{d-1})$. Formally we
have to apply $(n-1)$ times the operator $\bDelta$ for each $i$,
consequently we have $f(g_{0},\dots,g_{n-1})\in\bDelta^{d(n-1)}(h)$.

In order to apply Lemma~\ref{lem:adjointproperty}, we are going to prove the following 

\underline{Claim:} $h\in\ostar^{d(n-1)}(\rho)$, or, equivalently
(cf.~\ref{ostar-2}), $\ebtrl^{d(n-1)}(h)\subseteq\rho$.

Let $h'\in\ebtrl^{d(n-1)}(h)$ (we must show $h'\in\rho$), i.e., $h'$
can be obtained from $h$ by substituting $nd-d$ variables in
$X^{(0)},\dots,X^{(n-1)}$ by a constant, i.e., we have
\begin{align*}
   h'(x_{0},\dots,x_{d-1}):=h(\underbracket{\dots}_{\in B},x_{0},\underbracket{\dots}_{\in B},x_{1},\underbracket{\dots}_{\in B},\dots,x_{d-1},\underbracket{\dots}_{\in B})=:h(Y^{(0)},\dots,Y^{(n-1)})
\end{align*}
where $Y^{(i)}$ is the $i$-th block if we divide the $dn$-tuple
$(\underbracket{\dots}_{\in B},x_{0},\underbracket{\dots}_{\in
  B},x_{1},\underbracket{\dots}_{\in
  B},\dots,x_{d-1},\underbracket{\dots}_{\in B})$ into $n$ blocks of
equal length $d$.
Therefore each $Y^{(i)}$ must be of the form
\begin{align*}
  (\underbracket{\dots}_{\in B},x_{u},\underbracket{\dots}_{\in
  B},x_{u+1},\underbracket{\dots}_{\in
  B},\dots,x_{u+t-1},\underbracket{\dots}_{\in B})
\end{align*}
 (for $t=0$ the
tuple $Y^{(i)}$ contains only constants and no $x_{j}$). 
 Define $g'_{i}\in
A^{B^{d}}$ by $g'_{i}(x_{0},\dots,x_{d-1}):=g_{i}(Y^{(i)})$ for
$i\in\{0,\dots,n-1\}$. 
According to \ref{trl-1} 
 we have
$g_{i}'\in\trl_{d}(g_{i})$.

Since there are only $d$ variables $x_{i}$, there are at
most $d$ tuples among $Y^{(0)},\dots,Y^{(n-1)}$ where possibly not all
components are constants, 
in other words, at least $n-d$ of the $Y^{(i)}$'s
are constant tuples ($\in B^{d}$) and the corresponding $g'_{i}(x_{0},\dots,x_{d-1})=g_{i}(Y^{(i)})$ are
evaluated to $c_{i}:=g_{i}(Y^{(i)})=g_{i}(b_{0}^{(i)},\dots,b_{d-1}^{(i)})$ (for
some $b_{0}^{(i)},\dots,b_{d-1}^{(i)}\in B$). 
Note that $c_{i}\in\Ima\{g_{i}\}$.
The $d$ remaining $Y^{(i)}$
may contain one, more or no
variables from $\{x_{0},\dots,x_{d-1}\}$.
Thus $(g'_{0},\dots,g'_{n-1})\in\Sigma(g_{0},\dots,g_{n-1})$. Further,
\begin{align*}
  h'(x_{0},\dots,x_{d-1})&=h(Y^{(0)},\dots,Y^{(n-1)})
   =f(g_{0}(Y^{(0)}),\dots,g_{n-1}(Y^{n-1}))\\
   &=f(g'_{0}(x_{0},\dots,x_{d-1}),\dots,g'_{n-1}(x_{0},\dots,x_{d-1})).
\end{align*}
Consequently, $h'=f(g'_{0},\dots,g'_{n-1})\in
f(\Sigma(g_{0},\dots,g_{n-1}))\subseteq\rho$ by assumption, i.e.,
$h'\in\rho$, and the claim is proved.

From transitivity of $\rho$, the claim and Lemma~\ref{lem:adjointproperty}\ref{adj-iii} we
finally conclude 
\begin{align*}
  f(g_{0},\dots,g_{d-1})\in\bDelta^{d(n-1)}(h)\subseteq\bDelta^{d(n-1)}(\ostar^{d(n-1)}(\rho))\subseteq\rho
\end{align*}
which finishes the proof.
\end{proof}




Now one could ``easily'' prove the following 
theorem which shows that (weak) generalized quasiorders of simple
arity satisfy our ``motivating'' property $\Xi_{d}$ (or $\wXi_{d}$,
respectively). However, we can omit the proof here because it is a
special case of 
Theorem~\ref{thm:PolofGQuordMain2} (case $s=1$) the proof of which
however is based on the results of this section.

\begin{theorem}[see Theorem~\ref{thm:PolofGQuordMain2}]\label{thm:PolofGQuordMain}
  Let $\rho \subseteq A^{B^d}$ ($d\in\N_{+}$) be a weak generalized
  quasiorder. Then $\wXi_{d}(\rho)$. 
If $\rho$ is a generalized quasiorder then we have $\Xi_{d}(\rho)$. \qed

\end{theorem}

\begin{remark}\label{rem:gQuord}
As already mentioned in the introduction, the results of this section
generalize the $1$-dimensional case $d=1$ treated in \cite{JakPR2024}
where the authors used quite other notation. To keep ``compatibility''
we give here some hints for readers familiar with this paper.

In \cite{JakPR2024} the finite set $B$ is always taken as $B=\{1,\dots,m\}$ for some
$m\in\N_{+}$, i.e., $1$-dimensional generalized quasiorders are just
usual relations $\rho\subseteq A^{m}$. The elements of $ A^{B^{1}}$
and $A^{B^{2}}$ are presented as tupels $(a_{i})_{i\in
  B}=(a_{1},\dots,a_{m})$ and matrices $(a_{i,j})_{i,j\in B}$,
respectively. The crucial notation needed for generalized quasiorders
is as follows: $\rho\models(a_{i,j})$ expresses the fact that each row
and each column of the matrix $(a_{i,j})$ belongs to $\rho$. Then we
get for $h:B^{2}\to A:(i,j)\mapsto a_{i,j}$:
\begin{align*}
  \rho\models (a_{i,j})_{i,j\in B} \iff& h\in \ostar(\rho),
\quad(a_{i,i})_{i\in B}=\bDelta(h),
\\ \text{thus }
\ostar(\rho)=&\ \{(a_{i,j})_{i,j\in B}\mid 
     \rho\models (a_{i,j})_{i,j\in B}\},\\
 \bDelta(\sigma)=&\ \{(a_{i.i})_{i\in B}\mid (a_{i,j})_{i,j\in
  B}\in\sigma\} \text{ for } \sigma\subseteq A^{B^{2}}.
\end{align*}

\end{remark}

\section{The case of higher-dimensional generalized quasiorders of
  compound arity}\label{sec:GQuord2}

\begin{notation}\label{notationGQuord}
Let
$\alpha=\alpha_{1}\times\ldots\times\alpha_{s}=B_{1}^{d_{1}}\times\ldots\times
B_{s}^{d_{s}}$ be a (rectangular) compound arity
(cf.~\ref{def:Rrelation}).
For $j\in\{1,\dots,s\}$ let 
\begin{align*}
  \alpha/\alpha_{j}:=\alpha_{1}\times\ldots\times
\alpha_{j-1}\times \alpha_{j+1}\times\ldots\times\alpha_{s}
 \text{ (deleting the $j$-th factor $\alpha_{j}$).}
\end{align*}
For elements $g\in A^{\alpha}$, i.e., $g:\alpha\to
A:X^{(\alpha)}\mapsto g(X^{(\alpha)})$, we use the notation 
\begin{align*}
  X^{(\alpha)}&:=(X^{(\alpha_{1})},\dots,X^{(\alpha_{s})})\in\alpha,\\
  X^{(\alpha_{j})}&:=
      (x_{0}^{(\alpha_{j})},\dots,x_{d_{j}-1}^{(\alpha_{j})})\in\alpha_{j}=B_{j}^{d_{j}},\\
  X^{(\alpha/\alpha_{j})}&:=
(X^{(\alpha_{1})},\dots,X^{(\alpha_{j-1})},X^{(\alpha_{j+1})},\dots,X^{(\alpha_{s})})\in\alpha/\alpha_{j},
\end{align*}
for $j\in\{1,\dots,s\}$. Formally these
$X^{(\alpha)},X^{(\alpha_{j})},X^{\alpha/\alpha_{j}}$ are elements of
the indicated sets $\alpha,\alpha_{j},\alpha/\alpha_{j},$ but they can also be
interpreted as variables for functions $\alpha\to A$,
$\alpha_{j}\to A$ and $\alpha/\alpha_{j}\to A$, and sometimes it is convenient
to denote the function, e.g., $g:\alpha\to A$ by the term $g(X^{(\alpha)})$.

 There is a
canonical bijection
$A^{\alpha}\to(A^{\alpha/\alpha_{j}})^{\alpha_{j}}:g\mapsto
g^{(\alpha_{j})}$ where the image $g^{(\alpha_{j})}\in A^{\alpha/\alpha_{j}}$ of $g\in
A^{\alpha}$ is given as
follows: for $X^{(\alpha_{j})}\in\alpha_{j}$ ($j\in\{1,\dots,s\}$) we define 
$g^{(\alpha_{j})}(X^{\alpha_{j}})\in A^{\alpha/\alpha_{j}}$ according to
\begin{align*}
  (g^{(\alpha_{j})}(\underbrace{X^{\alpha_{j}}}_{\in \alpha_{j}}))
(\underbrace{X^{\alpha/\alpha_{j}}}_{\in\alpha/\alpha_{j}})
              :=g(\underbrace{X^{(\alpha)}}_{\in\alpha})
     =g(X^{(\alpha_{1})},\dots,X^{(\alpha_{s})}).
\end{align*}
Therefore, any R-relation $\rho\subseteq A^{\alpha}$ (on base set
$A$) with compound relational arity $\alpha$ can be considered as an R-relation, denoted by
$\rho^{(\alpha_{j})}$, on base set $A^{\alpha/\alpha_{j}}$ with simple
arity $\alpha_{j}$ as follows:
\begin{align*}
  \rho^{(\alpha_{j})}:=\{g^{(\alpha_{j})}\mid g\in\rho\}\subseteq
  (A^{\alpha/\alpha_{j}})^{\alpha_{j}}.
\end{align*}
Clearly, $\Ima\rho^{(\alpha_{j})}=\{g^{(a_{j})}(C^{(\alpha_{j})})\mid
g\in\rho, C^{(\alpha_{j})}\in\alpha_{j}\}$ (here
$C^{(\alpha_{j})}=(c_{0}^{(\alpha_{j})},\dots,c_{d_{j}-1}^{(\alpha_{j})})\in
B_{j}^{d_{j}}$).

Further, any operation $f:A^{n}\to A$ induces an operation on any power of $A$ in the standard way.
Here, we take this power to be $A^{\alpha/\alpha_{j}}$ and obtain a function
$f^{(\alpha_{j})}:(A^{\alpha/\alpha_{j}})^{n}\to A^{\alpha/\alpha_{j}}$ as
follows:
\begin{align*}
  f^{(\alpha_{j})}(h_{0},\dots,h_{n-1})(X^{\alpha/\alpha_{j}})
 :=f(h_{0}(X^{\alpha/\alpha_{j}}),\dots,h_{n-1}(X^{\alpha/\alpha_{j}}))
\end{align*}
for $h_{0},\dots,h_{n-1}\in A^{\alpha/\alpha_{j}}$.
\end{notation}

We provide in Figure~\ref{fig:compoundarity} the picture that we have
in mind when working with, for example, the compound arity
$\alpha = \alpha_1 \times \alpha_2 = 3^1 \times 2^1$. In the top left
corner of the figure we show a typical element of an $R$-relation of
arity $\alpha$. A picture of $g^{(\alpha_2)}$ is shown on the right of
the figure, where we have equivalently depicted $g$ as a pair of
triples. The other way of interpreting $g$ as a triple of pairs is
shown below $g$.

\begin{figure}
    \centering
    \includegraphics[width=1\linewidth]{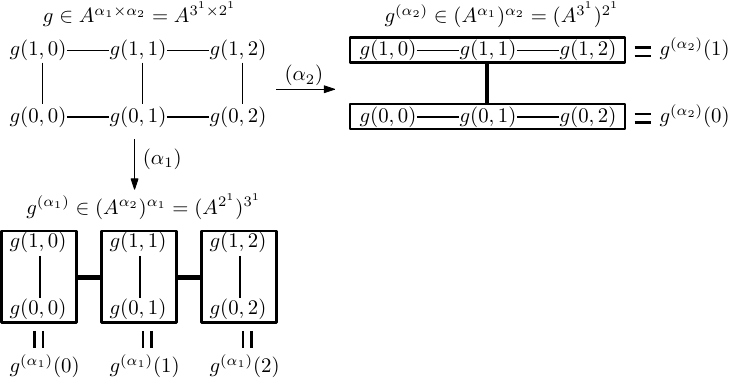}
    \caption{Compound arity $\alpha = \alpha_{1}\times\alpha_{2}=3^1 \times 2^1$}
    \label{fig:compoundarity}
\end{figure}

\begin{lemma}\label{lem:1forThm} 

Let $\rho\in\RRela[\alpha](A)$, i.e., $\rho\subseteq A^{\alpha}$,
$f:A^{n}\to A$ and $g,g_{0},\dots,g_{n-1}\in A^{\alpha}$. With the notation
from \ref{notationGQuord}, for $j\in\{1,\dots,s\}$ we have
\begin{enumerate}[label=\textup{(\roman*)}]

\item \label{1forThm-i}
$g\in\rho\iff g^{(\alpha_{j})}\in\rho^{(a_{j})}$,

\item \label{1forThm-ii}
$(f(g_{0},\dots,g_{n-1}))^{(\alpha_j)} = 
f^{(\alpha_{j})}(g_{0}^{(\alpha_{j})},\dots,g_{n-1}^{(\alpha_{j})})$,

\item \label{1forThm-iii} 
$f(g_0, \dots, g_{n-1}) \in \rho\iff 
f^{(\alpha_{j})}(g_0^{(\alpha_j)}, \dots, g_{n-1}^{(\alpha_j)}) \in \rho^{(\alpha_{j})}$,

\item \label{1forThm-iv} 
$f\preserves\rho\iff f^{(\alpha_{j})}\preserves\rho^{(\alpha_{j})}$,

\item \label{1forThm-v} $g$ is constant at $\alpha_{j}$ (cf.~\ref{def:depend}) $\iff$
  $g^{(\alpha_{j})}$ is constant.

\end{enumerate}

\end{lemma}

\begin{proof}
\ref{1forThm-i} directly follows from the definitions (note that
$g\mapsto g^{(\alpha_{j})} $ is a bijection). 
\ref{1forThm-iii}
follows from \ref{1forThm-i} and \ref{1forThm-ii}, and
\ref{1forThm-iv} follows directly from \ref{1forThm-iii}. 

\ref{1forThm-ii}: Both sides are 
$\alpha_{j}$-tuples on the set $A^{\alpha/\alpha_{j}}$, so (by the
above mentioned bijection) there exist $h_{1},h_{2}\in A^{\alpha}$
such that $h_{1}^{(\alpha_{j})}$ is the left side and 
$h_{2}^{(\alpha_{j})}$ is the right side of
\ref{1forThm-ii} (clearly, $h_{1}=f(g_{0},\dots,g_{n-1})$). Evaluating elementwise these functions for 
each coordinate we get equality $h_{1}=h_{2}$, which finishes the
proof of~\ref{1forThm-ii}: in fact, using the
notation and definitions from \ref{notationGQuord}, we have
\begin{align*}
(h_{1}^{(\alpha_{j})}(X^{\alpha_{j}}))(X^{\alpha/\alpha_{j}})&=(f(g_{0},\dots,g_{n-1}))(X^{\alpha})= f(g_{0}(X^{\alpha}),\dots,g_{n-1}(X^{\alpha})),\\
(h_{2}^{(\alpha_{j})}(X^{\alpha_{j}}))(X^{\alpha/\alpha_{j}})&= 
(f^{(\alpha_{j})}(g_{0}^{(\alpha_{j})}(X^{\alpha_{j}}),\dots,g_{n-1}^{(\alpha_{j})}(X^{\alpha_{j}})))(X^{\alpha/\alpha_{j}})\\
&=f(g_{0}^{(\alpha_{j})}(X^{(\alpha_{j})})(X^{\alpha/\alpha_{j}}),
  \dots,g_{n-1}^{(\alpha_{j})}(X^{(\alpha_{j})})(X^{\alpha/\alpha_{j}}))\\
&=f(g_{0}(X^{\alpha}),\dots,g_{n-1}(X^{\alpha})).
\end{align*}
%

\ref{1forThm-v}: According to and using the notation from
\ref{def:depend}, $g$ is constant at
$\alpha_{j}$ if
\begin{align*}
g(X^{(\alpha_{1})},\dots,X^{(\alpha_{s})})=
g(X^{(\alpha_{1})},\dots,X^{(\alpha_{j-1})},
C^{(\alpha_{j})},X^{(\alpha_{j+1})},\dots,X^{(\alpha_{s})}),
\end{align*}
equivalently
\begin{align*}
g^{(\alpha_{j})}(X^{(\alpha_{j})})(X^{(\alpha/\alpha_{j})})
=
g^{(\alpha_{j})}(C^{(\alpha_{j})})(X^{(\alpha/\alpha_{j})}).
\end{align*}

This shows that $g^{(\alpha_{j})}\in
(A^{\alpha/\alpha_{j}})^{\alpha_{j}}$ maps each $X^{(\alpha_{j})}$ to the
constant 
element $h\in A^{\alpha/\alpha_{j}}$, given by 
$h(X^{(\alpha/\alpha_{j})}):=g^{(\alpha_{j})}(C^{(\alpha_{j})})(X^{(\alpha/\alpha_{j})})$
for some
$C^{(\alpha_{j})}\in \alpha_{j}=B_{j}^{d_{j}}$.
\end{proof}

\begin{definition}\label{Adef2:Refl+Trans} Let $\rho \subseteq A^{\alpha}$ be an R-relation of compound arity
  $\alpha=\alpha_{1}\times\ldots\times\alpha_{s}$. We say that $\rho$ is
  \begin{enumerate}[label=\textup{(\roman*)}]
 \item \label{Adef2:internally-reflexive}\emph{internally reflexive}
    if $\rho^{(\alpha_{j})}$ is internally reflexive for all $j\in\{1,\dots,s\}$,

  \item \label{Adef2:transitive} \emph{transitive} if
    $\rho^{(\alpha_{j})}$ is transitive for all $j\in\{1,\dots,s\}$,

  \item \label{Adef2:weakQuord} a \emph{$d_{\alpha}$-dimensional weak
      generalized quasiorder} if $\rho$ is internally reflexive and transitive.

  \item \label{Adef2:reflexive}\emph{reflexive} if it is internally
    reflexive and additionally contains every constant function
    $g_{a}:(x_0, \dots, x_{d-1})\mapsto a$ for $a \in A$, 

  \item \label{Adef2:Quord} a \emph{$d_{\alpha}$-dimensional
      generalized quasiorder} if $\rho$
    is reflexive and transitive. 
  \end{enumerate}

\end{definition}
\

\begin{example}\label{ex:easycompoundarityexample}

  Here we provide a basic example of a $2$-dimensional
  generalized quasiorder of compound arity over $A = \{0,1,2\}$ (the
  arity is taken to 
  be $\alpha = \alpha_1 \times \alpha_2 = 3^1 \times 2^1$).

Our
  example is pp-defined from two $1$-dimensional generalized
  quasiorders $\rho_1$ and $\rho_2$, which are defined as follows. Let
  $\theta$ be the equivalence relation on $A$ with nontrivial class
  $\{0,1\}$ and let $\leq$ be the usual total order on $A$. We define

  \begin{enumerate}[label=\textup{(\roman*)}]
  \item $\rho_1: = \Pola[1]\theta$, and
  \item $\rho_2 := \{ (x,y) \mid x \leq y \} $.
  
  \end{enumerate}
  Obviously, both $\rho_1$ and $\theta$ are $1$-dimensional
  generalized quasiorders. It follows from the results in \cite{JakPR2024} that the unary fragment
  $\Pola[1]\theta$ of $\Pol\theta$ is a $1$-dimensional generalized quasiorder.

  We now define $ \rho \subseteq A^{3^1 \times 2^1} $ as
  \[
    \rho:= \{ g \in A^{3^1 \times 2^1}\mid g^{(3^1)} \in (\rho_2)^{3^1} 
    \text{ and } g^{(2^1)} \in (\rho_1)^{2^1} \}.
  \]
  So, $\rho$ is the relation consisting of those $3^1 \times 2^1$
  rectangles over $A$ in which all rows determine triples belonging to
  $\rho_1$ and all columns determine pairs belonging to $\rho_2$. We
  illustrate this at the top of Figure~\ref{fig:compoundGQuord}, where
  pictures of both $\rho^{(\alpha_1)}$ and $\rho^{(\alpha_2)}$ are
  superimposed on a picture of a typical element $g\in \rho$.

\begin{figure}
  \centering \includegraphics[width=1\linewidth]{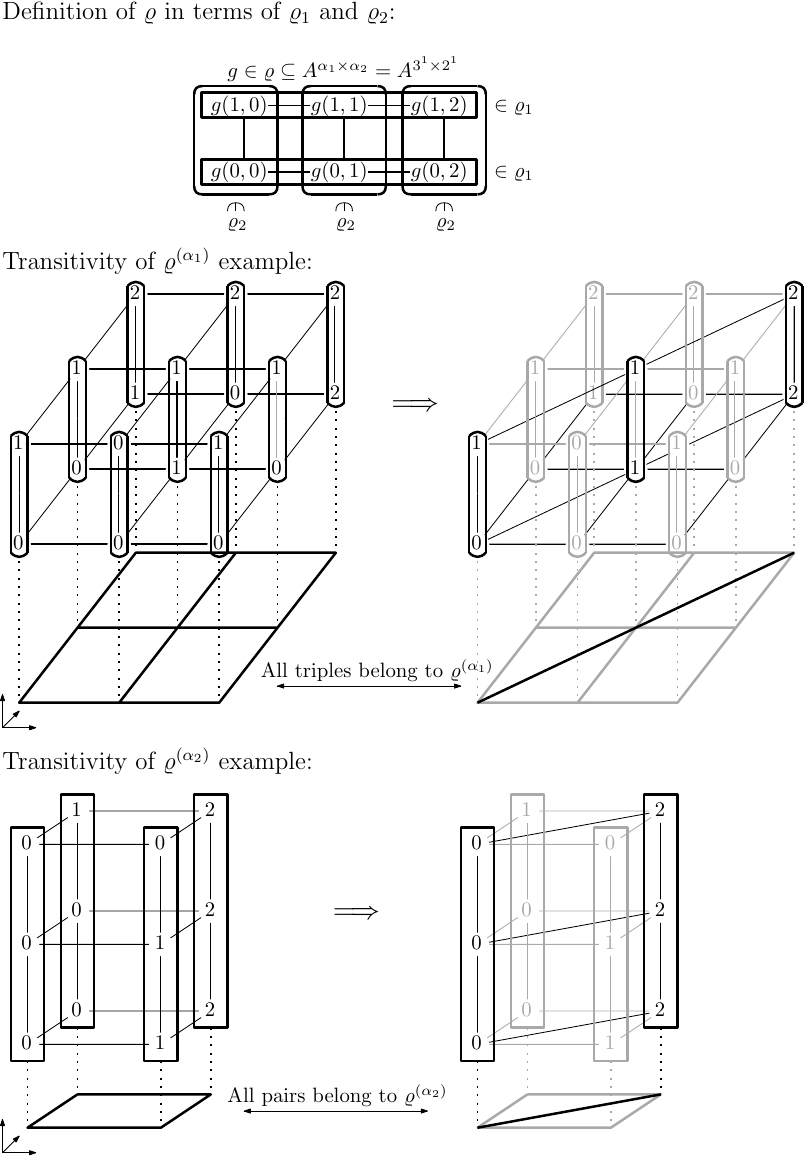}
  \caption{Compound arity $\alpha = 3^1 \times 2^1$ generalized
    quasiorder $\rho$ (see Example~\ref{ex:easycompoundarityexample})
  }
  \label{fig:compoundGQuord}
\end{figure}
We also include in Figure~\ref{fig:compoundGQuord} two verifications
of the implications that must be checked to demonstrate that $\rho$ is
transitive. To check that $\rho^{(\alpha_1)}$ is transitive, we need
to see that
$ \bDelta \ostar (\rho^{(\alpha_1)}) \subseteq \rho^{(\alpha_1)}. $ In
the middle of the figure on the left, we show a typical element that
belongs to $\ostar (\rho^{(\alpha_1)})$, i.e.\ a function which takes
inputs from $3^2$, outputs elements of $2^1$, and also
satisfies the property that any row or column determine an element of
$\rho^{(\alpha_1)}$. Axes are drawn which orient the coordinates:
$3^1$ increases to the right, $3^1$ increases into the page, and $2^1$
increases upwards. The square $3^2$ is drawn below and the outputs of
vertices (elements of $\rho_2 \subseteq A^{2^1}$) are indicated with
dotted lines. We also draw lines that indicate that this function can
be equivalently thought of as belonging to the R-relation
$A^{3^2 \times 2^1}$ (under the inverse of the canonical isomorphism
$(3^2)$ from $A^{3^2 \times 2^1}$ to $(A^{2^1})^{3^2}$). Transitivity
of $\rho$ means that we can identify the two variables of this
function and obtain an element of $\rho^{(\alpha_1)}$ and this is shown to
the right of the implication arrow.

The transitivity of $\rho^{(\alpha_2)}$ example is shown at the
bottom of Figure~\ref{fig:compoundGQuord} and it works
similarly. Since this example is mainly provided for building
intuition for compound arities, we leave the remaining details of the
verification that $\rho$ really is a generalized quasiorder to the
reader.
\end{example}

\begin{lemma}\label{lem:dependence}
 Let $\alpha=\alpha_{1}\times\ldots\times\alpha_{s}=B_1^{d_1}\times \dots \times B_s^{d_s}$. Suppose that
  $g \in A^{\alpha}$ is constant at $\alpha_j$, for some
  $j \in \{1, \dots, s\}$, and let $g'\in A^{\alpha}$  such that
  $(g')^{(\alpha_{\ell})} \in \trl_{d_{\ell}}(g^{(\alpha_\ell)})$ for
  some $ \ell \in \{1, \dots, s\}$. Then $g'$ also is constant at 
  $\alpha_{j}$.
\end{lemma}

\begin{proof}
The idea underlying the lemma is simple: the property being constant
at $\alpha_j$ cannot be destroyed by translations.

Using the notation of \ref{mu-translation} (with $\alpha=\beta=\alpha_{\ell}=B_{\ell}^{d_{\ell}}$), for
$(g')^{(\alpha_{\ell})} \in \trl_{d_{\ell}}(g^{(\alpha_\ell)})$ there
exists a suitable mapping $\mu:\alpha_{\ell}\to\alpha_{\ell}$ such that
$(g')^{(\alpha_{\ell})}(X^{(\alpha_{\ell})})=g^{(\alpha_{\ell})}(\mu(X^{a_{\ell})}))$
(for $X^{(\alpha_{\ell})}\in \alpha_{\ell}$)
and we have
\begin{align*}
  g'(X^{(\alpha_{1})},\dots,&X^{(\alpha_{\ell})},\dots,X^{(\alpha_{s})})\\&=
  (g')^{(\alpha_{\ell})}(X^{(\alpha_{\ell})})(X^{(\alpha_{1})},\dots,X^{(\alpha_{\ell-1})},X^{\alpha_{(\ell+1})},\dots,X^{(\alpha_{s})})\\ &=
  (g)^{(\alpha_{\ell})}(\mu(X^{(\alpha_{\ell})}))(X^{(\alpha_{1})},\dots,X^{(\alpha_{\ell-1})},X^{(\alpha_{\ell+1})},\dots,X^{(\alpha_{s})})\\&=
g(X^{(\alpha_{1})},\dots,X^{(\alpha_{\ell-1})},\mu(X^{(\alpha_{\ell})}),X^{(\alpha_{\ell+1})},\dots,X^{(\alpha_{s})}).
\end{align*}
Consequently, if $g$ is constant at $\alpha_{j}$ (i.e., the $j$-th block
$X^{(\alpha_{j})}$ of variables can be substituted by an arbitrary tuple
$C\in\alpha_{j}$) then also $g'$ has the same property, i.e., it
also is constant at $\alpha_{j}$.
\end{proof}

\begin{lemma} \label{lem:WeaklyReflexiveContainsImage}
  Let $\rho \subseteq A^{\alpha}$ be an internally reflexive
  R-relation of compound arity
  $\alpha=\alpha_{1}\times\ldots\times\alpha_{s} = B_1^{d_1}\times
  \dots \times B_s^{d_s}$. Then $\rho$ contains all constant maps
  $g_{a}:\alpha\to A:X^{(\alpha)}\mapsto a$
   for $a\in \Ima\rho$ {\rm(cf.\
  Remark~\ref{rm:equivalentreflexivity})}.
\end{lemma}

\begin{proof}
  Pick $a \in \Ima \rho$. Thus there exists a function $g \in \rho$ and
  some particular constant tuple 
 $C^{(\alpha)}=(C^{(\alpha_{1})},\dots,C^{(\alpha_{s})}) \in \alpha$ such that
  $g(C^{(\alpha)}) = a$. By~\ref{Adef2:internally-reflexive} of
  Definition~\ref{Adef2:Refl+Trans}, we have that $\rho^{(\alpha_j)}$
  is an internally reflexive R-relation (of simple arity) over
  $A^{\alpha / \alpha_j}$. We noticed in
  Remark~\ref{rm:equivalentreflexivity} that internally
  reflexive relations of simple arity contain for every constant in
  their image a corresponding constant function. Hence, the function
  $(g_j)^{(\alpha_{j})}\in (A^{\alpha/\alpha_{j}})^{\alpha_{j}}$ which is defined by
  \begin{align*}
        &(g_j)^{(\alpha_j)} (X^{(\alpha_j)}) := g^{(\alpha_j)}
    (C^{(\alpha_j)})
  \end{align*}
   belongs to $\rho^{(\alpha_{j})}$ (since 
$g^{(\alpha_j)}(C^{(\alpha_j)})\in\Ima\rho^{(\alpha_{j})}$). Consequently (by
   Lemma~\ref{lem:1forThm}\ref{1forThm-i}) we have $g_{j}\in\rho$,
   where $g_{j}$ by definition is of the form
   \begin{align*}
     g_j(X^{(\alpha_{1})},\dots,X^{(\alpha_{s})})=g(X^{(\alpha_{1})}, 
     \dots,X^{(\alpha_{d-1})},C^{(\alpha_{j})},X^{(\alpha_{j+1})},
                    \dots,X^{(\alpha_{s})}).
   \end{align*}

 We perform this substitution for every factor of
  $\alpha$ to obtain a function $g_{1, \dots, s}$, which is evidently
  given by
   $g_{1, \dots, s} (X^{(\alpha_1)}, \dots, X^{(\alpha_s)}) :=
    g(C^{(\alpha_1)}, \dots, C^{(\alpha_s)}) = a$.
\end{proof}

We expand the notation
introduced in Definition~\ref{def:SigmaofTuples} so that it can also
handle R-relations of compound arity.
\begin{definition}\label{def:SigmaofTuples2}
 Let $f:A^{n}\to A$, let $(g_{0},\dots,g_{n-1}) \in \left(A^{\alpha}\right)^n$, for $\alpha = B_1^{d}\times \dots \times B_s^{d_s}$, and let $j \in \{1, \dots, s\}$. We define
\[
  \Sigma_j(g_0, \dots, g_{n-1}) := \{ (g_0', \dots, g_{n-1}')\mid
  ((g_0')^{(\alpha_j)}, \dots, (g_{n-1}')^{(\alpha_j)}) \in
  \Sigma(g_0^{(\alpha_j)}, \dots, g_{n-1}^{(\alpha_j)}) \}.
\]
\end{definition}

With this notation, we now reformulate Lemma~\ref{lem:PolofGQuord} so
that it applies to a generalized
quasiorder of any arity (simple or compound).
\begin{lemma}\label{lem:PolofGQuordCompound}
  Let $\rho\in\wGQuorda[\alpha](A)$ for
  $\alpha=B_1^{d_1}\times \dots \times B_s^{d_s}$ and let $f:A^{n}\to A$. For
  each $j \in \{1, \dots, s \}$, the following hold.
 
  \begin{enumerate}[label=\textup{(\roman*)}]
  \item\label{compound-i} If $ (g_0', \dots, g_{n-1}') \in \Sigma_j (g_0, \dots, g_{n-1})$,
    then at most $d_j$-many of the $g'_0, \dots, g'_{n-1}$ depend on
    (are nonconstant at) $\alpha_j$.
  \item\label{compound-ii} $f(\Sigma_j(g_{0},\dots,g_{n-1}))\subseteq\rho
               \implies f(g_{0},\dots,g_{n-1})\in\rho.$
  \end{enumerate}
\end{lemma}

\begin{proof}
  Before we embark on the proof, we point out the main idea. By
  definition, $\rho^{(\alpha_j)}$ is a $d_j$-dimensional weak
  generalized quasiorder over $A^{\alpha/ \alpha_j }$, so in
  particular Lemma~\ref{lem:PolofGQuord} applies. Both statements
  follow from transforming the hypothesis of the current lemma into
  hypothesis formulated for $\rho^{(\alpha_j)}$ and then transforming
  our conclusions into the desired conclusions for $\rho$.

  So, let us prove \ref{compound-i}. Suppose that
  $ (g_0', \dots, g_{n-1}') \in \Sigma_j (g_0, \dots, g_{n-1}), $
  which by the definition of $\Sigma_j(g_0, \dots, g_{n-1})$ is
  equivalent to
  \[
    ((g_0')^{(\alpha_j)}, \dots, (g_{n-1}')^{(\alpha_j)}) \in \Sigma
    (g_0^{(\alpha_j)}, \dots, g_{n-1}^{(\alpha_j)}).
  \]
  By definition of
  $\Sigma (g_0^{(\alpha_j)}, \dots, g_{n-1}^{(\alpha_j)})$, at most
  $d_j$-many of the
  $(g_0')^{(\alpha_j)}, \dots, (g_{n-1}')^{(\alpha_j)}$ are
  nonconstant, equivalently (cf.\ Lemma~\ref{lem:1forThm}\ref{1forThm-v}), at most $d_j$-many of the
  $g_0', \dots, g_{n-1}'$ depend on $\alpha_j$.

  Now we prove \ref{compound-ii}. Suppose that
  \[
    f(\Sigma_j(g_0, \dots, g_{n-1})) \subseteq \rho.
  \]
  By definition of $\Sigma_j(g_0, \dots, g_{n-1})$ and \ref{1forThm-iii} of
  Lemma~\ref{lem:1forThm}, this is equivalent to
  \[
    f^{(\alpha_j)}(\Sigma(g_0^{(\alpha_j)}, \dots,
    g_{n-1}^{(\alpha_j)})) \subseteq \rho^{(\alpha_j)}.
  \]

  By definition of a generalized quasiorder of compound
  arity,
  $\rho^{(\alpha_j)}$ is a  $d_j$-dimen\-sional weak
  generalized quasiorder of simple arity $\alpha_{j}$ over $A^{\alpha/ \alpha_j }$, so
  Lemma~\ref{lem:PolofGQuord} applies. We conclude that
  \[
    f^{(\alpha_j)}(g_0^{(\alpha_j)}, \dots, g_{n-1}^{(\alpha_j)})
    \subseteq \rho^{(\alpha_j)}.
  \]
  Applying Lemma~\ref{lem:1forThm}\ref{1forThm-iii} finishes the proof.
\end{proof}

\begin{theorem}\label{thm:PolofGQuordMain2}
  Let $\rho\subseteq A^{\alpha}$ be a $d$-dimensional weak generalized
  quasiorder of rectangular arity
  $\alpha=\alpha_{1}\times\ldots\times\alpha_{s}=B_{1}^{d_{1}}\times\ldots\times
  B_{s}^{d_{s}}$ ($d=d_{1}+\ldots+d_{s}$). Then $\wXi_{d}(\rho)$
  holds. If $\rho$ is a generalized quasiorder the we have $\Xi_{d}(\rho)$.
\end{theorem}

\begin{proof} Let $\rho$ be a weak generalized quasiorder.
  Fix $n \in \N$ and let $f:A^{n}\to A$. We have to show
   $f\preserves\rho\iff \trl^{\Ima\rho}_{d}(f)\preserves\rho$ (cf.~\ref{weakXi-3}).
The implication ``$\Longrightarrow$'' is trivially
  fulfilled, in fact 
$\trl^{\Ima\rho}_{d}(f)\subseteq\Sg{\{f\}\cup \cC_{\Ima\rho}}\subseteq\Pol\rho$ 
 since internal reflexivity of $\rho$ implies that each constant
 $c_{a}$ with $a\in\Ima\rho$ is a
  polymorphism of $\rho$ (cf.~Lemma~\ref{lem:WeaklyReflexiveContainsImage}).





The crucial part of the proof is the implication
  ``$\Longleftarrow$''. Thus suppose 
  $\trl_d^{\Ima\rho}(f) \preserves \rho$. We want to show $f\preserves\rho$,
  i.e., $f(T_{0})\subseteq\rho$ for $T_{0}=\rho^{n}$.

We analyze a decreasing sequence of subsets 
  \[
    \rho^n = T_0 \supseteq T_1 \supseteq \dots \supseteq T_s,
  \]
  which are defined as follows. Let $\Phi_0$ be $\top$ (true) and
  $\Phi_j(g_0, \dots, g_{n-1})$ be the statement given by:
  
  \begin{align*}
    \Phi_j(g_0, \dots, g_{n-1}):\iff&
    \textit{at most $d_j$-many of the $g_0, \dots, g_{n-1}$ depend on
      $\alpha_j$}\\
&\textit{(at least $(n-d_j)$-many are constant at
      $\alpha_j$)}.
  \end{align*}

  We then set
  \[
    T_j:= \{ (g_0, \dots, g_{n-1}) \in \rho^n\mid \bigwedge_{0 \leq \ell
      \leq j} \Phi_\ell(g_0, \dots, g_{n-1})\}.
  \]

  
We prove that $f(T_j) \subseteq \rho$ for all $j \in \{0, \dots, s\}$
  by induction (in the reverse order) on $j$ which will finish the
  proof with $j=0$.
  
  \underline{Induction basis ($j=s$):} Let us establish that
  $f(T_s) \subseteq \rho$. 

Let $(g_0, \dots, g_{n-1}) \in T_s$. Below we are going to show that
there are at least $n-d$ many $g_{i}$'s which are constant (with
constants from $\Ima\rho$ since $g_{i}\in\rho$). Thus 
there exists $f'\in\trl_{d}^{\Ima\rho}(f)$ such that
$f(g_{0},\dots,g_{n-1})=f'(g_{j_{0}},\dots,g_{j_{d-1}})$ (for some
$j_{0},\dots,j_{d-1}\in\{0,\dots,n-1\}$ with possibly nonconstant $g_{j}$) which implies
$f(g_{0},\dots,g_{n-1})\in\rho$ since, by assumption, $f'\preserves\rho$ and
$g_{0},\dots,g_{n-1}\in\rho$, and we are done.

In fact,
since, by definition of $T_{s}$, at most $d_{\ell}$ many $g_{i}$'s
depend on $\alpha_{\ell}$ for 
each $\ell\in\{1,\dots,s\}$, there are at most $d=d_{1}+\ldots+d_{s}$
many $g_{i}$ ($i\in\{0,\dots,n-1\}$) which may depend on at least one
$\alpha_{\ell}$. Consequently, at least $n-d$ many of the $g_{i}$'s
are constant at each $\alpha_{\ell}$ and therefore constant at all, which was
to be shown.

  \underline{Induction step (from $j+1$ to $j$):} Suppose that
  $f(T_{j+1}) \subseteq \rho$, for $j \in \{0, \dots, s-1\}$. We want
  to see that $f(T_j) \subseteq \rho$ also. So, let
  $(g_0, \dots, g_{n-1}) \in T_j$. We have to show  that
  $f(g_0, \dots, g_{n-1}) \in \rho$. By
  Lemma~\ref{lem:PolofGQuordCompound}\ref{compound-ii}, 
it is enough to show that
  $f(\Sigma_{j+1}(g_{0},\dots,g_{n-1}))\subseteq\rho$. 
Since $f(T_{j+1})\subseteq\rho$ by induction hypothesis, we are done
if we can prove  $\Sigma_{j+1}(g_0, \dots, g_{n-1}) \subseteq
T_{j+1}$. 

Thus let $(g_0', \dots, g_{n-1}') \in \Sigma_{j+1}(g_0, \dots,
g_{n-1})$. We must show $(g_0', \dots, g_{n-1}') \in T_{j+1}$, i.e.,
we have to show that both
\[
(g_0', \dots, g_{n-1}') \in \rho^n
\qquad
\text{and}
\qquad
\bigwedge_{0 \leq \ell \leq j+1} \Phi_{\ell}(g'_0,\dots, g'_{n-1})
\] 
hold.

Let us first show that $(g_0', \dots, g_{n-1}') \in \rho^{n}$. 
Because $(g_0', \dots, g_{n-1}')\in\Sigma_{j+1}(g_0, \dots, g_{n-1})$,
we have, by Definition~\ref{def:SigmaofTuples2}, 
$((g_0')^{(\alpha_{j+1})}, \dots, (g_{n-1}')^{(\alpha_{j+1})}) \in
  \Sigma(g_0^{(\alpha_{j+1})}, \dots, g_{n-1}^{(\alpha_{j+1})})$, 
which, by Definition~\ref{def:SigmaofTuples}, is equivalent to 
\[
(g_0')^{(\alpha_{j+1})} \in \trl_{d_{j+1}} \{  g_0^{(\alpha_{j+1})}\},
\dots, 
(g_{n-1}')^{(\alpha_{j+1})} \in \trl_{d_{j+1}} \{ g_{n-1}^{(\alpha_{j+1})}\}.
\]
Since $\rho$ is internally reflexive, we have by
Definition~\ref{Adef2:Refl+Trans}\ref{Adef2:internally-reflexive} that
$\rho^{(\alpha_{j+1})}$ is internally reflexive. We apply
Lemma~\ref{lem:adjointproperty}\ref{adj-i} and conclude that
$\trl_{d_{j+1}}\rho^{(\alpha_{j+1})} \subseteq
\rho^{(\alpha_{j+1})}$. Each of the $g_i^{(\alpha_{j+1})}$ belongs to
$\rho^{(\alpha_{j+1})}$, so we conclude that each of the
$(g_i')^{(\alpha_{j+1})}$ also belongs to $\rho^{(\alpha_{j+1})}$,
equivalently (cf.~\ref{lem:1forThm}\ref{1forThm-i}), that 
$(g_0',\dots, g_{n-1}') \in \rho^n$. 

Now we will show that
$\bigwedge_{0 \leq \ell \leq j+1} \Phi_{\ell}(g'_0,\dots, g'_{n-1})$ holds,
which will finish the proof. Indeed, $(g_0, \dots, g_{n-1}) \in T_j$,
so it satisfies $\Phi_{\ell}$ for all $ \ell \in \{1, \dots, j\}$. Fix
a particular such $\ell$ and suppose that $g_i$ is constant at
$\alpha_\ell$ (recall that by definition of $T_j$, there are at least
$(n-d_\ell)$-many such $g_i$). Then, for
$(g_0', \dots, g_{n-1}') \in \Sigma_{j+1}(g_0, \dots, g_{n-1})$,
Lemma~\ref{lem:dependence} guarantees that $g_i'$ also is constant
at $\alpha_j$, since by definition of $\Sigma_{j+1}$, we know that
$(g'_i)^{(\alpha_{j+1})} \in \trl_{d_{j+1}}\{g_i^{(\alpha_{j+1})}\}$. 
Since there are at least $(n-d_\ell)$-many of the $g_i$ that do
not depend on $\alpha_\ell$, we conclude there are at least
$(n-d_\ell)$-many of the $g_i'$ that do not depend on $\alpha_\ell$,
which proves that $\Phi_\ell(g_0', \dots, g_{n-1}') $ holds. This
argument works for every $\ell \in \{1, \dots, j \}$, so we conclude
that $\bigwedge_{0 \leq \ell \leq j} \Phi_\ell(g'_0, \dots, g'_{n-1})$
holds. Since it follows from
Lemma~\ref{lem:PolofGQuordCompound}\ref{compound-i} that $(g_0',
\dots, g_{n-1}')$ also satisfies $\Phi_{j+1}$, we are done.
\end{proof}

As a direct consequence of Theorem~\ref{thm:PolofGQuordMain2} we get
($\dGQuord$ denotes the $d$-dimensional generalized quasiorders, which
are defined in a formally restricted sense in the next section):
\begin{corollary}\label{A1c}
Let $F\subseteq\Op(A)$ and $Q\subseteq\dGQuord(A)$. Then we have:
  \begin{enumerate}[label=\textup{(\roman*)}]
  \item\label{A1ci} $\dGQuord(A,F)=\dGQuord(A,\trl_{d}F)$.
  \item\label{A1cii} $\Xi_{d}(Q)$ holds, i.e., $\Pol Q=(\Pola[d] Q)^{*_{d}}$,
    in particular, 
    $(\Pola[d] Q)^{*_{d}}$ is a 
    clone and $\Pola[d] Q$ is \dR-closed.
  \end{enumerate}
\end{corollary}

\begin{proof}
\ref{A1ci} directly follows from
Theorem~\ref{thm:PolofGQuordMain2}. Concerning \ref{A1cii}, 
we have 

$\Pol Q=\bigcap_{\rho\in Q}\Pol\rho
=_{\ref{thm:PolofGQuordMain2},\ref{def:Xi}\ref{Xi3}}\bigcap_{\rho\in Q}(\Pola[d]\rho)^{*_{d}}
=_{\ref{rem:MstarPoe}}(\bigcap_{\rho\in Q}\Pola[d]\rho)^{*_{d}}
=(\Pola[d] Q)^{*_{d}}$, i.e., $\Xi_{d}(Q)$. 
\end{proof}


\section{The Galois connection $\boldsymbol{\Pola[d]-\dGQuord}$
}\label{sec:PolQuord}

In this section we consider the Galois connection $\Pola[d]-\dGQuord$
($\Pol-\dGQuord$, resp.), characterize the Galois closure $\Pola[d]\dGQuord
M$ ($\Pol\dGQuord M$, resp.) for $M\subseteq A^{A^{d}}$
(Theorem~\ref{thm:dPolQuord}, Corollary~\ref{A3cor}\ref{A3a}), and answer
the question which $d$-ary clone fragments are \dR-closed from the
point of view of operations (\ref{A3cor}\ref{A3b}) as well as of
relations (\ref{A3cor}\ref{A3c}). Moreover, we characterize sets
$Q\subseteq\Rel(A)$ 
of relations satisfying our motivating property $\Xi_{d}$ (\ref{A3Xi}).

Recall that $|\alpha|$ is the relational arity of
$\rho\in\RRela[\alpha](A)$ (cf.~\ref{def:Rrelation}). In order to keep the cardinality of R-relations with
bounded relational arity finite (for finite
$A$) we can and shall restrict to ``canonical'' rectangular aritities of the
form $\alpha=m_{1}^{d_{1}}\times\ldots\times m_{s}^{d_{s}}$ with
$d_{i},m_{i}\in \N_{+}$ (where the ordinals $m_{i}$ are considered as
sets $m_{i}=\{0,1,\dots,m_{i}-1\}$, $i\in\{1,\dots,s\}$), as mentioned
already in Remark~\ref{rem:Rrelation}, i.e., we define
\begin{align*}
   \dRRel(A)&:=\{\rho\in\RRela[\alpha](A)\mid \alpha\text{
     canonical and } d_{\alpha}=d\},\\
  \RRel(A)&:=\bigcup_{d\in\N_{+}}\dRRel(A)
   .
\end{align*}
If we restrict to generalized quasiorders then we use the corresponding notation
$\dGQuord(A)$ and $\GQuord(A)$.

Moreover, if we consider a set $M\subseteq \Opa[d](A)=A^{A^{d}}$ of
$d$-ary operations (e.g., a $d$-ary clone fragment) as $d$-dimensional
R-relation of simple arity $\alpha=A^{d}$ then it is convenient to
think of $A$ as an ordinal (to make $\alpha$ canonical), i.e., to
identify it with $|A|$, thus $M\in\dRRel(A)$ according to the above
definition.

It is a well-known fact that any $d$-ary operation which preserves a
$d$-ary clone fragment (considered as R-relation) must belong to it
(apply the operation to the 
projections). We explicitly want to mention this here:

\begin{lemma}\label{dPol}
  Let $M\leq A^{A^{d}}$ be a $d$-ary clone fragment. Then  
$M=\Pola[d]M$ (where on the right-hand side $M$ is
considered as R-relation in $\RRela[A^{d}](A)$).\qed
\end{lemma}

The following proposition for clones is the (extended) analogue of
Proposition~\ref{prop:doublearrowisminion} for minions:

\begin{proposition}\label{prop:StarisCloneIffGQuord}

  Let $M \subseteq \Op^{(d)}(A)$
  be a $d$-ary clone fragment ($M\leq A^{A^d}$). 
 The following are equivalent:
  \begin{enumerate}[label=\textup{(\roman*)}]
  \item \label{StarisClone-i} $M^{*_d}$ is a clone.
  \item \label{StarisClone-ii} $M$  considered as an $R$-relation of
    simple arity $\alpha= A^d$ is a $d$-dimensional generalized
    quasiorder.
  \item \label{StarisClone-iii} $M^{*_{d}}=\Pol M$ (for the right side,
    $M$ is considered as R-relation as in \ref{StarisClone-ii}).
  \end{enumerate}
  \begin{remarknote}
    Condition \ref{StarisClone-ii} also implies that 
    $(M^{*_{d}})^{(n)}$ is an $n$-dimensional
    generalized quasiorder for all $n\geq d$. 
  \end{remarknote}
\end{proposition}

\begin{proof} \ref{StarisClone-i}$\Rightarrow$\ref{StarisClone-ii}:
Since $M$ is a clone fragment we have  $\cC^{(d)}\subseteq M$  
  by Remark \ref{rem:MstarPoe}(B). From
  Lemma~\ref{lem:clonesGiveReflexive} we can conclude that $M$ is
  reflexive and thus $M^{\leftrightarrow}=M^{*_{d}}$ by Proposition~\ref{prop:starcharacterization}.
 Assume $M^{*_{d}}$ is a clone.
  Clones are special minions, therefore we can apply
  Proposition~\ref{prop:doublearrowisminion} (with $\rho=M$) to 
   $M$, consequently $M$ is a weak generalized quasiorder.
  Since $M$ is reflexive we get $M\in\GQuord(A)$.

\ref{StarisClone-ii}$\Rightarrow$\ref{StarisClone-iii}:
Suppose that $M \in\GQuorda[\alpha](A)$. We apply
Theorem~\ref{thm:PolofGQuordMain2} (for the case $s=1$) and conclude
that $\Xi_d(M)$ holds, equivalently, that $ \Pol M = (\Pola[d]M)^{*_d} $ (see
Definition~\ref{def:Xi}). Since $\Pola[d]M= M$ by Lemma~\ref{dPol}, we
are done.

\ref{StarisClone-iii}$\Rightarrow$\ref{StarisClone-i} is trivial since
$\Pol Q$ is always a clone for an arbitrary set $Q$ of relations.

The remark follows from
Lemma~\ref{lm:adjointproperties}\ref{adjointprops3} because
$(M^{*_{d}})^{(n)}=\ostar^{n-d}(M)$ by
Proposition~\ref{prop:starcharacterization}.
\end{proof}

\begin{theorem}\label{thm:dPolQuord} Let $M\subseteq A^{A^{d}}$. Then
  we have:
  \begin{align*}
    \Rclose[d]{M}=\Pola[d]\dGQuord M.
  \end{align*}
\end{theorem}

\begin{proof} 
At first we observe $M\subseteq\Pola[d]\GQuord M$
(this holds for every Galois connection, cf.~\ref{Galoisconnections}(3)(III)). Moreover, by
Proposition~\ref{prop:StarisCloneIffGQuord}, we have
$\Rclose{M}\in\GQuord(A)$ since $N^{*_{d}}$ is a clone for $N=\Rclose
M$ (cf.\ Definition~\ref{def:Rclosure} and Remark~\ref{rem:Rclosure}).
Further, 
$M\subseteq\Rclose{M}=_{\text{\ref{dPol}}}\Pola[d] \Rclose{M}$, in particular
$M\preserves \Rclose{M}$ (here $\Rclose{M}$ considered as R-relation
of arity $\alpha=A^{d}$).
Consequently we get
$\Rclose{M}\subseteq\Rclose{\Pola[d]\GQuord M}
=_{\ref{A1c}\ref{A1cii}}\Pola[d]\GQuord M\subseteq\Pola[d]{\Rclose{M}}
=\Rclose{M}$, and we are done.
 \end{proof}

In the following Corollary \ref{A3cor}\ref{A3c} we need to apply pp-formulas to
R-relations (namely for using the relational clone closure
$[Q]_{(\exists,\land,=)}$). Therefore we want to treat them as
``usual'' relations (cf.\ \ref{rem:Rrelation}).
This can be achieved just by ``forgetting'' the rectangular structure
(switching from the rectangular arity to a relational arity).
Formally, for $\rho\in\RRela[\alpha](A)$, let 
$\mathring\rho:=\rho\circ\eta\in\Rela[m](A)$ be the $m$-ary
relation ($m:=|\alpha|$) obtained from $\rho$ via an arbitrary
(but then fixed) bijection $\eta:|\alpha|\to\alpha$ as discussed in
Section~\ref{sec:Rrelations}. Clearly, then we can define $\mathring
Q:=\{\mathring\rho\mid\rho\in Q\}$ for $Q\subseteq\RRel(A)$. Note that
$\Pol\rho=\Pol\mathring\rho$ by definition of preservation.

\begin{corollary}\label{A3cor}

  \begin{enumerate}[label=\textup{(\alph*)}]
  \item\label{A3a} $(\Rclose{M})^{*_{d}}=\Pol\dGQuord M$ for $M\subseteq
    A^{A^{d}}$.

  \item\label{A3b} 
   The following are equivalent for a clone fragment $M\leq A^{A^{d}}$:
    \begin{itemize}
    \item[\rm(i)] $M$ is \dR-closed,
   \hfill {\rm(i)$'$} $M^{*_{d}}$ is a clone,
   \hfill {\rm(i)$''$} $M\in\dGQuord(A)$,
    \item[\rm(ii)] $M=\Pola[d] Q$ for some $Q\subseteq\dGQuord(A)$,
    \item[\rm(iii)] $M^{*_{d}}=\Pol Q$ for some $Q\subseteq\dGQuord(A)$,
    \end{itemize}
where the same $Q$ can be taken in {\rm(ii)} and {\rm(iii)}.

  \item\label{A3c}
   The following are equivalent for $Q\subseteq\Rel(A)$:
    \begin{itemize}
    \item[\rm(i)] $\Pola[d] Q$ is \dR-closed,
     {\rm(i)$'$ } \hspace*{-1,5ex} $(\Pola[d] Q)^{*}$ is a clone,
      {\rm(i)$''$ } \hspace*{-1.5ex} ${\Pola[d] Q}\in\dGQuord(A)$,
    \item[\rm(ii)]  $\exists
      Q'\subseteq\dGQuord(A): \Pola[d] Q=\Pola[d] Q'$, 
    \item[\rm(ii)$'$] 
 $\exists Q'\subseteq\gQuord(A):
      \dLOC{}[Q]_{\exists,\land,=}=\dLOC{}[\mathring Q']_{\exists,\lor,=}$\\ 
   (the closures under pp-formulas and $d$-directed unions coincide, cf.~\ref{rem:PolInv}),
    \item[\rm(iii)] $\exists Q'\subseteq\dGQuord(A): (\Pola[d] Q)^{*_{d}}=\Pol Q'$,
    \end{itemize}
where the same $Q'$ can be taken in {\rm(ii)} and {\rm(iii)}.
Instead of `` $\exists Q'\subseteq\dGQuord(A)$'' one can take
`` $\exists\rho\in\dGQuord(A)$'' and $Q'=\{\rho\}$.
 
\end{enumerate}
\end{corollary}

\begin{proof}
  \ref{A3a}: Let $Q:=\dGQuord M$. Then $\Xi_{d}(Q)$ by Theorem~\ref{thm:PolofGQuordMain2}, i.e.,
$\Pol Q=(\Pola[d] Q)^{*_{d}}$ (cf.\ Definition~\ref{def:Xi}\ref{Xi1}). 
Thus $\Pol Q=(\Pola[d]\dGQuord M)^{*}=(\Rclose[d]{M})^{*_{d}}$ by
Theorem~\ref{thm:dPolQuord}.

\ref{A3b}: For (i)$\iff$(i)$'$ $\iff$(i)$''$ see
Remark~\ref{rem:Rclosure} and
Proposition~\ref{prop:StarisCloneIffGQuord}.

(i)$\implies$(ii): Take $Q:=\dGQuord M$. If $M$ is $\dR$-closed, then
$M=\Rclose[d]{M}=_{\ref{thm:dPolQuord}}\Pola[d] Q$.

(ii)$\implies$(iii): $(\Pola[d] Q)^{*_{d}}=\Pol Q$ directly follows from
Corollary~\ref{A1c}\ref{A1cii}.

(iii)$\implies$(i)$'$ is obvious, because $M^{*_{d}}=\Pol Q$ is a clone.

\ref{A3c}: This is just \ref{A3b} for $M=\Pola[d] Q$. (ii)$\iff$(ii)$'$
follows from the properties of the Galois connection $\Pola[d]-\Inv$
(in particular $\dLOC[d][Q]_{\exists,\land,\lor,=}=\Inv\Pola[d] Q$,
cf.~\ref{rem:PolInv}).
Further note, that, instead of arbitrary $Q'\subseteq\dGQuord(A)$ we
can take the set $\{ M\}$ with the single R-relation $M=\Pola[d]Q'$ since 
$\Pol Q'=_{\ref{A1c}\ref{A1cii}}(\Pola[d]Q')^{*_{d}}=M^{*_{d}}=_{\ref{dPol}}(\Pola[d]M)^{*_{d}}=\Pol M$.
\end{proof}

Now we are also able to answer the question which (sets of) relations
satisfy the property $\Xi_{d}$ (cf.\ Definition~\ref{def:Xi}):

\begin{proposition}\label{A3Xi}  The following are equivalent for $Q\subseteq\Rel(A)$:
  \begin{itemize}
    \item[\rm(i)] $\Xi_{d}(Q)$ holds, i.e., $\Pol Q=(\Pola[d] Q)^{*_{d}}$,

    \item[\rm(ii)] $\exists Q'\subseteq\dGQuord(A): 
              \Pol Q=\Pol Q'$ $(=\Pol\mathring Q')$,

    \item[\rm(ii)$'$] $\exists Q'\subseteq\dGQuord(A):
      [Q]_{\exists,\land,=}=[\mathring Q']_{\exists,\land,=}$\\ 
   (the closures under primitive positive formulas coincide),
\item[\rm(iii)]

There exists a $d$-ary clone fragment $M\leq
  A^{A^{d}}$ containing all constants $\cC^{(d)}$ such that $M\in\dGQuord(A)$ and
  $[Q]_{\exists,\land,=}=[\mathring M]_{\exists,\land,=}$. 
    \end{itemize}
  \end{proposition}

  \begin{proof}
    (i)$\implies$(ii): Assume $\Pol Q=(\Pola[d] Q)^{*_{d}}$, let $M:=\Pola[d] Q$
  and $Q':=\dGQuord M$. $M$ is $\dR$-closed
  (since $M^{*_{d}}$ is a clone), therefore
  $\Pol Q=M^{*_{d}}=(\Rclose{M})^{*_{d}}=_{\ref{A3cor}\ref{A3a}}\Pol
  Q'$.

(ii)$\implies$(i): Assume $\Pol Q=\Pol Q'$ ($Q'\subseteq \GQuord(A)$).
Then $\Pola[d] Q=\Pola[d] Q'$ and we have 
$\Pol Q=\Pol Q'=_{\ref{A1c}\ref{A1cii}}(\Pola[d] Q')^{*_{d}}=(\Pola[d]
Q)^{*_{d}}$, consequently $\Xi_{d}(Q)$ by Definition~\ref{def:Xi}.

(ii)$\iff$(ii)$'$
follows from the properties of the classical Galois connection $\Pol-\Inv$
(in particular $[Q]_{\exists,\land,=}=\Inv\Pol Q$, cf.~Remark in
\ref{Galoisconnections}(3)).

(i)$\implies$(iii): Take $M:=\Pola[d]Q$. Then $M^{*_{d}}$ is a clone
by assumtion (i). Thus $\Pol M=M^{*_{d}}$ by 
Proposition~\ref{prop:StarisCloneIffGQuord}\ref{StarisClone-iii}.
Consequently $\Pol Q=M^{*_{d}}=\Pol M=\Pol \mathring M$, 
equivalently $[Q]_{\exists,\land,=}=[\mathring M]_{\exists,\land,=}$.

(iii)$\implies$(ii)' trivially holds.
  \end{proof}

\section{The $\boldsymbol{\GQuord}$-dimension of a clone}
\label{sec:GQuord-dimension}

In this section we go through several examples in order to illustrate
aspects of the theory and to highlight some of the distinctions it
allows us to make. Before we proceed to the examples, we make the
following definition, which is a strengthening of the notion `finitely
related' for a clone.

\begin{definition}\label{def:GQuordDim}
  Let $F$ be a clone (over a set $A$). We define the
  \New{$\GQuord$-dimension} of $F$ to be the minimum natural number
  $d$ such that $ F = (F^{(d)})^{*_d}. $ If no such natural number
  exists, we say that $F$ has infinite $\GQuord$-dimension.
\end{definition}

Although Definition~\ref{def:GQuordDim} makes sense for any function
clone, it is only interesting for clones $F$ which contain
$\cC$ (all constant operations). This is one of the
consequences of Proposition~\ref{prop:StarisCloneIffGQuord}.
\subsection*{(A) The Boolean case}\label{subsec:boolean}

\begin{example}\label{ex:Booleancase}

Here we examine Boolean clones. We will see that every
clone $F$ over a two-element set containing all constant operations
has $\GQuord$-dimension of either $1$ or $2$. It is well-known that
there are exactly seven distinct clones with constant operations on
the set $\{0,1\}$, so we just need to check that each is either $1$ or
$2$-dimensional. The seven clones, their generators, and dimensions
are listed below ($+$ denotes addition modulo 2).

\newcommand{\FM}{F_{\land\lor}} 
\newcommand{\FC}{F_{\text{c}}}  
\newcommand{\FU}{F_{\lnot}}     
\newcommand{\FL}{F_{+}}         
\newcommand{\FA}{F_{\land}}     
\newcommand{\FO}{F_{\lor}}      

\begin{center}
  \begin{tabular}{||c c c c ||}
    \hline
    Clone & Description & Generators & $\GQuord$-dimension  \\ [0.5ex] 
    \hline\hline
    $O_2$ & Full clone & $\wedge, \lnot$ & 1\\ 
    \hline
    $\FM$ & Monotone functions  & $\wedge, \vee, \cC $ & 1 \\ 
    \hline
    $\FC$ & Projections + Constants & $\cC$  & 2 \\
    \hline
    $\FU$ & Essentially unary  & $\lnot, \cC$  &  2\\
    \hline
    $\FL$ & Linear functions & $+, \cC$ & 2 \\
    \hline
    $\FA$ & Conjunctive functions  & $\wedge, \cC$ & 2 \\ 
    \hline
    $\FO$ & Disjunctive functions  & $\vee, \cC$ & 2 \\ 
    \hline
 
  \end{tabular}
\end{center}

In the following we shall present function tables in a slightly
unusual way if we consider functions as elements of a $d$-ary
``rectangular cube'' as in the figure in Example~\ref{RR3}. For a
unary function $g\in A^{A^{1}}$ we have a $1$-dimensional object
(line) and for a binary function $h\in A^{A^{2}}$ we get a
$2$-dimensional object (sqare), presented as follows, here for
$A=\{0,1\}$: 
\begin{center}
  $\begin{array}[c]{|cc}
     g(0)& g(1)\\\hline
     0& 1
   \end{array}$
   \quad and \quad
   $\begin{array}[c]{r|c c c}
      y=1&f(0,1)&f(1,1)\\
      0&f(0,0)&f(1,0)\\ 
      \cline{1-3}
      f(x,y)&0&1&=x
    \end{array}$
  \end{center}
i.e., each table is drawn upside down relative to what
  is customary, since we want the orientation to be compatible with
  the orientation we chose for our `geometrical' depiction of
  $3$-dimensional R-relations.

  We prove that the $\GQuord$-dimensions listed above are accurate by
  case analysis. We begin by examining all $1$-dimensional generalized
  quasiorders over $\{0,1\}$ which are also clone fragments. It is
  easy to see that there are exactly two: the full monoid over
  $\{0,1\}$, which is equal to $O_2^{(1)}$, and the monoid with two
  constant operations and the identity, which is equal to
  $\FM^{(1)}$. These are each listed in the table below.

\begin{center}
\begin{tabular}{||c c |c ||} 
 \hline 
 $1$-ary clone fragment $M$& Elements & $M^{*_1}$ \\ [0.5ex] 
 \hline\hline & & \\
 $\FM^{(1)}$ &  \begin{tabular}{ |cc } 
 0 & 1 \\\hline
 0 & 1  \\ 
 
\end{tabular},  \begin{tabular}{ |cc } 
 0 & 0 \\\hline
 0 & 1  \\ 

\end{tabular} , \begin{tabular}{ |cc } 
1 & 1 \\ \hline
 0 & 1  \\ 
 
\end{tabular} & $\FM$  \\[3ex] 
 \hline & & \\ 
 $O_2^{(1)}$ & \begin{tabular}{ |cc } 
0 & 1 \\ \hline
 0 & 1  \\ 
 
\end{tabular},
\begin{tabular}{ |cc } 
1 & 0 \\  \hline
 0 & 1  \\ 
 
\end{tabular},
\begin{tabular}{ |cc } 
 0 & 0 \\ \hline
 0 & 1  \\ 
 
\end{tabular},
\begin{tabular}{ |cc } 
 1 & 1 \\ \hline 
 0 & 1  \\ 
\end{tabular}
 & $O_2$   \\[3ex] 
 \hline
\end{tabular}
\end{center}
  Observe that both $\wedge$ and $\vee$ belong to
  $(\FM^{(1)})^{*_1}$. Since $\lnot \notin \FM^{(1)}$, it must be that
  $(\FM^{(1)})^{*_1} = \FM$, which shows that $\FM$ is a $1$-dimensional
  clone (since by Proposition~\ref{prop:StarisCloneIffGQuord}
  $(\FM^{(1)})^{*_1}$ is a clone and it is not $O_2$ since it does not
  contain unary negation). Obviously, $(O_2^{(1)})^{*_1} = O_2$. Since
  these are the only $1$-dimensional generalized quasiorders which are
  clone fragments, it follows that $\FM$ and $O_2$ are the only
  $1$-dimensional clones over $\{0,1\}$ (this also follows from
  \cite[Theorem~6.5]{JakPR2026}).

  To see that the other clones are $2$-dimensional, it is enough to
  see that their respective binary clone fragments $M$ are distinct
  generalized quasiorders, since then the respective distinct
  $M^{*_2}$ preclones (cf.\ Proposition~\ref{M6}) are all clones
  (cf.\ Proposition~\ref{prop:StarisCloneIffGQuord}). These binary clone
  fragments are all listed in the following table
(again elements are listed as
  operation tables) and it is obvious that they are distinct. 

  \begin{center}
    \begin{tabular}{||c| c |c ||}
      \hline 
     $2$-ary~clone &&\\ fragment $M$& {elements} & $M^{*_2}$ \\ [0.5ex] 
      \hline\hline & & \\ 
      $\FC^{(2)}$ &  \begin{tabular}[b]{ c|c c } 
                     1 & 0 & 0\\ 
                     0 & 0 & 0 \\ 
                     \hline
                     $c_0^{(2)}$ & x & 1\\
                   \end{tabular},
      \begin{tabular}[b]{ c|c c } 
        1 & 1 & 1\\ 
        0 & 1 & 1 \\  
        \hline
        $c_1^{(2)}$ & 0 & 1\\
      \end{tabular},
      \begin{tabular}[b]{ c|c c } 
        1 & 0 & 1\\  
        0 & 0 & 1 \\ 
        \hline
        $e_0^2 $& 0 & 1\\
      \end{tabular},
      \begin{tabular}[b]{ c|c c } 
        1 & 1 & 1 \\ 
        0 & 0 & 0\\  
        \hline
        $e_1^2$ & 0 & 1\\
 
      \end{tabular}

                              & $\FC$   \\[4ex] 
      \hline  & & \\
      $\FA^{(2)}$ & $ \FC^{(2)}$ \hspace{.1cm} and \hspace{.1cm} 
      \begin{tabular}[b]{ c|c c } 
        1 & 0 & 1\\   
        0 & 0 & 0 \\   \hline
        $\wedge$ & 0 & 1\\
       \end{tabular}   & $\FA$  \\[4ex] 

      \hline & & \\ 
      $\FO^{(2)}$ & $ \FC^{(2)}$ \hspace{.1cm} and
      \hspace{.1cm} \begin{tabular}[b]{ c|c c }1 & 1 & 1\\   
        0 & 0 & 1 \\  \hline
       $\vee$ & 0 & 1\\
        \end{tabular} & $\FO$   \\[4ex] 
      \hline  & & \\ 
      $\FU^{(2)}$ & $ \FC^{(2)}$ \hspace{.1cm} and
       \hspace{.1cm} \begin{tabular}[b]{ c|c c }  1 & 1 & 0\\ 0 & 1 & 0 \\ 
      \hline $\lnot x$& 0 & 1\\
       \end{tabular}, 
\begin{tabular}[b]{ c|c c } 
   1 & 0 & 0\\  0 & 1 & 1 \\  \hline $\lnot y$& 0 & 1\\                             \end{tabular} & $\FU$   \\[4ex] 
      \hline  & & \\ 
      $\FL^{(2)}$ & $\left\{ \begin{tabular}{ c|c c } 
                             1 & $z$ & $w$\\   
                             0 & $u$ & $v$\\ 
                             \hline
                               & 0 & 1\\ 
                            \end{tabular} 
               \;\middle|\; u + v + z + w \equiv 0 \right\}$
               & $\FL$   \\[4ex] 
      \hline  
    \end{tabular}
  \end{center} 

We show
  that $\FA^{(2)}$ is a generalized quasiorder and leave the
  verification of the others to the reader.
  Since all relations we consider in this example are clone fragments,
  they are all reflexive by Lemma~\ref{lem:clonesGiveReflexive}. Hence,
  we just need to check transitivity, in this case of
  $\FA^{(2)}$. Suppose that $g\in \ostar \FA^{(2)}$. We depict $g$ as a
  labeled cube as follows:
  \[
    g=
    \Cube[g(000)][g(001)][g(100)][g(101)][g(010)][g(011)][g(110)][g(111)][2],
  \]
  where the coordinates for $g(x,y,z)$ are given a `right-handed'
  orientation. The condition that $g \in \ostar \FA^{(2)}$ requires that
  each of its six faces belong to $\FA^{(2)}$, i.e.\ that each of the
  following six elementary basic translations belongs to $\FA^{(2)}$:
 \[
    \begin{tabular}{ c|c c }
      1 & g(001) & g(011)\\ 
      0 & g(000) & g(010) \\ 
      \hline
      $g(0xy)$  & 0 & 1\\

 \end{tabular}
 \hspace{.3cm}, \dots, \hspace{.3cm}
 \begin{tabular}{ c|c c }
   1 & g(011) & g(111)\\  
   0 & g(001) & g(101) \\ 
   \hline
   $g(xy1)$ & 0 & 1\\

 \end{tabular}.
\]
%
We want to show $\bDelta(g)=\{ g(xxy), g(xyx), g(xyy)\} \subseteq \FA^{(2)}$. We
proceed by case analysis. It is easy to see that any fragment of a
trivial clone with constants is a generalized quasiorder, so in
particular, $\FC^{(2)}$ is a generalized quasiorder. Hence, if
$\ebtrl \{g\} \subseteq \FC^{(2)}$, then
$\{ g(xxy), g(xyx), g(xyy)\} \subseteq \FC^{(2)} \subseteq \FA^{(2)}$.

So, we may assume that one of the elementary basic translations of $g$
is $\wedge$. Since here $\FA^{(2)}$ is a clone fragment and closed under
permutation of variables, $\ostar \FA^{(2)}$ is also closed under
permutation of variables. So, we may assume without loss of generality
that we are in one of two cases.


\begin{enumerate}[label=\textup{(\roman*)}]
\item $g(x0z) = x\wedge z$, i.e.\ that
  \[
    g= \Cube[0][0][0][1][g(010)][g(011)][g(110)][g(111)][1.5] .
  \]
  Observe that $g(110)$ is now forced to equal $0$, since otherwise
  $g(1yz) \notin \FA^{(2)}$. Propagation of this kind of reasoning shows
  that in this case
  \[
    g= \Cube[0][0][0][1][0][0][0][1][1.5],
  \]
  and we obtain that
  \[
    \{ g(xxy), g(xyx), g(xyy)\} = \left\{
      \begin{tabular}{ c|c c }
        1 & 0 & 1\\ 
        0 & 0 & 1 \\ 
        \hline
        $ e_0^2$ & 0 & 1\\

      \end{tabular} \hspace{.2cm},\hspace{.2cm}
      \begin{tabular}{ c|c c }
        1 & 0 & 1\\ 
        0 & 0 & 0 \\ 
        \hline
        $\wedge$ & 0 & 1\\

\end{tabular} 
\right\} \subseteq \FA^{(2)}.
\]

\item $g(x1z) = x \wedge z$, i.e.\ that
  \[
    g= \Cube[g(000)][g(010)][g(100)][g(110)][0][0][0][1][1.5] .
  \]
  In this case the following are the only two ways to complete $g$:
  \[
    g= \Cube[0][0][0][0][0][0][0][1][1.5], \text{ or } g=
    \Cube[0][0][0][1][0][0][0][1][1.5],
  \]
  and it is similarly straightforward that
  $\{ g(xxy), g(xyx), g(xyy)\} \subseteq \FA^{(2)}$.
\end{enumerate}

\end{example}

\subsection*{(B) Examples on a three element domain}

While on a Boolean domain every clone with constants has a
$\GQuord$-dimension of either $1$ or $2$, in general any positive
natural number can be the $\GQuord$-dimension of some clone on a
finite set. Unsurprisingly, this is already possible for a domain with
three elements, as we show in the next example.
\begin{example}\label{ex:3GQuord}
  We examine an infinite sequence of clones over $A = \{0,1,2\}$ for
  which the $\GQuord$-dimension of each clone is one greater than its
  predecessor. For $ d \geq 1$, we define a $d$-ary operation
  $f_d: \{0, 1, 2\}^d \to \{0, 1, 2\} $ by
  \[
    f_d(x_0, \dots, x_{d-1}) =
    \begin{cases}
      1 &\text{ if } x_0 = \dots = x_{d-1} = 2 \\
      0 &\text{else}.
    \end{cases}
  \]

  Then we set $F_d$ to be the clone generated by
  $\{ f_d\} \cup \cC$ and $F = \bigcup_{1 \leq d } F_d$. We
  record some easy facts about these clones.

  \begin{enumerate}[label=\textup{(\roman*)}]

  \item $f_{d} \in F_{d+1}$, since
    $f_{d}(x_0, \dots, x_{d-1}) = f_{d+1}(x_0, \dots, x_{d-1},2)$. It
    follows that
    \[ F_1 \subseteq F_2 \subseteq \dots \subseteq F_d \subseteq \dots
      \subseteq F
    \]
    is a chain whose union is $F$.
  \item \emph{Every operation belonging to $F_d$ is either a projection, a
    constant operation, or obtained from one of the $f_{d'}$
    ($1\leq d' \leq d$) by addition of some fictitious
    variables.} Indeed, it is easy to check that this collection is
    closed under composition and hence must equal $F_d$. It follows
    that all functions belonging to $F_d$ have essential arity bounded
    by $d$, so proper containment $F_d \subsetneq F_{d+1}$ in the
    above chain holds for all $d\geq 1$ (since $f_{d+1}$ depends on all of
    its $(d+1)$-many variables). Morevover, it follows that for all
    $n,d,d'$ satisfying $1\leq n \leq d$ and $n \leq d'$, we
    have $F^{(n)}_d = F^{(n)}_{d'}$.

  \item \emph{For every $d \geq 1$, the clone fragment $F^{(d)}_d$ is a
    generalized quasiorder.} The proof of this proceeds inductively. It
    is easy to check that $F_1^{(1)}$ is a generalized quasiorder and,
    using the characterization of the functions belonging to $F_d$
    obtained earlier, it is straightforward to show that
    $F^{(d+1)}_{d+1} = \ostar F^{(d)}_d$. We apply
    Lemma~\ref{lm:adjointproperties} and the inductive hypothesis that
    $F^{(d)}_d$ is a generalized quasiorder to conclude that
    $F^{(d+1)}_{d+1}$ is also generalized quasiorder.


  \item \emph{For all $d \geq 1$, we have $F = (F_d^{(d)})^{*_d}$.} Indeed,
    by (iii) above and Proposition~\ref{prop:StarisCloneIffGQuord}, we know
    that $(F_d^{(d)})^{*_d}$ is a clone whose $d'$-ary fragment is
    equal to $F_{d'}^{(d')}$ for all $d' \geq d$. Hence,
    $(F_d^{(d)})^{*_d}$ is a clone which contains each $F_{d'}$ and no
    more, which means it is equal to $F$.

  \item \emph{For all $d \geq 1$ and $d' > d$, we have
    $F_d = (F_d^{(d')})^{*_{d'}}$.} To see this, it is enough to
    establish that $F_{d}^{(d'+1)} = \ostar F_{d}^{(d')}$ for all
    $d' >d$. We first observe that $\ostar F_{d}^{(d')} \subseteq F$,
    since $F_d^{(d')} \subseteq F_{d'}^{(d')}$ and
    $F_{d'+1}^{(d'+1)} = \ostar F_{d'}^{(d')}$. The question is, can
    there be a function $g \in \ostar F_{d}^{(d')} $ which depends on
    more than $d$-many of its variables? Suppose there is such a
    $g(x_0, \dots, x_{d'})$. If $g$ depends on all of its variables,
    then by (ii) we know that
    $g(x_0, \dots, x_{d'}) = f_{d'+1}(x_0, \dots, x_{d'})$. This is
    impossible, since then
    $f_{d'}(x_0, \dots, x_{d'-1})= g(x_0, \dots, x_{d'-1},2) \in
    F_d^{(d')}$, which contradicts the fact that $F_d$ has essential
    arity bounded by $d$. If $g$ fails to depend on one of its
    variables, say $x_0$, then
    $g(x_0, \dots, x_{d'-1},2) \in F_d^{(d')}$, again contradicting
    the essential arity bound on $F_d$. Hence,
    $\ostar F_{d}^{(d')} \subseteq F_{d}^{(d'+1)}$. The other
    containment always holds for a clone with constants, so we are
    done.
 
  \end{enumerate}

  Hence, the $\GQuord$-dimension of $F$ is equal to $1$ (by (iv)
  above), while the $\GQuord$-dimension of each $F_d$ is equal to
  $d+1$ (by (v) above).

\end{example}

We conclude the section by showing that a clone with constants need
not be determined by its compatible generalized quasiorders (we thank
Gerg{\H o} Gyenizse for showing us this example).

\begin{example}\label{ex:3noGQuord}
  Let $A = \{0,1,2\}$ be a $3$-element set and let
  $\rho \subseteq A^3$ be the `not-all-not-equal' relation, i.e.\
  \[
    \theta := A^3 \setminus \{ (\pi 0,\pi 1,\pi 2) \mid \pi
    \in S_3 \}.
  \]
  Then set $F = \Pol\theta$. It is well-known that $F$ is a maximal
  clone~\cite{Ros1970}. We claim that $F$ has no nontrivial
  generalized quasiorders.

  Suppose towards a contradiction that there exists a nontrivial
  generalized quasiorder $\rho \in \Inv(F)$. Since $F$ is maximal, it
  must be that $F = \Pol\rho$ and so $F$ has finite generalized
  quasiorder dimension. Hence, there exists $d \geq 1$ such that
  $F = (F^{(d)})^*$. We will define a function which belongs to
  $\ostar F^{(d)}$, but does not belong to $F$, which is a
  contradiction.

  Indeed, define $f_{d+1} \in \Op^{(d+1)}(A)$ by
\[
f_{d+1}(x_0, \dots, x_{d}) = 
\begin{cases}
0 &\text{ if } x_0 = \dots = x_{d} = 0 \\
1 &\text{ if } x_0 = \dots = x_{d} = 1 \\
2 &\text{else}.
\end{cases}
\]
First, we notice that $f_{d+1} \notin F$, since it fails to preserve
$\theta$. To see this, apply it coordinatewise to the tuples
$ (1,2,1), (1,2,2), \dots (1,2,2) $ to obtain the tuple
\[
  (f_{d+1}(1,1,\dots, 1), f_{d+1}(2,2, \dots, 2), f_{d+1}(1, 2, \dots,
  2)) = (0,1,2).
\]
On the other hand, $\ebtrl(f) \subseteq F^{(d)}$, since fixing any
variable of $f$ to a constant value produces a function with at most
two output values and any such function preserves $\theta$. Thus
$f\in\ostar F^{(d)}$ by \ref{def:ostar}(2).

\end{example}

\section{Simple versus compound arity}\label{sec:simvscom}

In this section we compare the definitions of transitivity and
reflexivity for simple arity $\alpha = B^{d_1 + \dots + d_s}$ and
compound arity $\beta= B^{d_1} \times \dots \times B^{d_s}$. There is
an obvious bijection $\mu$ between $\alpha$ and $\beta$, since their
elements are, respectively, of the form
\[
  ( a_0, \dots, a_{d_1 + \dots + d_s-1}) \qquad \text{ and } \qquad
  \left( (b_0^1, \dots, b^1_{d_1-1}), \dots, (b^s_0, \dots,
    b^s_{d_s-1}) \right).
\]
Hence, we can define $\mu: \beta \to \alpha$ to be the function which
removes the inner parenthesis from $\beta$ tuples, i.e.\
\[
  \mu\left( (b_0^1, \dots, b^1_{d_1-1}), \dots, (b^s_0, \dots,
    b^s_{d_s-1}) \right) := ( b_0^1, \dots, b^1_{d_1-1}, \dots, b^s_0,
  \dots, b^s_{d_s-1} ).
\]

Hence, we see that $\alpha$ and $\beta$ have the same `rectangular'
geometry. The reason to distinguish $\alpha$ from $\beta$ is that the
definitions of transitivity and reflexivity for 
relations of arity $\alpha$ differ from those of arity $\beta$. A priori, the
distinction between compound and simple arity in this setting is only
formal. Hence, it is natural to wonder which of the following
implications hold (notation cf.~\ref{def:aritymap}):
\begin{align*}
  \rho \in \GQuord (A)^{(\alpha)} &\xRightarrow{ \text{ ? } } \rho \circ \mu \in \GQuord (A)^{(\beta)}\\
  \rho \in \GQuord (A)^{(\beta)} &\xRightarrow{ \text{ ? } } \rho \circ \mu^{-1} \in \GQuord (A)^{(\beta)}.
\end{align*}

The reader can check that all of the $2$-dimensional clone fragments
of simple arity $2^2$ in Example~\ref{ex:Booleancase} are also 
$2$-dimensional generalized quasiorders of compound arity 
$2^1 \times 2^1$  (i.e. $B_1 = 2$, $d_1= 1$, $B_2= 2$, $d_2 = 1$ in the 
notation of Definition~\ref{def:Rrelation}), so both of the
implications above hold for Boolean binary clone fragments. However,
neither implication holds in general, as we shall demonstrate in
Examples~\ref{ex:compoundnotsimple} and~\ref{ex:simplenotcompound} (in
the examples we allow ourselves the slight abuse of notation where we
do not distinguish between $\rho$ and $\rho \circ \mu$).

\begin{example}\label{ex:compoundnotsimple}
  In this example we show that there exists a
  $d$-dimensional generalized quasiorder of compound arity which is not a
  $d$-dimensional generalized quasiorder of simple arity (in the sense described
  above). Of course, there exist $d$-dimensional
  generalized quasiorders of compound arity whose arity prevents the
  definition of
  $d$-dimensional generalized quasiorder of simple arity from applying
  (e.g.\ if $\alpha = 3^1 \times 2^1$), so we restrict our attention
  to R-relations which have arity $B^d$ and $B^{d_1 + \dots + d_s}$,
  where $d = d_1 + \dots + d_s$. In fact, we refrain from providing
  examples for all possibilities and consider for this example a
  trivial situation where the relational arity is $4$.

  So, let
  \begin{align*}
    \rho &:= 
           \left\{ g(x,y) = 
           \SquareXY[g(0,0)][g(0,1)][g(1,0)][g(1,1)][1.5] \in 2^{2^1 \times 2^1} \;\middle|\; g(0,0) = g(1,0) , g(0,1)= g(1,1) 
           \right\}\\
         &= 
           \left\{
           \Square[0][0][0][0], \Square [1][1][1][1], \Square[0][1][0][1], \Square[1][0][1][0] 
           \right\}
  \end{align*}
  be an R-relation of rectangular arity
  $\alpha = \alpha_1 \times \alpha_2$ with
  $\alpha_1 = \alpha_2 = 2^1$, where the matrices are oriented so that
  $\alpha_1$ is the horizontal direction and $\alpha_2$ is the
  vertical direction. We claim that $\rho$ is a $2$-dimensional
  generalized quasiorder of compound arity. Indeed, $\rho$ contains both
  constant functions and it is easy to see that both
  $\rho^{(\alpha_1)}$ and $\rho^{(\alpha_2)}$ are internally reflexive
  and transitive (as $R$-relations over $2^{2^1}$).

  On the other hand, $\rho$ is \emph{not} internally reflexive if we
  consider it as a relation of arity $2^2$. To see this, consider the
  third function listed in the definition above, which we call
  $g_3$ (explicitly we have $g_{3}(x,y)=y$). Define $h(x,y) = g_3(0,x)$. Observe that
  \[
    h(x,y) = \Square[0][0][1][1]\in\trl_{2}(g_{3})
  \]
  would belong to $\rho$ if it were internally reflexive, a contradiction. So we
  conclude that $\rho$ is not internally reflexive.
\end{example}

\begin{example}\label{ex:simplenotcompound}
  In this example we give an example of a clone fragment which is a
  generalized quasiorder of simple arity, but fails to be a
  generalized quasiorder of compound arity. We first analyze the
  maximal clone of 
  $\mathbb{Z}_3$-quasilinear operations and then point out that all
  maximal clones of quasilinear functions have behave similarly.

  So, let $A = \{0,1,2\}$ be the underlying set of $\mathbb{Z}_3$ and
  let $m(x,y,z) = x-y+z$ be the standard Mal'cev operation for
  $\mathbb{Z}_3$. We then define
  \[
    \rho := \left\{ \Square[a][b][c][d] \in A^{2^1 \times 2^1} \;\middle|\;
      d = c -a +b \right\}
  \] and set $F := \Pol\rho$, which is the set of all quasilinear
  functions over $A$ (with respect to $\Z_{3}$).


We record
  the following basic facts about $F$.

  \begin{enumerate}[label=\textup{(\roman*)}]

  \item It is easy to check that $\rho$ is a
    $2$-dimensional generalized quasiorder of compound arity (those familiar with the
    Universal Algebra commutator will recognize $\rho$ as the algebra
    of $(1,1)$-matrices for $\mathbb{Z}_3$). It follows (cf.\ Proposition~\ref{prop:StarisCloneIffGQuord}) that $F^{(2)}$
    is a $2$-dimensional generalized quasiorder (of simple arity~$3^2$).

  \item The clone fragment $F^{(1)}=\{g(x)=a+bx\mid a,b\in\Z_{3}\}$
    (consisting of all permutations and constants) is not a
    generalized quasiorder 
    (here there is only one candidate for the arity, which is
    $3^1$). Indeed, this follows from the fact that
    \[
      f:= \ThreeSquareXY[2][2][2][1][2][0][1][2][0] \in \ostar
      F^{(1)} \text{ (all rows and colums belong to $F^{(1)}$),}
    \]
    but the diagonal $f(x,x)$ does not belong to $F^{(1)}$, which shows
    that $F^{(1)}$ is not transitive
   (cf.\ Example~\ref{ex:reftra} and Figure~\ref{fig:tra1}).


  \item
The clone fragment $F^{(2)}$ is \emph{not} a generalized
    quasiorder of compound arity
    $\alpha_1 \times \alpha_2 = 3^{1}\times 3^1$. 
To see this, it is
    enough to show that $\left(F^{(2)}\right)^{(\alpha_1)}$ is not
    transitive. To that end, consider the cube $h(x,y,z):=f(x,y)$ in
    Figure~\ref{fig:example} which can be seen as a mapping
$\overline{f} \in \left( A^{3^1}\right)^{3^2}$ defined by $\overline{f}(x,y) := c_{f(x,y)}$ where $c_{f(x,y)} \in A^{3^1}$ is the unary mapping with constant
    value $f(x,y)$ (the same $f$ given in (ii)).


    \begin{figure}    
\begin{center}
     \includegraphics{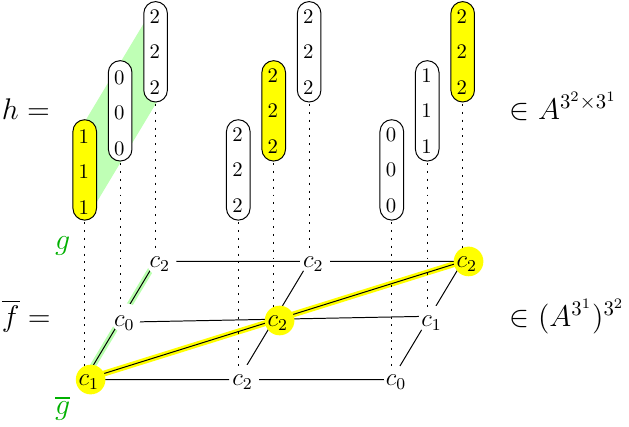}
    \end{center}
  \caption{The function $\overline f$\label{fig:example} in \ref{ex:simplenotcompound}(iii)}
    \end{figure}
    It is straightforward to see that
    $\overline{f} \in \ostar \left( F^{(2)}\right)^{(\alpha_1)}$. 
For example, the elementary translation
    $\overline{g}(y):=\overline{f}(0,y) \in (A^{3^1})^{3^1}$ is the function
    given by the triple $(c_1, c_0, c_2)$ (marked green in
    Figure~\ref{fig:example}), which belongs to 
    $(F^{(2)})^{(\alpha_1)}$, since $\overline{g}=g^{(\alpha_{1})}$
    for the binary operation $g(y,z):=h(0,y,z)=f(0,y)$ (the green marked
    left vertical face in the cube $h$), which obviously belongs to
    $F^{(2)}$ (since it is a permutation $(1,0,2)$ with a fictitious
    second variable).



    However, the diagonal $\overline{f}(x,x)$ does not belong to
    $\left(F^{(2)}\right)^{(\alpha_1)}$, since it
    corresponds to the operation
    $g':=$\scalebox{0.5}{$\ThreeSquareXY[1][2][2][2][2][1][1][2][2]$} which
   evidently does not belong to $F^{(2)}$ since $g'(x,0)=(1,2,2)$ is neither a
   permutation nor a constant and therefore does not preserve $\rho$.

    Hence $\bDelta(\ostar(F^{(2)})^{(\alpha_{1})})\not\subseteq
    (F^{(2)})^{(\alpha_{1})}$, i.e.,
    $\left(F^{(2)}\right)^{(\alpha_1)}$ is not transitive 
     (cf.~\ref{Adef:Refl+Trans}\ref{Adef:transitive}), as
    claimed, so $F^{(2)}$ is not a generalized quasiorder of compound
    arity (cf.\ \ref{Adef2:Refl+Trans}\ref{Adef2:transitive}
    and Figure~\ref{fig:compoundGQuord}).

\end{enumerate}

  Actually, the above observations hold for any maximal clone of
  quasilinear functions. In general, let $(A; +)$ be an elementary
  abelian $p$-group, let $m(x,y,z) = x-y+z$ be the standard Mal'cev
  operation associated to $(A,+)$, and set
  \[
    \rho := \left\{ \SquareXY[b][c][a][m(a-b+c)][2]\;\middle|\; a,b,c \in A
    \right\} \subseteq A^{2^1 \times 2^1}.
  \]

  Now take $F := \Pol\rho$. It is well known that $F$ is
  a maximal clone on $A$ (one of Rosenberg's five types~\cite{Ros1970}) and that $F$
  is equal to the collection of all $\mathbb{Z}_p$-quasilinear
  operations on $A$.

  Like in the special case above, it is easy to check that $\rho$ is a
  $2$-dimensional generalized quasiorder  of compound
  arity (with arity
  $\alpha_1 \times \alpha_2 := 2^1 \times 2^1$), so it follows that
  $F$ has $\GQuord$-dimension at most $2$. In (\cite{JakPR2027} ), the
  authors show that $F$ does not have any 
  nontrivial $1$-dimensional generalized quasiorder, so actually the
  $\GQuord$-dimension of $F$ is always $2$ (since otherwise the
  $\GQuord$-dimension is $1$, so $F= (F^{(1)})^{*_1}$ and $F^{(1)}$
  would be
  a nontrivial $1$-dimensional $\GQuord$).

  Suppose now that $F^{(1)}$ is not equal to $\Op(A)^{(1)}$ (so in
  particular, it is not a generalized quasiorder). Notice that this
  assumption only rules out the maximal boolean clone of quasilinear
  functions, since in all other cases not every unary operation is
  quasilinear. Since $F^{(1)}$ is reflexive, there must exist some
  function $f \in \ostar F^{(1)}$ such that $f(x,x) \notin F$. The
  construction of $\overline{f}$ in the above example works in this
  situation also, which shows that such $F^{(2)}$ are
  $2$-dimensional generalized quasiorders of simple arity but not of
  compound arity. 


\end{example}


\section{Concluding remarks, questions and problems}\label{sec:conclusion}

The theory of generalized quasiorders (of finite $\GQuord$-dimension)
provides a new tool to study clones of operations with constants. 
Since the notion of a $d$-dimensional
generalized quasiorder is new in the generality introduced in this
paper, the field for further research is wide and it remains to be
seen which direction is most promising. We mention some topics.

(I) \emph{The lattice $\cL_{A}$ of clones.} Generalized quasiorders can serve
for the investigation of 
the lattice $\cL_{A}$ of all clones on a (finite) set $A$. 

For
instance, for each set $Q\subseteq\GQuord(A)$ of the form $Q=\GQuord
F_{0}$ (for some clone $F_{0}$), the set $\{F\in\cL_{A}\mid \GQuord
F\subseteq Q\}$ is an order-filter in the clone lattice. What is the
structure of these filters, e.g., what are the minimal
elements for particular $F_{0}$ and which maximal
clones belong to them? Similar to the results in \cite{JakPR2027},
probably only few maximal clones have non-trivial generalized quasiorders at
all.

Further, the following are each easily seen to determine a
sub-meet-semilattice of the lattice $\cL_{A}$:

\begin{itemize}
\item[(A)] The clones $F\in\cL_{A}$ with generalized quasiorder
  dimension at most $d$, for some finite $d$.
\item[(B)] The clones $F\in\cL_{A}$ with finite generalized
  quasiorder dimension.
\item[(C)] The clones $F\in\cL_{A}$ which are determined by
  their generalized quasiorders, i.e., satisfying $F=\Pol\GQuord F$.
\end{itemize}

Actually, the clones listed in (A) and (C) above determine
\emph{complete} sub-meet-semilattices of the clone lattice over
$A$. We don't know if the clones of kind (B) above are the same as the
clones of kind (C). That is, given a clone $F$ on a finite set $A$
which is determined by its invariant generalized quasiorders, does it
always follow that there exists a $d$ so that $(F^{(d)})^{*_d} = F$?
This holds if and only if every descending chain of finite-dimensional
clones is finite, equivalently, if the clones in (B) also determine a
complete sub-meet-semilattice of the clone lattice over $A$.

We remark that it is straightforward to construct an infinite
descending chain of finite dimensional clones over an infinite domain
$A$: relabel the set of outputs for each of the R-relations $F_d$
given in Example~\ref{ex:3GQuord} so that outputs of functions in
$F_d$ are disjoint from the other $F_{d'}$ outputs, let $A$ be the
union of all such relabeled outputs, and extend each of the relabeled
$F_d$ by all constant functions over $A$ to obtain a set of relations
$S = \{ F_d' \mid 1 \leq d \}$. It is easy to see that $\Pol S$ is
not a finite-dimensional clone, so it is not listed in (B) above, but
it is obviously listed in (C). The infinite descending chain that
terminates in $F$ is given by polymorphism clones of initial segments
of the sequence of relations in $S$.

Also, can anything be said about joins in the clone lattice over $A$
for each of the classes (A), (B), and (C) above?

(II) \emph{Properties of varieties.} It is known that a variety
$\mathcal{V}$ of algebras is congruence meet semidistributive if and
only if all $2$-dimensional congruence intervals are trivial
\cite{KeaSW2022} \cite{Moo2025}. Are there any other interesting
characterizations of Mal'cev conditions in terms of the behavior of
the higher dimensional generalized quasiorders across a variety?

(III) \emph{Galois-closures.} In
Theorem~\ref{thm:dPolQuord} and
Corollary~\ref{A3cor}(i),(i)$'$,(i)$''$ we characterized the
Galois-closure $\Pola[d]\dGQuord M$ for $M\subseteq \Opa[d](A)$. What
can be said about the relational side of the Galois connection
$\Pola[d]-\dGQuord$, i.e., about $\dGQuord\Pola[d]Q$ for some
$Q\subseteq\dGQuord(A)$? Or, more general, is there an internal
characterization of the Galois closures with respect to the Galois
connection $\Pol-\GQuord$?

(IV) \emph{Constructions with generalized quasiorders.} 
Which relational constructions produce
generalized quasiorders if the input relations are generalized
quasiorders? In connection with the questions in (III) it is
particularly interesting to know which pp-formula constructions have
this property. Many results given in the paper \cite{JakPR2026} for
$1$-dimensional generalized quasiorders can be extended or generalized to
higher-dimensional generalized quasiorders.
Specializing to arity transformations (cf.~\ref{def:aritymap}) we may ask
for which arity transformations $\mu:\beta\to\alpha$ we have one of the following
properties for each (or for which) $\rho\in\RRela[\alpha](A)$ (cf.\
Section~\ref{sec:simvscom}): 
\begin{align*}
  \rho\in\GQuorda[\alpha](A)& \Longleftrightarrow
\rho\circ\mu\in\GQuorda[\beta](A),\\
\rho\in\GQuorda[\alpha](A)&\Longrightarrow
\rho\circ\mu\in\GQuorda[\beta](A),\\
\rho\in\GQuorda[\alpha](A)&\Longleftarrow \rho\circ\mu\in\GQuorda[\beta](A).
\end{align*}

(V) \emph{Generalized quasiorders of particular algebras or clones.} In addition
to the questions in (II), one can ask for the ($d$-dimensional)
generalized quasiorders of particular algebras $(A,F)$, e.g., in a
chosen variety or with special clone properties like being a maximal
clone (cf.\ Example~\ref{ex:3noGQuord}).





\def\cprime{$'$} \def\cprime{$'$}

Andrew Moorhead: \texttt{apmoorhead@gmail.com}\\
Reinhard Pöschel: \texttt{reinhard.poeschel@tu-dresden.de}

\end{document}